\documentclass[12pt,letterpaper]{scrartcl}

\usepackage{amsmath}
\usepackage{amsthm}
\usepackage{amssymb}
\usepackage{stmaryrd}
\usepackage{accents}
\usepackage{bbm}
\usepackage[mathscr]{euscript}
\usepackage{xcolor}
\definecolor{Myblue}{rgb}{0,0,0.6}  
\usepackage[colorlinks,citecolor=Myblue,linkcolor=Myblue,urlcolor=Myblue,pdfpagemode=None]{hyperref}
\usepackage{subcaption}
\usepackage{graphicx}
\usepackage{epstopdf}
\usepackage[totalwidth=420pt, totalheight=590pt]{geometry}
\usepackage[inline,shortlabels]{enumitem}
\usepackage[nottoc]{tocbibind}
\usepackage{hyperref}
\usepackage{fancyhdr}
\usepackage{xcolor}
\usepackage{tikz}
\usetikzlibrary{decorations.markings}
\usetikzlibrary{knots}
\usetikzlibrary{arrows}
\usetikzlibrary{cd}

\allowdisplaybreaks

\theoremstyle{definition}
\newtheorem{defn}{Definition}
\newtheorem{thm}[defn]{Theorem}
\newtheorem{prp}[defn]{Proposition}
\newtheorem{lem}[defn]{Lemma}
\newtheorem{cor}[defn]{Corollary}
\newtheorem{defnprp}[defn]{Definition and Proposition}

\newtheorem{rem}[defn]{Remark}
\newtheorem{example}[defn]{Example}
\newtheorem{conv}[defn]{Convention}

\numberwithin{equation}{section}
\numberwithin{defn}{section}
\numberwithin{figure}{section}

\newcommand{\pic}[2][0.75]{
	\begin{tikzpicture}[scale=0.5,baseline={([yshift=-.5ex]current bounding box.center)}]
	\node at (0,0) {\includegraphics[scale=#1]{figures/#2}};
	\end{tikzpicture}
}

\begin{document}
\def\it{\textit}
\def\mcA{\mathcal{A}}
\def\mcB{\mathcal{B}}
\def\mcC{\mathcal{C}}
\def\mcD{\mathcal{D}}
\def\mcE{\mathcal{E}}
\def\mcF{\mathcal{F}}
\def\mcI{\mathcal{I}}
\def\mcM{\mathcal{M}}
\def\mcN{\mathcal{N}}
\def\mcS{\mathcal{S}}
\def\mcZ{\mathcal{Z}}
\def\mcACA{{_A \mcC_A}}
\def\mcACB{{_A \mcC_B}}
\def\mcBCA{{_B \mcC_A}}
\def\mcBCB{{_B \mcC_B}}
\def\mcD{\mathcal{D}}
\def\euT{\mathscr{T}}
\def\opk{\Bbbk}
\def\opA{\mathbb{A}}
\def\opB{\mathbb{B}}
\def\opC{\mathbb{C}}
\def\opF{\mathbb{F}}
\def\opR{\mathbb{R}}
\def\opZ{\mathbb{Z}}
\def\opid{\mathbbm{1}}
\def\a{\alpha}
\def\b{\beta}
\def\g{\gamma}
\def\d{\delta}
\def\D{\Delta}
\def\vareps{\varepsilon}
\def\l{\lambda}
\def\abar{\overline{\a}}
\def\bbar{\overline{\b}}
\def\gbar{\overline{\g}}
\def\ddbar{\overline{\d}}
\def\mubar{\overline{\mu}}
\def\pibar{{\bar{\pi}}}
\def\taub{\bar{\tau}}

\def\opp{{\operatorname{op}}}
\def\id{\operatorname{id}}
\def\im{\operatorname{im}}
\def\Hom{\operatorname{Hom}}
\def\End{\operatorname{End}}
\def\tr{\operatorname{tr}}
\def\ev{\operatorname{ev}}
\def\coev{\operatorname{coev}}
\def\evt{\widetilde{\operatorname{ev}}}
\def\coevt{\widetilde{\operatorname{coev}}}
\def\Id{\operatorname{Id}}
\def\Vect{\operatorname{Vect}}
\def\loc{{\circ}}
\def\orb{{\operatorname{orb}}}
\def\coker{{\operatorname{coker}}}
\def\Ind{{\operatorname{Ind}}}
\def\Irr{{\operatorname{Irr}}}
\def\Dim{\operatorname{Dim}}
\def\Indz{\operatorname{Ind}^{\circ}}
\def\Iz{I^{\circ}}
\def\Izbar{\overline{I^{\circ}}}

\newcommand{\taubar}[1]{\overline{\tau_{#1}}}

\def\Dhalf{D^{\scalebox{.5}{1/2}}}
\def\dhalf{d^{\scalebox{.5}{1/2}}}
\newcommand{\dhalfob}[1]{d^{\scalebox{.5}{1/2}}_{#1}}
\newcommand{\dquartob}[1]{d^{\scalebox{.5}{1/4}}_{#1}}

\def\Lra{\Leftrightarrow}
\def\Ra{\Rightarrow}
\def\ra{\rightarrow}
\def\lra{\leftrightarrows}
\def\xra{\xrightarrow}

\newcommand{\eqrefO}[1]{\hyperref[eq:O#1]{\text{(O#1)}}}
\newcommand{\eqrefT}[1]{\hyperref[eq:T#1]{\text{(T#1)}}}

\title{Condensation inversion and Witt equivalence via generalised orbifolds}

\author{
Vincentas Mulevi\v{c}ius\\[0.5cm]
\normalsize{\texttt{\href{mailto:mulevicius@mpim-bonn.mpg.de}{mulevicius@mpim-bonn.mpg.de}}}\\[0.1cm]
{\normalsize\slshape Max-Planck-Institut f\"ur Mathematik, Bonn, Germany}\\[-0.1cm]
}

\date{}
\maketitle

\begin{abstract}
In~\cite{MR1} it was shown how a so-called orbifold datum $\opA$ in a given modular fusion category (MFC) $\mcC$ produces a new MFC $\mcC_\opA$.
Examples of these associated MFCs include condensations, i.e.\ the categories $\mcC_B^\loc$ of local modules of a separable commutative algebra $B\in\mcC$.
In this paper we prove that the relation $\mcC \sim \mcC_\opA$ on MFCs is the same as Witt equivalence.
This is achieved in part by providing one with an explicit construction for inverting condensations, i.e.\ finding an orbifold datum $\opA$ in $\mcC_B^\loc$ whose associated MFC is equivalent to $\mcC$.
As a tool used in this construction we also explore what kinds of functors $F\colon\mcC\ra\mcD$ between MFCs preserve orbifold data.
It turns out that $F$ need not necessarily be strong monoidal, but rather a `ribbon Frobenius' functor, which has weak monoidal and weak comonoidal structures, related by a Frobenius-like property.
\end{abstract}
\newpage

\setcounter{tocdepth}{2}

\tableofcontents
\newpage

\section{Introduction}
\label{sec:introduction}
The notion of a \textit{modular fusion category (MFC)} was introduced by Turaev~\cite{Tu1} in order to generalise the 3-manifold invariants~\cite{RT}, obtained from the categories of representations of modular Hopf algebras, in particular quantum groups at roots of unity.
This resulted in defining a MFC to be a ribbon fusion category $\mcC$ with a non-degenerate braiding.
The 3-manifold invariants were later collected into what is now called the Reshetikhin--Turaev construction of a topological quantum field theory (TQFT) (see e.g.\ \cite[Ch.\,IV]{Tu2}), which is a symmetric monoidal functor
\begin{equation}
\label{eq:ZRT}
Z^{\operatorname{RT}}_\mcC\colon
\widehat{\operatorname{Bord}}{}^{\operatorname{rib}}_3(\mcC) \ra \Vect_\opk \, ,
\end{equation}
where the source category is a central extension of the category of 3-dimensional bordisms with embedded ribbon graphs whose components are labelled with objects/morphisms of the input MFC $\mcC$.

\medskip

MFCs also appear in the study of conformal field theories (where categories of representations of rational vertex operator algebras provide examples of them) and 2-dimensional topological phases of matter (where the datum of a MFC describes point excitations (or anyons) of the system, their fusion and spin statistics).
Altogether, they constitute interesting mathematical objects, whose classification is an important problem, see e.g.~\cite{BNRW1,BNRW2,Cr,EG,GM,Gr,HRW,JMS,RSW}.
Also important are the constructions allowing one to obtain new MFCs out of a given one.
They can be seen e.g.\ as a way to `engineer' new topological orders from a one that can already be implemented in a lab.
A myriad of such constructions arise as instances of Hopf monads~\cite{CZW}.
In this paper we will look into another construction developed in~\cite{CRS3,MR1,MR2,CMRSS2}, called a \textit{generalised orbifold}.
The definition of a generalised orbifold is best motivated by the study of defect TQFTs~\cite{DKR,CR,Ca,CRS1}, which we now recall.

\subsubsection*{Defect TQFTs and generalised orbifolds}
A \textit{defect TQFT} is a symmetric monoidal functor $Z\colon\operatorname{Bord}^{\operatorname{def}}_n(\mathbb{D})\ra\Vect_\opk$, where the source category is that of stratified $n$-dimensional bordisms, whose strata carry labels from a set $\mathbb{D}$.
In~\cite{CRS2} the Reshetikhin--Turaev TQFT~\eqref{eq:ZRT} was extended to an example of a defect TQFT in dimension 3, which produced the symmetric monoidal functor
\begin{equation}
\label{eq:Zdef}
Z^{\operatorname{def}}_\mcC\colon
\widehat{\operatorname{Bord}}{}^{\operatorname{def}}_3(\mathbb{D}^\mcC) \ra \Vect_\opk \, ,
\end{equation}
The penultimate step in defining this defect TQFT is to implement surface defects, which can be done by a 2-dimensional state-sum construction internal to the Reshetikhin--Turaev TQFT~\cite{KSa}.
It has the following intuitive description due to~\cite{FSV}: If a surface defect is such that punching a contractible hole in it does not change the invariant, one can replace the defect with a ribbon graph, obtained by punching enough holes and taking the deformation retract, at which point the resulting bordism with an embedded ribbon graph can be readily evaluated with $Z^{\operatorname{RT}}_\mcC$.
The construction is choice independent if the lines and the points in the resulting graph are labelled by a symmetric Frobenius algebra object $A$ in $\mcC$ and its (co)multiplication morphisms.
This procedure is illustrated in Figure~\ref{fig:ribbonisation}.
In it, the $\psi$-labelled points contain information about the boundary condition on the rim of the hole encoded by a morphism $\psi\colon\opid\ra A$ in $\mcC$ (with $\psi_l$, $\psi_r$ defined as in~\eqref{eq:psi_lr_defs} below), such that the Frobenius algebra $A$ satisfies
\begin{equation}
\label{eq:intro_separability_cond}
\pic[1.25]{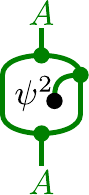} =
\pic[1.25]{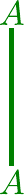} \, ,
\end{equation}
in which case it is called \textit{separable}, see Section~\ref{subsec:separability}.
\begin{rem}
\label{rem:psi_insertions}
The separability condition~\eqref{eq:intro_separability_cond} was already used in my PhD thesis~\cite{Mul}, where some of the results of~\cite{CRS2,CRS3,MR1,CMRSS2} were also adapted to this setting.
It generalises the use of the `window element' in the state-sum construction of 2-dimensional TQFTs as described in~\cite{LP}.
\end{rem}
\noindent
One can also include line defects, labelled by modules of symmetric separable Frobenius algebras, and point defects, labelled by module morphisms.
These data then constitute the label set $\mathbb{D}^\mcC$ in~\eqref{eq:Zdef}.
\begin{figure}
%\captionsetup{format=plain, indention=0.5cm}
\centering
\pic[1.25]{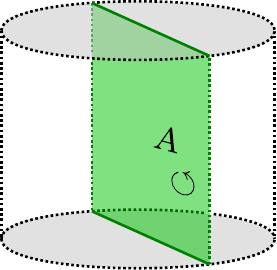} $\rightsquigarrow$
\pic[1.25]{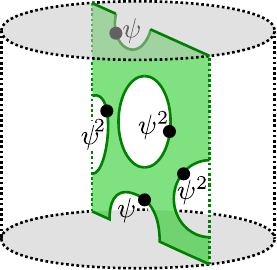} $\rightsquigarrow$
\pic[1.25]{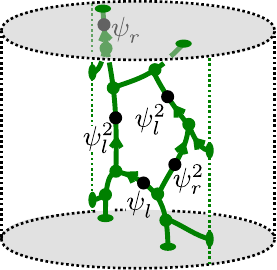}
\caption{
Punching holes in a surface defect and taking the retract yields the internal state sum construction.
}
\label{fig:ribbonisation}
\end{figure}

\medskip

The idea behind generalised orbifolds is to similarly perform a $n$-dimensional state-sum construction internal to a given $n$-dimensional defect TQFT.
In the case of $Z^{\operatorname{def}}_\mcC$ this means starting with an unstratified 3-bordism and then stratifying it by adding 3-dimensional solid balls, whose interior is labelled with the bulk theory of $Z^{\operatorname{RT}}_\mcC$ of the defect TQFT $Z^{\operatorname{def}}_\mcC$ (this should be compared to punching 3-dimensional holds in the bulk).
Taking the retract of the remaining bulk results in a network of defects, which looks like a 3-dimensional `sponge' or `foam' and can be evaluated with $Z^{\operatorname{def}}_\mcC$.
To make this construction independent of the stratification constituting the foam, one labels its components with the entries of a so-called \textit{orbifold datum}~\cite{CRS1,CRS3}.
This is a tuple $\opA=(A,T,\a,\abar,\psi,\phi)$, where $(A,\psi)$ is a symmetric separable Frobenius algebra that labels surfaces in the foam, $T$ is an $A$-$(A\otimes A)$-bimodule which labels the lines, $\a,\abar\colon T \otimes T \ra T \otimes T$ are certain morphisms labelling the two points (two labels for two possible orientations) and $\phi$ is a normalisation constant.
The 3-dimensional analogue of the Frobenius moves in dimension 2 are the so-called bubble, lune and triangle moves (see e.g.\ \cite{CMRSS1}).
Given various orientations of the strata in these moves, they yield 8 independent conditions on $\opA$, which can be written as algebraic identities listed in equations~\eqrefO{1}--\eqrefO{8} below.
Since the construction of $Z^{\operatorname{def}}_\mcC$ only depends on the MFC $\mcC$, we can call $\opA$ an orbifold datum in $\mcC$ and treat it as a purely algebraic entity.

\medskip

In dimension 3, the generalised orbifold TQFT was further developed to itself include framed line and point defects, or equivalently, embedded ribbon graphs \cite{CMRSS1}.
The intuition for the construction is as before: punch holes in the bulk (i.e.\ away from the ribbon graph) and take the retract so that upon evaluation the graph becomes trapped (or embedded) within the foam.
The strands of such ribbon graphs carry labels for lines, that can be embedded into the 2-strata of the foam, and those for points, at which a strand intersects the 1-strata of the foam.
The labels for strands and points must be such that upon evaluation the choice of embedding of the ribbon graph into the foam does not matter.
The strand and point labels are then collected into a ribbon category, which for the case of defect TQFT $Z^{\operatorname{def}}_\mcC$ we denote by $\mcC_\opA$, where $\opA=(A,T,\a,\abar,\psi,\phi)$ is the orbifold datum in $\mcC$~\cite{MR1}.
The objects of $\mcC_\opA$ are tuples $(M,\tau_1,\tau_2,\taubar{1},\taubar{2})$, where $M$ is an $A$-$A$-bimodule (to label a line defect within an $A$-labelled surface defect in $Z^{\operatorname{def}}_\mcC$) and $\tau_i$, $\taubar{i}$, $i\in\{1,2\}$ are $A$-$A\otimes A$-bimodule morphisms (to label intersection points with 1-strata in the foam), which satisfy a set of algebraic identities listed in the equations~\eqrefT{1}--\eqrefT{7} below (encoding the conditions for independence on the embedding in the foam).
Altogether this yields a TQFT~\cite{CMRSS2}
\begin{equation}
\label{eq:Zorb}
Z^{\operatorname{orb}\opA}_\mcC\colon
\widehat{\operatorname{Bord}}{}^{\operatorname{rib}}_3(\mcC_\opA) \ra \Vect_\opk \, .
\end{equation}
Of central importance to this paper is the following result (see~\cite[Thm.\,3.17]{MR1}, restated as Theorem~\ref{thm:CA_is_MFC} below): $\mcC_\opA$ is a multifusion category, which is a MFC in case it is fusion (i.e.\ in case the monoidal unit $\opid_{\mcC_\opA}$ is a simple object).
As shown in~\cite[Thm.\,4.1]{CMRSS2}, in this case one also has the isomorphism of TQFTs
\begin{equation}
\label{eq:intro_Zorb_equiv_ZRT}
Z^{\operatorname{orb}\opA}_\mcC \cong Z^{\operatorname{RT}}_{\mcC_\opA} \, .
\end{equation}

\subsubsection*{Condensation inversion}
Although best understood in terms of TQFTs, the construction of the MFC $\mcC_\opA$ from an orbifold datum $\opA$ in a MFC $\mcC$ can be seen as a purely algebraic one.
It is possible for $\mcC_\opA$ to have higher rank than $\mcC$, which makes this construction a potentially useful tool for classifying MFCs.
In particular, in~\cite[Sec.\,4.2]{CRS3}, \cite[Sec.\,4.2]{MR1} an example of an orbifold datum in the trivial MFC $\Vect_\opk$ of finite dimensional vector spaces was constructed from an arbitrary spherical fusion category $\mcS$ with non-vanishing global dimension, whose associated MFC is the Drinfeld centre $\mcZ(\mcS)$.
Furthermore, in~\cite{MR2} examples of orbifold data in rank~3 MFCs of Ising type were found, which produce rank~11 MFCs.
Both of these examples are instances of inverting the condensation construction: given a \textit{condensable algebra} $B\in\mcC$ (by which we mean a commutative haploid (or connected) symmetric separable Frobenius algebra), its \textit{condensation} is the category $\mcC_B^\loc$ of local modules of $\mcC$, which is known to be a MFC~\cite{KO} and typically has a lower rank than $\mcC$.
Concerning the aforementioned examples, Drinfeld centres $\mcZ(\mcS)$ posses a so-called Lagrangian algebra, whose condensation is equivalent to $\Vect_\opk$ (see~\cite[Sec.\,4.2]{DMNO}), whereas one of the Ising type MFCs arises as the condensation of the $E_6$ algebra in the category $\mcC(sl(2),10)$ of integrable highest weight representations of the affine Lie algebra $\widehat{sl}(2)$ at level 10 (having rank 11).
One of the main results of this paper is proving this result in general (see Theorem~\ref{thm:cond_inv_orb_datum}):
\begin{thm}
\label{thm:intro_inv_orb_datum}
Let $B$ be a condensable algebra in a MFC $\mcC$.
Then there is a an orbifold datum $\opA$ in the condensation $\mcC_B^\loc$ such that $(\mcC_B^\loc)_\opA \simeq \mcC$.
\end{thm}

What is more, we give an explicit construction of the orbifold datum $\opA$ as in Theorem~\ref{thm:intro_inv_orb_datum}, which is again inspired by the orbifold TQFT.
To explain it we note that condensations were shown in~\cite[Sec.\,3.4]{CRS3}, \cite[Sec.\,4.1]{MR1} to also be instances of the MFCs associated to orbifold data: a condensable algebra $B\in\mcC$ yields an orbifold datum $\opB$ in $\mcC$ such that $\mcC_\opB \simeq \mcC_B^\loc$.
Together with the isomorphism of TQFTs~\eqref{eq:intro_Zorb_equiv_ZRT} this helps to think of $\opA$ as labelling a `foam within the $\opB$-foam'\footnote{This can be defined rigorously since for the case of orbifold data $\opB$ obtained from condensable algebras, the TQFT $Z^{\operatorname{orb}\opB}_\mcC$ was generalised to include defects~\cite{KMRS}.}.
Our ansatz for $\opA$ is then comprised of gap defects, at which a $\opB$-foam first terminates and then starts again.
For example the 2-strata in the $\opA$-foam will be labelled by the surface gap defects as in Figure~\ref{fig:gap_defect}, etc.
Applying the internal state sum construction using $\opA$ then has the effect of breaking the $\opB$-foam into contractible pieces, at which point it is clear that the resulting bulk theory is precisely $Z^{\operatorname{RT}}_{\mcC}$ and so the condensation has been inverted.

There is an extra nuisance when trying to express the constituents of $\opA$ as labels for defects in the defect TQFT $Z^{\operatorname{def}}_{\mcC_B^\loc}$, for example the surface defect in Figure~\ref{fig:gap_defect} should be labelled by a symmetric separable Frobenius algebra in $\mcC_B^\loc$.
This can be achieved by the `punching holes' algorithm as explained above: in this case it looks like connecting the two pieces of foam by a tunnel.
Doing so is not however an invertible procedure: intuitively this undermines the idea of breaking the $\opB$-foam into contractible pieces.
The idea can be saved if one decorates the tunnel by surrounding it with a line defect labelled with the Kirby colour, i.e.\ the object $C=\bigoplus_{i\in\Irr_\mcC} i\in\mcC$ with a point insertion $d = \bigoplus_{i\in\Irr_\mcC}(\dim_\mcC i)\cdot\id_i$, see Figure~\ref{fig:gap_tunnel}.
This is because the tunnel filled with $\opB$-foam can be retracted to a single $B$-labelled line and the extra line defect has the effect of cutting it due to what we call the `scissors identity' (see equations~\eqref{eq:scissors_id} and~\eqref{eq:scissors_on_B} below).
In the end, the symmetric separable Frobenius $A$ labelling the 2-strata in $\opA$ is obtained by localising the induced module $C^*\otimes C\otimes B$ and has the interpretation of labelling tubes within the $\opB$-foam having two $C$-labelled inner lines with opposite orientations, see Figure~\ref{fig:A_as_tubes}.
We note that the necessity to add the dimension factors $d$ when using the Kirby colour is what motives using the separability condition~\eqref{eq:intro_separability_cond}.
One can similarly find the expressions for the other entries in $\opA$, the explicit end result is listed in the equations \eqref{eq:inv-orb-dat_A}--\eqref{eq:inv-orb-dat_alphas} below.
\begin{figure}
\captionsetup{format=plain, indention=0.5cm}
%\centering
\begin{subfigure}[b]{0.32\textwidth}
	\centering
	\pic[1.25]{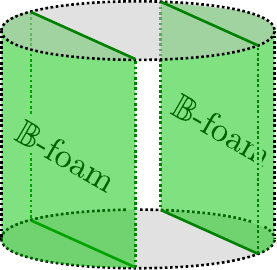}
	\caption{}
	\label{fig:gap_defect}
\end{subfigure}
\begin{subfigure}[b]{0.32\textwidth}
	\centering
	\pic[1.25]{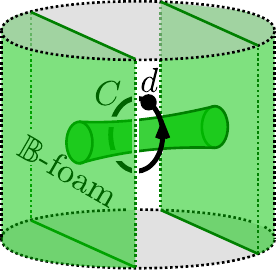}
	\caption{}
	\label{fig:gap_tunnel}
\end{subfigure}
\begin{subfigure}[b]{0.32\textwidth}
	\centering
	\pic[1.25]{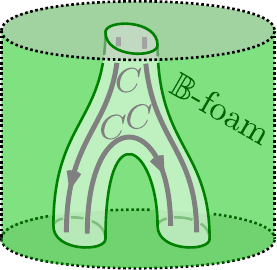}
	\caption{}
	\label{fig:A_as_tubes}
\end{subfigure}
\caption{
(a) A gap between two regions filled with $\opB$-foam is a surface defect in $Z^{\operatorname{orb}\opB}_\mcC\cong Z^{\operatorname{RT}}_{\mcC_B^\loc}$.
(b) Adding a tunnel, surrounded with a line labelled by Kirby colour is the analogue of punching a hole in the gap defect (cf.\ Figure~\ref{fig:ribbonisation}).
(c) As an object in $\mcC_B^\loc$, the Frobenius algebra corresponding to the gap defect can be interpreted as labelling hollow tubes in the $\opB$-foam with two $C$-lines inside.
}
\label{fig:gap}
\end{figure}

\subsubsection*{Ribbon Frobenius functors}
The above explanation of how the orbifold datum $\opA$ as in Theorem~\ref{thm:intro_inv_orb_datum} arises is more intuitive than precise and so in principle we should prove that it indeed satisfies the identities~\eqrefO{1}--\eqrefO{8}, as well as that the associated MFC is equivalent to $\mcC$.
This can be done by direct computations, but it seems more far reaching to ask the following more general question: having two MFCs $\mcC$ and $\mcD$ and an orbifold datum $\opA$ in $\mcC$, what kind of functors $F\colon\mcC\ra\mcD$ preserve $\opA$ (i.e.\ when can one define a natural orbifold datum $F(\opA)$ in $\mcD$) and how are the categories $\mcC_\opA$ and $\mcD_{F(\opA)}$ related?
One answer is given by ribbon functors $\mcC\ra\mcD$, but they are rather restrictive (in fact such functors are always full embeddings, see~\cite[Cor.\,3.26]{DMNO}).
A candidate for a more interesting answer is given by what we call \textit{ribbon Frobenius functors} $F\colon\mcC\ra\mcD$, which have weak monoidal and weak comonoidal structures, compatible with each other by a Frobenius-like relation and preserving braidings and twists in an appropriate sense.
They are slight generalisations of similar functors studied in~\cite{Sz1, Sz2, DP, MS} (mostly for the case of $\mcC$ and $\mcD$ being monoidal only), where they were shown to automatically preserve Frobenius algebras and, in case a certain separability condition applies, also those morphisms in $\mcC$ that can be depicted by connected string diagrams.
The morphisms appearing on both sides of~\eqrefO{1}--\eqrefO{8} happen to be given by connected string diagrams, and furthermore, the objects involved are always $A$-$A$-bimodule morphisms, so that the aforementioned separability condition can be exchanged for a weaker condition which we call compatibility with the orbifold datum $\opA$.
The promising result for finding new instances of orbifold data is then (see Proposition~\ref{defnprp:FA_orb_datum}):
\begin{thm}
\label{thm:intro_FA_orb_datum}
Let $\mcC$, $\mcD$ be MFCs, $\opA$ an orbifold datum and $F\colon\mcC\ra\mcD$ a ribbon Frobenius functor compatible with $\opA$.
Then $F(\opA)$ is an orbifold datum in $\mcD$ and one has a braided functor $F_\opA\colon\mcC_\opA \ra \mcD_{F(\opA)}$.
\end{thm}
In our main example, given an arbitrary orbifold datum $\opA$ in $\mcC$, one will always have the compatibility between $F$ and a certain Morita transport $R(\opA)$ (see~\cite[Sec.\,3.3]{CRS3}), which is an orbifold datum obtained from $\opA$ by exchanging the algebra $A\in\mcC$ for a Morita equivalent algebra in $\mcC$.
In Sections~\ref{subsec:Morita_eq}, \ref{subsec:Morita_transports} we discuss the details regarding Morita transports of orbifold data and quickly show that they preserve associated MFCs (see Proposition~\ref{defnprp:RA_orb_datum}).

We use Theorem~\ref{thm:intro_FA_orb_datum} to prove Theorem~\ref{thm:intro_inv_orb_datum} in the following way: given a condensable algebra $B$ in a MFC $\mcC$, there is a ribbon Frobenius functor $\Iz_B=\Iz\colon\mcC\ra\mcC_B^\loc$, which maps an object $X\in\mcC$ to the localisation of the induced module $X\otimes B$, see Section~\ref{subsec:condensable_algs}.
In general $\Iz$ is compatible with a Morita transport $\opA_C$ of an arbitrary orbifold datum $\opA$ in $\mcC$, obtained by exchanging the algebra $A$ for a Morita equivalent algebra $C^*\otimes A \otimes C$.
This yields a functor $\Iz_\opA\colon \mcC_\opA\simeq\mcC_{\opA_C}\ra (\mcC_B^\loc)_{\Iz(\opA_C)}$, which can be checked to be an equivalence by hand (this is done in Appendix~\ref{appsec:proof_Iz_is_equiv}).
Specifying this to the trivial orbifold datum in $\mcC$ then yields Theorem~\ref{thm:intro_inv_orb_datum}.

\medskip

\begin{figure}
\captionsetup{format=plain, indention=0.5cm}
%\centering
\begin{subfigure}[b]{0.49\textwidth}
	\centering
	\pic[1.25]{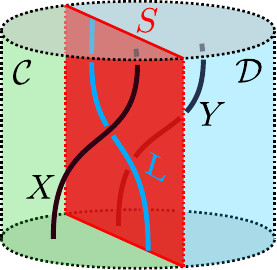}
	\caption{}
	\label{fig:Witt_ito_defects}
\end{subfigure}
\begin{subfigure}[b]{0.49\textwidth}
	\centering
	\pic[1.25]{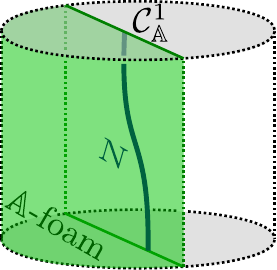}
	\caption{}
	\label{fig:CA1_lines}
\end{subfigure}
\caption{
(a) A domain wall $S$ between TQFTs $Z^{\operatorname{RT}}_\mcC$, $Z^{\operatorname{RT}}_\mcD$ comes with a category $\mcS$ of lines $L$ within $S$ and a functor $\mcC\boxtimes\widetilde{\mcD}\ra\mcZ(\mcS)$ which describes moving bulk lines $X\in\mcC$, $Y\in\mcD$ next to $S$.
(b) There is a natural domain wall between $Z^{\operatorname{orb}\opA}_\mcC\cong Z^{\operatorname{RT}}_{\mcC_\opA}$ and $Z^{\operatorname{RT}}_\mcC$ at which $\opA$-foam terminates.
}
\label{fig:Witt_eq}
\end{figure}
\subsubsection*{Witt equivalence}
Our final result expresses the notion of \textit{Witt equivalence of MFCs}~\cite{DMNO} in terms of orbifold data in them.
Recall that two MFCs $\mcC$, $\mcD$ are called Witt equivalent if there exists a spherical fusion category $\mcS$ and a ribbon equivalence $\mcC\boxtimes\widetilde{\mcD}\simeq\mcZ(\mcS)$, where $\widetilde{\mcD}$ denotes the category with reversed braiding.
There is also an alternative formulation of Witt equivalence (see~\cite[Prop.\,5.15]{DMNO}): there exists another MFC $\mcE$ and two condensable algebras $B,B'\in\mcE$ such that $\mcC\simeq\mcE_B^\loc$, $\mcD\simeq\mcE_{B'}^\loc$.
We prove (see Theorem~\ref{thm:Witt_vs_orb_eq}):
\begin{thm}
\label{thm:intro_Witt_orb_equiv}
Two MFCs $\mcC$ and $\mcD$ are Witt equivalent if and only if $\mcD\simeq\mcC_\opA$ for some orbifold datum $\opA$ in $\mcC$.
\end{thm}
One direction in the above theorem is immediate since the orbifold datum $\Iz_B(\opB'_C)$ in $\mcC\simeq\mcE_B^\loc$ has the associated MFC equivalent to $\mcE_{B'}^\loc\simeq\mcD$ (here $\opB'$ is the orbifold datum in $\mcE$ for the $B'$-condensation and $\opB'_C$ is the above mentioned Morita transport of it).
The other direction is inspired by a very intuitive interpretation of Witt equivalence in terms of TQFTs due to~\cite{FSV}.
It argues that having a Witt trivialisation $\mcC\boxtimes\widetilde{\mcD}\simeq\mcZ(\mcS)$ is the same as providing one with a domain wall $S$ separating bulk theories $Z^{\operatorname{RT}}_\mcC$ and $Z^{\operatorname{RT}}_\mcD$.
Indeed, the objects of $\mcS$ label line defects within $S$, whereas the two functors $\mcC,\mcD\ra\mcS$ are seen as describing the process of moving line defects from the two bulks to $S$.
Such bulk lines merely hover next to the domain wall and can cross to the other side of any other line $L\in\mcS$ (see Figure~\ref{fig:Witt_ito_defects}), so that as objects in $\mcS$ they have natural half-braidings.
The two functors are therefore central and combine into a candidate Witt trivialisation $\mcC\boxtimes\widetilde{\mcD}\ra\mcZ(\mcS)$.
To argue that it is an equivalence one in principle needs to impose assumptions on $S$, which we will ignore as we are only interested in a special case.
In particular, we note that the two bulk theories $Z^{\operatorname{RT}}_{\mcC_\opA}\cong Z^{\operatorname{orb}\opA}_{\mcC}$ and $Z^{\operatorname{RT}}_\mcC$ have a natural domain wall, at which the $\opA$-foam terminates, see Figure~\ref{fig:CA1_lines}.
One labels the line defects within this domain wall with the objects of the category $\mcC_\opA^1$, which consist of triples $(N,\tau_1,\taubar{1})$, defined analogously as those of $\mcC_\opA$, but lacking the other crossing morphisms $\tau_2$, $\taubar{2}$.
This has the effect of confining an $N$-labelled line to the domain wall; a similarly defined category $\mcC_\opA^2$ would then label lines in a domain wall at which the $\opA$-foam terminates from the other side.
Altogether this helps one to guess functors (see Proposition~\ref{prp:CCA-ZCA1_equiv}, Remark~\ref{rem:CAC-ZCA2_equiv})
\begin{equation}
\label{eq:intro_Witt_triv}
\mcC_\opA \boxtimes \widetilde{\mcC} \ra \mcZ(\mcC_\opA^1) \, , \quad
\mcC \boxtimes \widetilde{\mcC_\opA} \ra \mcZ(\mcC_\opA^2) \, ,
\end{equation}
which can then be checked to be equivalences by hand.
There are extra complications due to $\mcC_\opA^1$, $\mcC_\opA^2$ being multifusion, instead of spherical fusion, which can be fixed by replacing them with component categories (see Lemma~\ref{lem:ZCAi-ZF_equiv}).

\subsubsection*{Some corollaries}
Firstly, the above results can be used to better understand the plethora of examples of orbifold data.
In particular, the above mentioned examples within the categories of Ising type constructed in~\cite{MR2} were inspired by but not strictly proved to be related to the $E_6$ algebra condensation of $\mcC(sl(2),10)$.
In this paper we use the explicit construction of the orbifold datum in Theorem~\ref{thm:intro_inv_orb_datum} to prove this (see Theorem~\ref{thm:E6_is_inversion} below).

Secondly, we explore very briefly the notion of \textit{unital orbifold data} in MFCs.
This is inspired by the point of view, seeing an orbifold datum $\opA$ in a MFC $\mcC$ as a certain higher algebra object in the monoidal bicategory $\mcA{}lg_\mcC$ of algebras in $\mcC$, their bimodules and bimodule morphisms.
The current definition of an orbifold datum via the conditions~\eqrefO{1}--\eqrefO{8} however does not include the analogue of a unit.
The reason for this is that the notion of an orbifold datum in 3-dimensional defect TQFTs was originally derived in~\cite{CRS1} using a geometric rather than algebraic approach, by analysing dual Pachner moves on triangulated manifolds.
Already in dimension 2 this yields a datum, similar to a Frobenius algebra without unit/counit and satisfying additional conditions, ensuring that the Frobenius move holds for all consistent orientations for 1-strata.
The proof of Theorem~\ref{thm:intro_Witt_orb_equiv} however shows that at least up to the equivalence of associated MFCs, the orbifold data in $\mcC$ can always taken to be unital.
This is because the orbifold data of the form $\Iz_B(\opB'_C)$, described below the statement of Theorem~\ref{thm:intro_Witt_orb_equiv}, happen to be unital (see Theorem~\ref{thm:unital_orb_data}).

Finally, we note that the equivalence~\eqref{eq:intro_Witt_triv} relates orbifold data in MFCs to other constructions on fusion categories.
In particular, a monoidal category enriched over a braided category $\mcB$ was shown in~\cite{MP} to be equivalently given by a pair $(\mcA,F)$, where $\mcA$ is a fusion category and $F\colon\mcB\ra\mcZ(\mcA)$ a braided functor.
Associated to $(\mcA,F)$ there is another braided category, the underlying category of the enriched Drinfeld centre~\cite{KZ3}, which is equivalent to the centraliser $F(\mcA)'$ (see~\eqref{eq:commutator_in_MFC}) of the image $F(\mcA)$ in $\mcZ(\mcA)$ (see~\cite[Thm.\,5.3]{KZ3}).
A restriction e.g.\ of the second functor in~\eqref{eq:intro_Witt_triv} to $\mcC\ra\mcZ(\mcC_\opA^2)$ proves that an orbifold datum yields a category enriched in $\mcC$, and its associated braided category is exactly $\widetilde{\mcC_\opA}$.
Theorem~\ref{thm:intro_Witt_orb_equiv} also shows that this relation is two sided: given an enriched over a MFC $\mcC$ category $(\mcS,F)$ (where it is more natural to take $\mcS$ to be spherical fusion and $F\colon\mcC\ra\mcZ(\mcS)$ a ribbon functor), the associated braided category with the reversed braiding $\widetilde{F(\mcC)'}$ is a MFC Witt equivalent to $\mcC$ (see~\cite[Cor.\,7.8]{Mug2}, \cite[Prop.\,2.11, Thm.\,3.14]{DGNO} restated as Theorem~\ref{thm:MFCs_factor} below) and therefore comes from an orbifold datum in $\mcC$.
In the end this also opens a way to relate orbifold data to other similar notions, e.g.\ module tensor categories~\cite{MPP} and anchored planar algebras~\cite{HPT}.

\medskip

This paper is organised as follows:
The introductory section is the only one using the language of TQFTs, the following sections will present the material in a purely algebraic way.
Section~\ref{sec:preliminaries} reviews the various notions related to monoidal and (multi)fusion categories.
We do assume familiarity with most of these notions, in particular we will heavily rely on graphical calculus.
Section~\ref{sec:ssFAs} reviews the use of Frobenius algebras in (multi)fusion categories.
Although the material in it is also not new, we still present it somewhat thoroughly in order to emphasise some technical nuances that are less abundant in the literature (these arise mostly due to our use of the separability condition~\eqref{eq:intro_separability_cond} and the requirement for all categories/functors we encounter to be pivotal and the Frobenius algebras symmetric).
Section~\ref{sec:monoidal_FFs} gives an introduction to Frobenius functors which are needed to state one of the main results of this work (see Theorem~\ref{thm:intro_FA_orb_datum} above).
Also here we find it necessary to be sufficiently thorough as some notions are either new or slightly differ from those found in the literature.
Section~\ref{sec:gen_orbifolds} is dedicated to reviewing the notion of an orbifold data in MFCs as well as stating and proving Theorem~\ref{thm:intro_FA_orb_datum}.
Section~\ref{sec:condensation_inversion} reviews condensable algebras and constructs the condensation inversion orbifold datum in Theorem~\ref{thm:intro_inv_orb_datum} (more technical proofs in this section are moved to Appendix~\ref{appsec:proof_Iz_is_equiv}).
The example of inverting the $E_6$ algebra condensation is also presented here.
Finally, Section~\ref{sec:Witt_equiv} includes the proof of Theorem~\ref{thm:intro_Witt_orb_equiv} on Witt equivalence as well as its implication on the existence of unital orbifold data.

\subsubsection*{Acknowledgements}
I would like to thank
Nils Carqueville,
Lukas M\"uller and
Ingo Runkel
for helpful suggestions and many comments on the draft.
I am also thankful to Vincent Koppen and Christoph Schweigert for numerous productive discussions which brought many of the notions used in this work into my attention, and likewise to David Penneys and Gregor Schaumann for, among other things, sharing the insights on unitality on which Section~\ref{subsec:unital_orb_data} is based.
This work was initiated at the University of Hamburg and carried out at the Max Planck Institute for Mathematics in Bonn, the support of both institutions is most appreciated.
\goodbreak

\section{Preliminaries}
\label{sec:preliminaries}
In this section we list some notions and results related to monoidal categories that will be used throughout the paper.
We assume familiarity with these notions, the exposition here is mostly to introduce our preferred terminology, conventions and notation.
They are mostly in accord with the books~\cite{EGNO,TV}, which can be consulted for more details.

\subsection{Monoidal categories}
Let us start with the most basic notions:
\begin{itemize}
\item
A \textit{monoidal category} $\mcC = (\mcC,\otimes,\opid,a,l,r)$ is a category equipped with a monoidal (or tensor) product functor $\otimes\colon\mcC\times\mcC\ra\mcC$ (we omit the symbol $\otimes$ sometimes) and the associator and unitor natural isomorphisms $a\colon (- \otimes -) \otimes - \Ra - \otimes (- \otimes -)$, $l\colon \opid \otimes - \Ra \Id_\mcC$, $r\colon - \otimes \opid \Ra \Id_\mcC$, satisfying the standard pentagon and triangle identities, i.e.\ for all $X,Y,Z,W\in\mcC$ one has
\begin{gather} \nonumber
(\id_X \otimes a_{Y,Z,W}) \circ a_{X,YZ,W} \circ (a_{X,Y,Z} \otimes \id_W) = a_{X,Y,ZW}\circ a_{XY,Z,W}\\
(\id_X\otimes l_Y)\circ a_{X,\opid,Y} = r_X \otimes \id_Y \, .
\end{gather}
There is also a similar notion of a (left) module category $\mcM$ over a monoidal category $\mcC$, which comes equipped with an action functor $\triangleright\colon\mcC\times\mcM\ra\mcM$ and natural isomorphisms $(-\otimes -)\triangleright- \Ra - \triangleright (- \triangleright -)$ and $\opid \triangleright - \Ra \Id_\mcM$.
\item
Given two monoidal categories $\mcC=(\mcC,\otimes,\opid,a,l,r)$ and $\mcD=(\mcD,\otimes',\opid',a',l',r')$, a \textit{monoidal functor} $F\colon\mcC\ra\mcD$ is a functor equipped with a \textit{monoidal structure} $(F_2, F_0)$, consisting of a natural transformation $F_2\colon F(-) \otimes' F(-) \Ra F(-\otimes -)$ and a morphism $F_0\colon\opid' \ra F(\opid)$ such that for all $X,Y,Z$ the coherence relations
\begin{gather} \nonumber
\begin{align*}
F_2(X,YZ)  \circ (\id_{F(X)} & \otimes' F_2(Y,Z)\circ a'_{F(X),F(Y),F(Z)})\\
&=
F(a_{X,Y,Z}) \circ F_2(XY,Z) \circ (F_2(X,Y) \otimes' \id_{F(Z)}) \, ,
\end{align*}\\ \nonumber
l'_{F(X)} =
F(l_X) \circ F_2(\opid,X)\circ (F_0 \otimes' \id_{F(X)}) \, ,\\ \label{eq:monoidal_funct_ids}
r'_{F(X)} =
F(r_X) \circ F_2(X,\opid) \circ (\id_{F(X)} \otimes' F_0) \, ,
\end{gather}
hold.
$(F,F_2,F_0)$ is said to be a \textit{strong monoidal functor} if for all $X,Y\in\mcC$ the morphisms $F_2(X,Y)$ and $F_0$ are isomorphisms and a \textit{weak monoidal functor} otherwise.
Given two monoidal functors $F,G\colon\mcC\ra\mcD$, a natural transformation $\varphi\colon F \Ra G$ is said to be \textit{monoidal} if $\varphi_{X\otimes Y}\circ F_2(X,Y) = G_2(X,Y)\circ(\varphi_X \otimes' \varphi_Y)$ and $\varphi_\opid\circ F_0 = G_0$.

Similarly one defines functors between module categories and their natural transformations.
\item
A \textit{comonoidal functor} between monoidal categories $\mcC$ and $\mcD$ comes equipped with a comonoidal structure $\overline{F_2}\colon F(-\otimes -)\Ra F(-) \otimes' F(-)$, $\overline{F_0}\colon F(\opid)\ra\opid'$ such that for all $X,Y,Z\in\mcC$ one has
\begin{gather}\nonumber
\begin{align*}
a'_{F(X),F(Y),F(Z)}\circ (&\overline{F_2}(X,Y) \otimes' \id_{F(Z)})\circ\overline{F_2}(XY,Z)\\
&=
(\id_{F(X)} \otimes' \overline{F_2}(Y,Z)) \circ \overline{F_2}(X,YZ) \circ F(a_{X,Y,Z}) \, ,
\end{align*}\\ \nonumber
F(l_X)=l'_F(X)\circ(\overline{F_0} \otimes' \id_{F(X)})\circ \overline{F_2}(\opid,X) \, ,\\
F(r_X)=r'_{F(X)}\circ (\id_{F(X)}\otimes' \overline{F_0}) \circ \overline{F_2}(X,\opid) \, .
\end{gather}
As with monoidal functors, a comonoidal structure is called \textit{strong/weak} if $(\overline{F_2}, \overline{F_0})$ are/are not isomorphisms.
Note that if $F$ is strong monoidal, $(\overline{F_2} = F_2^{-1}, \overline{F_0} = F_0^{-1})$ is a strong comonoidal structure.
\end{itemize}
From now on, when talking about monoidal functors, unless specified otherwise we will always mean strong monoidal.
Weak (co)monoidal functors are somewhat less abundant in the literature, but we will indeed use them in Section~\ref{sec:monoidal_FFs}.

\medskip

Morphisms in a monoidal category can be depicted with the help of \textit{graphical calculus} (see e.g.\ \cite[Ch.\,2]{TV}), which we will heavily rely upon in this paper.
According to our convention the diagrams are to be read from bottom to top.
A morphism will either be depicted by a labelled point or a coupon.

\medskip

Let us now look at various notions of dualities in monoidal categories:
\begin{itemize}
\item
A monoidal category $\mcC$ is said to be \textit{rigid} if each object $X\in\mcC$ has a left dual $(X^*, \ev_X\colon X^*X\ra\opid, \coev_X\colon\opid \ra XX^*)$ and a right dual
$({^*X}, \evt_X\colon {^*X}X\ra\opid, \coevt_X\colon\opid \ra X{^*X})$ where the (co)evaluation morphisms satisfy the zig-zag identities, see e.g.\ \cite[Sec.\,2.10]{EGNO}.
Two choices of left/right duals for an object $X$ are canonically isomorphic.
A fixed choice of duals equips $\mcC$ with functors $(-)^*, {^*(-)}\colon \mcC \ra \mcC^\opp$.
A monoidal functor $F\colon\mcC\ra\mcD$ automatically preserves duals, the canonical isomorphisms provide one with natural isomorphisms $F_l\colon F((-)^*) \Ra F(-)^*$, $F_r\colon F(^*(-)) \Ra {^*F(-)}$.
\item
A \textit{pivotal category} is a rigid category $\mcC$ equipped with a monoidal natural isomorphism $\d\colon \Id_\mcC \Ra (-)^{**}$ called the \textit{pivotal structure}.
In a pivotal category each left dual is canonically a right dual and in terms of graphical calculus the corresponding (co)evaluation morphisms of an object $X\in\mcC$ are denoted/defined by
\begin{align} \nonumber
&
\ev_X = \pic[1.25]{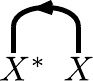} := \pic[1.25]{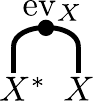} \, ,
&&
\coev_X = \pic[1.25]{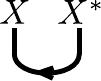} := \pic[1.25]{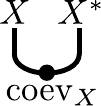} \, ,\\
&
\evt_X = \pic[1.25]{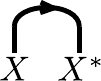} := \pic[1.25]{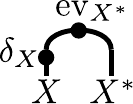} \, ,
&&
\coevt_X = \pic[1.25]{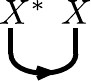} := \pic[1.25]{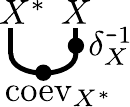} \, .
\end{align}
An equivalent way to equip $\mcC$ with a pivotal structure is to provide each object $X$ with a tuple $(X^*,\ev_X,\coev_X,\evt_X,\coevt_X)$ such that for all $X,Y\in\mcC$ and $f\in\mcC(X,Y)$ one has
\begin{equation}
\pic[1.25]{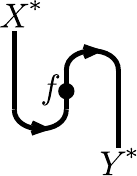} =  \pic[1.25]{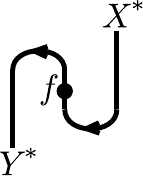} \, ,
\pic[1.25]{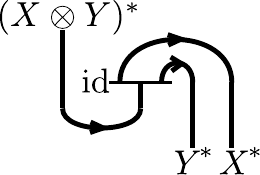} =  \pic[1.25]{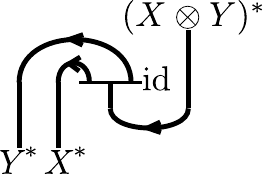} \, .
\end{equation}
As indicated above, the strands in the graphical calculus of a pivotal category are oriented with a downwards orientation indicating the dual object.
The axioms of a pivotal category allow one to deform a diagram up to a plane isotopy without changing the morphism it corresponds to.
\item
Given an object $X$ and a morphism $f\in\End_\mcC(X)$, one defines the \textit{left/right trace of $f$} to be the following morphisms in $\End_\mcC(\opid)$:
\begin{equation}
\tr_\mcC^l f :=\pic[1.25]{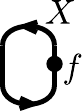} \, , \qquad
\tr_\mcC^r f :=\pic[1.25]{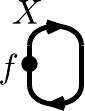} \, .
\end{equation}
Similarly, a left/right categorical dimension of $X\in\mcC$ is defined by $\dim_\mcC^{l/r} X := \tr_\mcC^{l/r} \id_X$.
$\mcC$ is called \textit{spherical} if one has $\tr_\mcC^l f = \tr_\mcC^r f (=: \tr_\mcC f)$ (and consequently $\dim_\mcC^l X = \dim_\mcC^r X =: \dim_\mcC(X)$).
\item
Let $(\mcC,\d)$, $(\mcD,\d')$ be pivotal categories.
A monoidal functor $F\colon\mcC\ra\mcD$ is called \textit{pivotal} if the following diagram commutes:
\begin{equation}
\label{eq:F-pivotal}
\begin{tikzpicture}[baseline={([yshift=-.5ex]current bounding box.center)}]
\node (P1) at (-3,2)    {$F(X)$};
\node (P2) at  (0,2)    {$F(X^{**})$};
\node (P3) at  (-3,0)   {$F(X)^{**}$};
\node (P4) at  (0,0) {$F(X^*)^*$};

\path[commutative diagrams/.cd,every arrow,every label]
(P1) edge node {$F(\delta_X)$} (P2)
(P2) edge node {$F_1(X^*)$} (P4)
(P1) edge node[swap] {$\delta_{F(X)}'$} (P3)
(P3) edge node {$F^*_1(X)$} (P4);
\end{tikzpicture} \, ,
\end{equation}
where one uses the notation $F_1\colon ((-)^*)\Ra F(-)^*$ for the dual functor $F_l$ as it is enough to only consider left duals in pivotal categories.
\end{itemize}

Finally, we will need some definitions on braided categories:
\begin{itemize}
\item
A \textit{braided category} is a monoidal category $\mcC$ equipped with a braiding natural isomorphism $\{c_{X,Y}\colon XY \ra YX \}_{X,Y\in\mcC}$ which satisfies the usual hexagon identities.
They will be denoted by
\begin{equation}
c_{X,Y} = 
\pic[1.25]{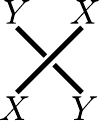} \, , \qquad
c_{X,Y}^{-1} =
\pic[1.25]{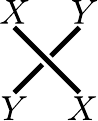} \, .
\end{equation}
If $\mcC$ is braided, $\widetilde{\mcC}$ will denote the braided category with the reversed braiding $\{c_{Y,X}^{-1}\colon XY \ra YX\}_{X,Y\in\mcC}$.
\item
A monoidal functor $F\colon\mcC\ra\mcD$ between braided categories $(\mcC,c)$, $(\mcD,c')$ is called \textit{braided} if it preserves braidings, i.e.\ for all $X,Y\in\mcC$ one has
\begin{equation}
F(c_{X,Y}) \circ F_2(X,Y) = F_2(Y,X)\circ c'_{F(X),F(Y)} \, .
\end{equation}
\item
If $\mcC$ is both braided and pivotal, for an object $X\in\mcC$ one defines the \textit{left/right twists} to be the natural isomorphisms $\theta^l,\theta^r\colon\Id_\mcC\ra\Id_\mcC$ (in general not monoidal) defined by
\begin{equation}
\theta_X^l =
\pic[1.25]{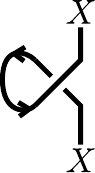} , \,
\theta_X^r =
\pic[1.25]{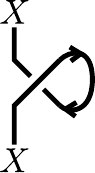}  , \,
(\theta_X^l)^{-1} =
\pic[1.25]{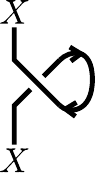}  , \,
(\theta_X^r)^{-1} =
\pic[1.25]{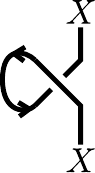}  .
\end{equation}
$\mcC$ is called \textit{ribbon} if one has $\theta^l_X = \theta^r_X (=: \theta_X)$.
In a ribbon category, the graphical calculus allows one to deform the diagrams as if they were ribbon tangles embedded in $\opR^3$.
\item
A \textit{ribbon functor} $F\colon\mcC\ra\mcD$ between two ribbon categories $\mcC$, $\mcD$ is a braided functor which preserves twists, i.e.\ for all $X\in\mcC$ one has $F(\theta_X) = \theta_{F(X)}$.
A braided functor between ribbon categories is ribbon if and only if it is pivotal.
\item
Given a monoidal category $\mcC$, its \textit{Drinfeld centre} is a braided category $\mcZ(\mcC)$ whose objects are pairs $(X,\gamma)$, where $X\in\mcC$ and $\gamma\colon X \otimes - \Ra - \otimes X$ is a \textit{half-braiding}, i.e.\ a natural isomorphism satisfying the hexagon identity for all $U,V\in\mcC$:
\begin{equation}
a_{U,V,X}\circ \gamma_{UV} \circ a_{X,U,V} =
(\id_U \otimes \gamma_V)\circ a_{U,X,V}\circ (\gamma_U \otimes \id_V) \, .
\end{equation}
A morphism $f\colon(X,\gamma)\ra (Y,\gamma')$ between two objects in $\mcZ(\mcC)$ is a morphism $f\colon X\ra Y$ which commutes with the half-braidings, i.e.\ for all $U\in\mcC$ one has $\gamma'_U\circ (f \otimes \id_U) = (\id_U \otimes f) \circ \gamma_U$.
The braiding in $\mcZ(\mcC)$ is defined to be: $c_{(X,\gamma),(Y,\gamma')} = \gamma_Y$.
If $\mcC$ is pivotal, so is $\mcZ(\mcC)$.
\end{itemize}

\subsection{Multifusion categories}
Throughout the entire paper, let $\opk$ be an algebraically closed field which, except for Section~\ref{sec:Witt_equiv}, is \textit{not} assumed to be of characteristic 0.
We assume familiarity with $\opk$-linear, abelian, semisimple, multitensor and multifusion categories, for more details see e.g.\ \cite[Ch.\,1\&{}4]{EGNO}, \cite[Ch.\,4]{TV}.
\begin{itemize}
\item
In a linear category $\mcA$, a \textit{direct sum} of objects $X_1,\dots,X_n\in\mcC$ is an object $X=X_1\oplus\dots\oplus X_n$ together with projection/inclusion morphisms $\pi_i\colon X \lra X_i : \imath_i$ such that $\sum_i \imath_i\circ\pi_i = \id_X$ and $\pi_j\circ\imath_i = \d_{ij} \cdot \id_{X_i}$.
\item
An object $X$ in a $\opk$-linear category $\mcA$ is called \textit{simple} if $\dim\End_\mcA(X)=1$, in which case one identifies $\End_\mcA(X)\cong\opk$ by $\id_X \mapsto 1$.
$\mcA$ is called \textit{semisimple} if every object is isomorphic to a finite direct sum of simples.
We will denote by $\Irr_\mcA$ a fixed set of representatives of isomorphism classes of all simple objects of $\mcA$.
If $\Irr_\mcA$ is finite, $\mcA$ is called \textit{finitely semisimple}.
\item
For any two $\opk$-linear categories $\mcA$, $\mcB$, their \textit{direct sum} $\mcA\oplus\mcB$ is the linear category of formal direct sums of objects in $\mcA$ and $\mcB$.
If $\mcA$, $\mcB$ are semisimple, their Deligne product $\mcA\boxtimes\mcB$ is the semisimple category of formal direct sums of objects of type $X\boxtimes Y$ with $X\in\mcA$, $Y\in\mcB$ (for general definition see~\cite[Sec.\,1.11]{EGNO}).
\item
Functors between $\opk$-linear categories will always be assumed to be linear (i.e.\ $\opk$-linear on spaces of morphisms).
Such functors then automatically preserve direct sums.
For $\mcA$, $\mcB$ semisimple, a  functor $F\colon\mcA\ra\mcB$ is called \textit{surjective} if every object $Y\in\mcB$ is a direct summand of $F(X)$ for some $X\in\mcA$.
\end{itemize}
Proving some of the results of this paper will involve showing that two semisimple categories are equivalent.
For that we always use
\begin{prp}
\label{prp:ssi_funct_equiv}
For $\mcA$, $\mcB$ semisimple, a functor $F\colon\mcA\ra\mcB$ is an equivalence if and only if it is fully faithful (i.e.\ an isomorphism on the spaces of morphisms) and surjective.
\end{prp}
\begin{proof}
$F$ is an equivalence if it has an inverse or equivalently is both fully faithful and essentially surjective.
If $F$ has an inverse it is obviously surjective.
Conversely, it is enough to show that every simple of $\mcB$ is in the essential image of $F$.
This follows from $F$ mapping simples of $\mcA$ to simples of $\mcB$ - if $i\in\Irr_\mcB$ is a direct summand of $F(X)$, $X\in\mcA$, decomposing $X \cong \bigoplus_k i_k$, $i_k\in\Irr_\mcA$, one gets $i \cong F(i_k)$ for some index $k$.
\end{proof}

Next we look at $\opk$-linear categories with monoidal structure:
\begin{itemize}
\item
A \textit{multifusion category} is a $\opk$-linear, finitely semisimple, rigid monoidal category $\mcA$ whose monoidal product is bilinear on morphism spaces.
If the monoidal unit $\opid\in\mcA$ is simple, $\mcA$ is called \textit{fusion}.
In this case we always assume $\opid\in\Irr_\mcA$.
\item
If $\mcA$, $\mcB$ are multifusion, so are $\mcA\oplus\mcB$ and $\mcA\boxtimes\mcB$.
In particular, we call $\mcA$ \textit{indecomposable} if it is not equivalent (as a multifusion category) to a direct sum of non trivial (i.e.\ not equivalent to $\{0\}$) multifusion categories.
\item
Let $\opid=\bigoplus_i \opid_i$ be the decomposition of the tensor unit of a multifusion category $\mcA$ into simples.
Then all $\opid_i$ are mutually non-isomorphic and $\opid_i \otimes \opid_j \cong 0$ for $i\neq 0$.
Moreover one has the decomposition $\mcA\simeq\bigoplus_{ij}\mcA_{ij}$, where $\mcA_{ij} := \opid_i \otimes \mcA \otimes \opid_j$ are the so-called \textit{component categories} of $\mcA$.
Note that each $\mcA_{ii}$ is a fusion category.
\end{itemize}
The non-semisimple counterpart of a (multi)fusion category is called a \textit{(multi)tensor category} (see e.g.\ \cite[Ch.\,4]{EGNO}).
The categories that we are going to encounter later on will however always be either fusion or multifusion.

\medskip

It follows from the last item in the above list that if $\mcB$ is a braided multifusion category, then $\mcB\simeq\bigoplus\mcB_{ii}$, i.e.\ its non-diagonal component categories vanish.
In particular, the Drinfeld centre $\mcZ(\mcA)$ of a multifusion category $\mcA$ is a direct sum  of tensor categories ($\mcZ(\mcA)$ need not always be semisimple, see Proposition~\ref{prp:ZS_is_MFC} for a criterion).
In fact one has (see~\cite[Thm.\,2.5.1]{KZ1}):
\begin{prp}
\label{prp:ZA-ZAii_equiv}
Let $\mcA$ be an indecomposable multifusion category.
Then $\mcZ(\mcA)\simeq\mcZ(\mcA_{ii})$ as braided tensor categories.
\end{prp}
The explicit equivalence $\mcZ(\mcA)\xra{\sim}\mcZ(\mcA_{ii})$ is given by
\begin{equation}
\label{eq:ZA-ZAii_equiv}
(X,\gamma) \mapsto \big(\opid_i X \opid_i, ~ \{\opid_i X \opid_i U \xra{\sim} \opid_i XU \opid_i \xra{\id_{\opid_i}\otimes\gamma_U\otimes\id_{\opid_i}} \opid_i UX \opid_i \xra{\sim} U \opid_i X \opid_i  \}_{U\in\mcA_{ii}} \big) \, ,
\end{equation}
where one uses the unitors of $\mcA_{ii}$ to define the half-braiding.

\medskip

We will mostly work with pivotal multifusion categories.
For such a category $\mcC$, the left/right traces of morphisms and categorical dimensions of objects, both being endomorphisms of $\opid\cong\bigoplus_i\opid_i$, are tuples of scalars.
The left/right dimensions of a simple object $i\in\Irr_\mcC$ are never 0 since one has for example $\mcC(ii^*,\opid)\cong\mcC(\opid,ii^*)\cong\mcC(i,i)\cong\opk$ and the (co)evaluation morphisms are non-zero.
This implies:
\begin{prp}
\label{prp:surj_pivotal_functs}
Let $\mcC, \mcD$ be pivotal multifusion and $F\colon\mcC\ra\mcD$ a pivotal functor.
Then if $F$ is surjective on the spaces of morphisms it is automatically fully faithful.
\end{prp}
\begin{proof}
Because it is surjective, $F$ can map a simple of $\mcC$ either to a simple of $\mcD$ or to $0$.
Since $F$ is pivotal and hence preserves the left/right dimensions, which for simple objects are non-zero, the $0$ case is excluded.
\end{proof}
If $\mcC$ is in addition fusion, the traces/dimensions are scalars obtained upon the standard identification $\mcC(\opid,\opid)\cong\opk$.
When writing $(\dim^{l/r}_\mcC X)^{1/2}$ we will always mean a fixed choice of the square root of the left/right categorical dimension of an object $X\in\mcC$.

If $\mcS$ is a spherical fusion category, one defines its global dimension to be the scalar
\begin{equation}
\Dim\mcS := \sum_{i\in\Irr_\mcC} (\dim_\mcC i)^2 \, .
\end{equation}

\subsection{Modular fusion categories (MFCs)}
The following notion is of central importance of this paper.
As before we assume familiarity with the material in this section which can be acquired e.g.\ in~\cite[Sec.\,II.1]{Tu2}, \cite[Ch.\,3]{BakK}, \cite[Ch.\,8]{EGNO}.
\begin{defn}
\label{def:MFC}
A \textit{modular fusion category (MFC)} is a ribbon fusion category $\mcC$ satisfying one of the following equivalent conditions:
\begin{enumerate}[i)]
\item
The braiding of $\mcC$ is non-degenerate, i.e.\ if $T\in\mcC$ is such that for all $X\in\mcC$ one has
\begin{equation}
    c_{T,X} \circ c_{X,T} = \id_{X\otimes T} \, ,
\end{equation}
then $T\cong\opid^{\oplus n}$ for some $n\ge 0$.
\item
The $s$-matrix
\begin{equation}
    s_{ij} = \tr_\mcC(c_{j,i}\circ c_{i,j}) \, , \quad i,j\in\Irr_\mcC
\end{equation}
is non-degenerate.
\item
The functor $\mcC \boxtimes \widetilde{\mcC} \ra \mcZ(\mcC)$, $X\boxtimes Y \mapsto (X\otimes Y, \gamma^{\operatorname{dol}})$, where
\begin{equation}
\gamma^{\operatorname{dol}}_U := \pic[1.25]{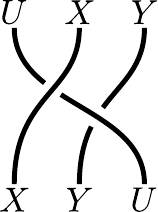} \, , \quad U\in\mcC
\end{equation}
is a so-called `dolphin' half-braiding, is a ribbon equivalence.
\end{enumerate}
\end{defn}
The trivial example of a MFC is the category $\Vect_\opk$ of finite dimensional vector spaces.
A further family of examples of MFCs is given by the following (see~\cite[Thm.\,3.16]{Mug2}, \cite[Thm.\,2.3, Rem.\,2.4]{ENO1}):
\begin{prp}
\label{prp:ZS_is_MFC}
The Drinfeld centre $\mcZ(\mcS)$ of a spherical fusion category $\mcS$ is a MFC if and only if $\Dim\mcS\neq 0$.
For $\operatorname{char}\opk=0$ the condition $\Dim\mcS\neq 0$ holds automatically.
\end{prp}

For a MFC $\mcC$ we define the following object/endomorphisms:
\begin{equation}
\label{eq:Kirby_col}
C := \bigoplus_{i\in\Irr_\mcC} i \, , \qquad
\dhalf := \bigoplus_{i\in\Irr_\mcC} (\dim_\mcC i)^{1/2} \cdot \id_i ~ \in\End_\mcC C \, , \qquad
d := \dhalf \circ \dhalf \, .
\end{equation}
In graphical calculus, a strand labelled by $C$ and having a single $d$-labelled insertion is said to carry the \textit{Kirby colour}.
The global dimension of a MFC $\mcC$
\begin{equation}
    \Dim\mcC = \tr_\mcC d
\end{equation}
is automatically non-zero.
In what follows we will often find the `scissors' identity useful (see~\cite[Cor.\,3.1.11]{BakK}):
\begin{equation}
\label{eq:scissors_id}
\pic[1.25]{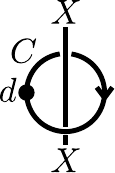} =
\Dim\mcC \cdot \sum_i
\pic[1.25]{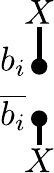} \, ,
\end{equation}
where $b_i$ run through a basis of $\mcC(\opid, X)$ and $\overline{b_i}$ is the dual basis of $\mcC(X,\opid)$ with respect to the composition pairing, i.e.\ $\overline{b_j}\circ b_i = \delta_{ij} \id_{\opid} \in \End_\mcC\opid \cong \opk$.

\section{Symmetric separable Frobenius algebras}
\label{sec:ssFAs}
In this section we review some notions and results related to Frobenius algebras in (multi)fusion categories.
Most of them can be found e.g.\ in~\cite{FRS1}.
Somewhat new is the separability condition for Frobenius algebras in Section~\ref{subsec:separability}.
It generalises the similar $\D$-separability condition required in the works~\cite{CRS2, CRS3, MR1, MR2, CMRSS2} in a way that allows one to better control the dependence of the constructions therein on the Morita class of a Frobenius algebra, see Section~\ref{subsec:Morita_eq}.

\subsection{Algebras and modules}
Let $\mcC$ be a monoidal category.
We list some definitions and conventions on algebras in $\mcC$ and their modules.
For more details one can consult e.g.\ \cite{FRS1,FFRS}.
\begin{itemize}
\item
An \textit{algebra in $\mcC$} is a tuple $A = (A,\mu,\eta)$, where $A\in\mcC$, and $\mu\colon A \otimes A \ra A$ is an associative multiplication with the unit $\eta\colon \opid \ra A$, i.e.\ morphisms depicted by \hspace{-8pt}\pic[1.25]{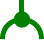}\hspace{-6pt} and \hspace{-8pt}\pic[1.25]{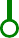}\hspace{-6pt} such that the following identities in $\mcC$ hold:
\begin{equation}
\label{eq:alg_assoc-unitality}
\pic[1.25]{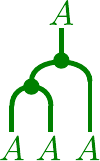} =
\pic[1.25]{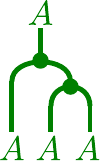} \, , \quad
\pic[1.25]{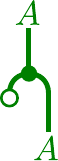} =
\pic[1.25]{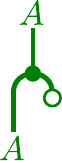} =
\pic[1.25]{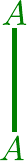} \, .
\end{equation}
\item
A \textit{coalgebra in $\mcC$} is a tuple $A = (A,\D,\vareps)$, where $A\in\mcC$ and $\D\colon A \ra A \otimes A$ is a coassociative comultiplication with the counit $\vareps\colon A \ra \opid$, to be depicted by \hspace{-8pt}\pic[1.25]{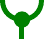}\hspace{-6pt} and \hspace{-8pt}\pic[1.25]{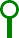}\hspace{-6pt}, i.e.\ such that
\begin{equation}
\label{eq:alg_coassoc-counitality}
\pic[1.25]{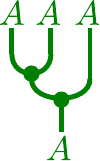} =
\pic[1.25]{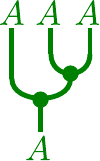} \, , \quad
\pic[1.25]{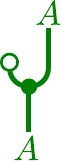} =
\pic[1.25]{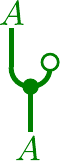} =
\pic[1.25]{31_alg_id.pdf} \, .
\end{equation}
\item
A \textit{Frobenius algebra in $\mcC$} is simultaneously an algebra and a coalgebra $A\in\mcC$ such that
\begin{equation}
\label{eq:Frob_cond}
\pic[1.25]{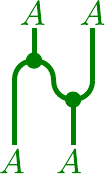} =
\pic[1.25]{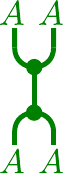} =
\pic[1.25]{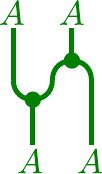} \, .
\end{equation}
If $\mcC$ is in addition pivotal, one calls a Frobenius algebra $A\in\mcC$ \textit{symmetric} if
\begin{equation}
\label{eq:Frob_symmetric}
\pic[1.25]{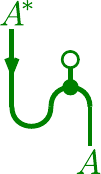} =
\pic[1.25]{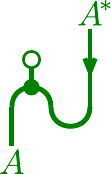} \, .    
\end{equation}
\item
If $\mcC$ is braided, an algebra (resp.\ coalgebra) $A\in\mcC$ is called \textit{commutative} (resp.\ \textit{cocommutative}) if one has
\begin{equation}
\label{eq:comm_algs}
\pic[1.25]{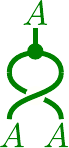} = 
\pic[1.25]{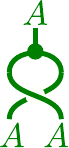} =
\pic[1.25]{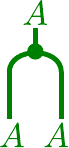} \, , \quad
\text{(resp. }
\pic[1.25]{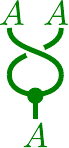} =
\pic[1.25]{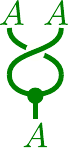} =
\pic[1.25]{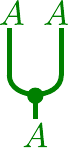} \,
\text{)} \, .
\end{equation}
\item
If $\mcC$ is braided and $A,B\in\mcC$ are algebras (resp.\ coalgebras) so is their product $A\otimes B$ with the multiplication/unit (resp.\ comultiplication/counit)
\begin{equation}
\label{eq:AxB_alg_morphs}
\pic[1.25]{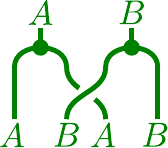} \, ,
\pic[1.25]{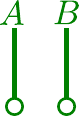} \, , \quad
\text{(resp. }
\pic[1.25]{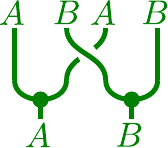} \, ,
\pic[1.25]{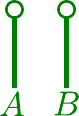}
\text{)} \, .
\end{equation}
In particular, if $A$, $B$ are Frobenius algebras, so is $A\otimes B$.
One can easily generalise this to more tensor factors by iterating.
\item
A \textit{left module} of an algebra $A\in\mcC$ is a tuple $L=(L,\l)$ where $\l\colon A \otimes L \ra L$ is an associative unital action to be depicted by \hspace{-8pt}\pic[1.25]{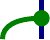}\hspace{-6pt}, i.e.\ so that one has
\begin{equation}
\pic[1.25]{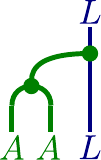} = 
\pic[1.25]{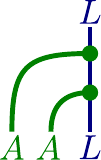} \, , \quad
\pic[1.25]{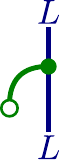} =
\pic[1.25]{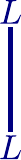} \, .
\end{equation}
A \textit{module morphism} $f\colon(L,\l) \ra (L',\l')$ is a morphism $f\colon L \ra L'$ which commutes with the $A$-actions, i.e.\ $f \circ \l = \l'\circ (\id_A \otimes f)$.
Similarly one defines the \textit{right modules} and their morphisms, and, for another algebra $B\in\mcC$, \textit{$A$-$B$-bimodules} and their morphisms.
We denote by ${_A\mcC}$ and $\mcC_A$ the categories of left and right $A$-modules and by $\mcACB$ the one of $A$-$B$-bimodules.
Note that for a Frobenius algebra $A\in\mcC$, the condition~\eqref{eq:Frob_cond} implies that the comultiplication $\Delta\colon A \ra A \otimes A$ is an $A$-$A$-bimodule morphism.
\item
If $\mcC$ is braided and $A,B\in\mcC$ are algebras, a left module $L\in{_{A\otimes B}\mcC}$ is the same as a simultaneous left $A$- and $B$-module whose actions commute in the sense that the following identity holds:
\begin{equation}
\label{eq:AxB_multimodul_action}
\pic[1.25]{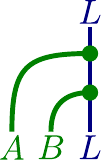} =
\pic[1.25]{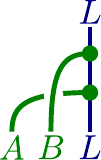} \, .
\end{equation}
It is easy to adapt this to right $A\otimes B$-modules or bimodules of tensor products of several algebras in $\mcC$.
\item
If $\mcC$ is pivotal and $A\in\mcC$ an algebra, the objects of ${_A\mcC}$ and $\mcC_A$ are in bijection, where $L\in{_A\mcC}$ is sent to the dual $L^*$ with the right action
\begin{equation}
\label{eq:lambda_dual_action}
\pic[1.25]{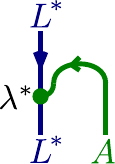}  \, .
\end{equation}
Similarly, for two algebras $A,B\in\mcC$ one has the bijection $\mcACB\ra\mcBCA$, $M\mapsto M^*$.
\item
By analogy to modules of algebras, for coalgebras $A,B\in\mcC$ one defines left and right \textit{$A$-comodules} and \textit{$A$-$B$-bicomodules}.
If $\mcC$ is pivotal and $A$, $B$ are symmetric Frobenius algebras, modules are canonically comodules where the coaction e.g.\ for a left module $L\in{_A\mcC}$ is defined to be
\begin{equation}
\label{eq:Frob_coaction}
\pic[1.25]{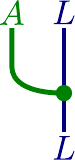} :=
\pic[1.25]{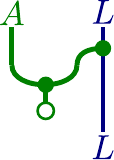} \, .
\end{equation}
\end{itemize}

In what follows we will mostly work with algebras in multifusion categories, especially in MFCs.
For algebras $A$, $B$ in a multifusion category $\mcC$, the categories of modules ${_A\mcC}$, $\mcC_A$ and $\mcACB$ are automatically abelian and $\opk$-linear, with the direct sum inherited from $\mcC$ and module morphisms automatically forming vector subspaces of the respective morphism spaces in $\mcC$.
It is easy to see that $A\oplus B$ in this case is an algebra as well, with e.g.\ the left modules having the form $L \oplus L'$, $L\in{_A\mcC}$, $L'\in{_B\mcC}$ so that one has $_{A\oplus B}\mcC \simeq {_A\mcC} \oplus {_B\mcC}$.
If an algebra $A\in\mcC$ cannot be written as a direct sum of two non-zero algebras in $\mcC$, it is called \textit{indecomposable}.

\subsection{Separability}
\label{subsec:separability}
Let $\mcC$ be a monoidal category.
For the constructions in the upcoming sections we will find it necessary to consider \textit{separable} algebras in $\mcC$
Recall that an algebra $A\in\mcC$ is called separable if the multiplication $\mu\colon A \otimes A \ra A$ has a section in $\mcACA$, i.e.\ a bimodule morphism $s\colon A \ra A \otimes A$ such that $\mu \circ s = \id_A$.
\begin{prp}
\label{prp:FA_sep_cond}
A Frobenius algebra $A\in\mcC$ is separable if and only if there is a morphism $\zeta\colon\opid\ra A$ such that
\begin{equation}
\label{eq:FA_sep_cond}
\pic[1.25]{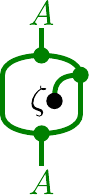} =
\pic[1.25]{1_sep_rhs.pdf} \, .
\end{equation}
\end{prp}
\begin{proof}
If such $\zeta$ exists, one has the section $s := (\id_A\otimes \mu)\circ(\id_A\otimes\zeta)\circ\D$, where $\D\colon A \ra A \otimes A$ denotes the comultiplication.
On the other hand, if $A$ is separable with a section $s$, define $\zeta := (\vareps \otimes \id_A) \circ s\circ\eta$, where $\eta\colon\opid\ra A$, $\vareps\colon A \ra \opid$ denote the unit and the counit respectively.
Then one has
\begin{equation}
\pic[1.25]{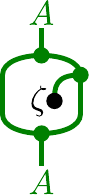} =
\pic[1.25]{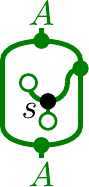} =
\pic[1.25]{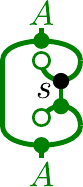} =
\pic[1.25]{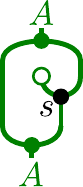} =
\pic[1.25]{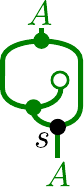} =
\pic[1.25]{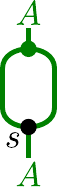} =
\pic[1.25]{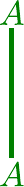}\, ,
\end{equation}
where one uses that $s$ is an $A$-$A$-bimodule morphism and therefore commutes with the (co)multiplication morphisms of $A$.
\end{proof}

Our constructions involving a separable Frobenius algebra $A\in\mcC$ will in principle be independent of the choice of a morphism $\zeta$ satisfying the condition~\eqref{eq:FA_sep_cond}.
For later convenience we will furthermore assume that $\zeta$ is given as a multiplicative square of another morphism $\psi\colon\opid\ra A$, i.e.\ $\zeta = \psi^2 := \mu\circ(\psi\otimes\psi)$ so that one has
\begin{equation}
\label{eq:FA_sep_cond_ito_psi}
\pic[1.25]{32_sep_ito_psi_1.pdf} =
\pic[1.25]{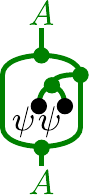} =
\pic[1.25]{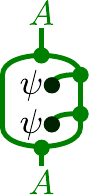} =
\pic[1.25]{1_sep_rhs.pdf} \, .
\end{equation}
If $\mcC$ is a (multi)fusion category over an algebraically closed field $\opk$, this assumption can be made without loss of generality (for the case of $\mcC$ being spherical fusion, this follows from the proof of Proposition~\ref{prp:FA_psi_inv} below).
Abusing the terminology, we will refer to such a morphism $\psi\colon\opid\ra A$ as a \textit{section of $A$} (since it provides one with a section of the multiplication map $\mu\colon A \otimes A \ra A$).

\medskip

Let $\mcC$ be pivotal and $A\in\mcC$ a symmetric separable Frobenius algebra with a section $\psi$.
For an arbitrary left $A$-module $L$, a right $A$-module $K$ and an $A$-$A$-bimodule $M$, we introduce the following endomorphisms:
\begin{equation}
\label{eq:psi_lr_defs}
\psi_l^L \hspace{-2pt}=\hspace{-4pt} \pic[1.25]{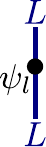} \hspace{-4pt}:=\hspace{-4pt} \pic[1.25]{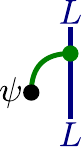}, \quad
\psi_r^K \hspace{-2pt}=\hspace{-4pt} \pic[1.25]{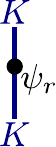} \hspace{-4pt}:=\hspace{-4pt} \pic[1.25]{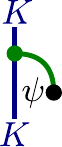}, \quad
\omega_M \hspace{-2pt}=\hspace{-4pt} \pic[1.25]{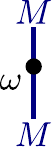} \hspace{-4pt}:=\hspace{-4pt} \pic[1.25]{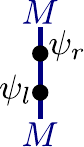} \hspace{-4pt}=\hspace{-4pt} \pic[1.25]{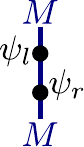}\, .
\end{equation}
They are compatible with the (co)action morphisms of $A$ in the sense that the identities
\begin{equation}
\label{eq:psi_comm_identities}
\pic[1.25]{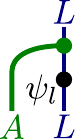} \hspace{-4pt}=\hspace{-4pt} \pic[1.25]{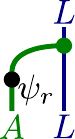} \, , \quad
\pic[1.25]{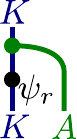} \hspace{-4pt}=\hspace{-4pt} \pic[1.25]{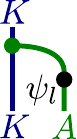} \, , \quad
\pic[1.25]{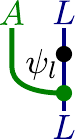} \hspace{-4pt}=\hspace{-4pt} \pic[1.25]{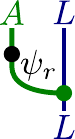} \, , \quad
\pic[1.25]{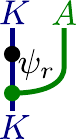} \hspace{-4pt}=\hspace{-4pt} \pic[1.25]{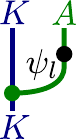} \, .
\end{equation}
hold, the first two of which follow from the associativity of the actions, while for the last two one needs the symmetry property, for example
\begin{equation}
\label{eq:symm_moves_psi}
\pic[1.25]{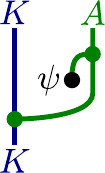} \hspace{-4pt}=\hspace{-4pt}
\pic[1.25]{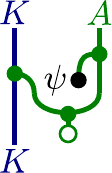} \hspace{-4pt}=\hspace{-4pt}
\pic[1.25]{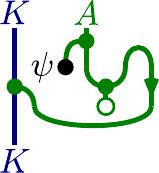} \hspace{-4pt}=\hspace{-4pt}
\pic[1.25]{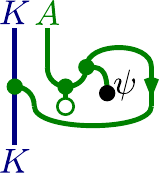} \hspace{-4pt}=\hspace{-4pt}
\pic[1.25]{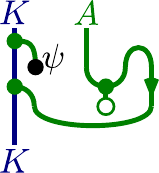} \hspace{-4pt}=\hspace{-4pt}
\pic[1.25]{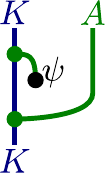} \, .
\end{equation}

For the rest of the section, unless specified otherwise, $\mcC$ will denote a spherical fusion category and $A\in\mcC$ a symmetric separable Frobenius algebra with a section $\psi\colon\opid\ra A$.
As it is the case with separable algebras in multifusion categories in general, the categories of left and right modules ${_A\mcC}$ and $\mcC_A$ (and, for another separable algebra $B\in\mcC$, also the one of bimodules $\mcACB$) are semisimple, see~\cite[Prop.\,2.7]{DMNO}, \cite[Sec.\,4]{KZ2}.
\begin{lem}
\label{lem:tr-psi_non_zero}
For all simple objects $\l\in {}_A\mcC$, $\kappa\in\mcC_A$, $\mu\in\mcACA$, the scalars $\tr_\mcC (\psi_l^\lambda)^2$, $\tr_\mcC (\psi_r^\kappa)^2$, $\tr_\mcC \omega_\mu^2$ are non-zero and independent of the choice of $\psi$.
\end{lem}
\begin{proof}
We show the above claim for a simple left $A$-module $\l$, the proofs of other cases are similar.

For a simple object $i\in\Irr_\mcC$, let $\{ b_p \}$ be a basis of the space $\mcC(i,\l)$ and $\{\overline{b_p}\}$ be the dual basis of $\mcC(\l,i)$ with respect to the composition pairing, i.e.\ $\overline{b_q} \circ b_p = \d_{pq} \cdot \id_i$.
One then has scalars $X_{pq}^{i,\l}$, such that
\begin{equation}
\pic[1.25]{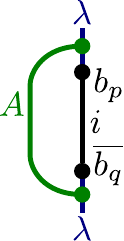} = X_{pq}^{i,\l}
\pic[1.25]{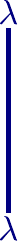} \, ,
\end{equation}
since $\l$ is simple and the morphism on the left-hand side is a left module morphism.
Precomposing both sides with $(\psi_l^\l)^2$ and taking the trace in $\mcC$ one gets:
\begin{equation}
\label{eq:tr_psi_calc}
X_{pq}^{i,\l} \cdot \tr_\mcC(\psi_l^\l)^2 = 
\pic[1.25]{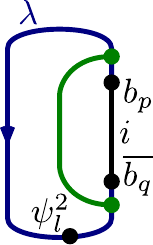} = 
\pic[1.25]{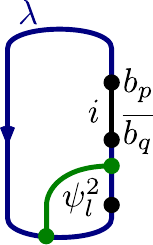} = 
\pic[1.25]{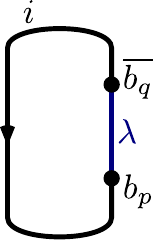} =
d_i \cdot \delta_{pq} \, ,
\end{equation}
where in the second equality one uses the canonical action~\eqref{eq:lambda_dual_action} of $A$ on $\l^*$, the coaction of the form~\eqref{eq:Frob_coaction} and the symmetry property~\eqref{eq:Frob_symmetric}.
Since for $p=q$ and $i$ such that $\mcC(i,\l)\neq\{0\}$ the right-hand side is non-zero, one gets $\tr_\mcC(\psi_l^\l)^2\neq 0$.

To show the independence of $\psi$, note that since $A$ is symmetric one has
\begin{equation}
\pic[1.25]{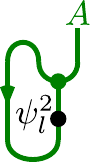} =
\pic[1.25]{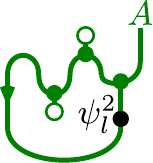} =
\pic[1.25]{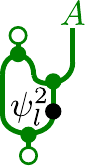} =
\pic[1.25]{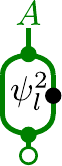} =
\pic[1.25]{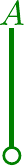} \, .
\end{equation}
Using this, for any other section $\psi'\colon\opid\ra A$ one computes:
\begin{equation}
\label{eq:tr-psi_indep_calc}
\pic[1.25]{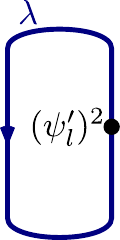} =
\pic[1.25]{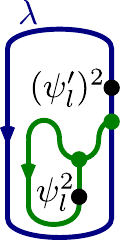} =
\pic[1.25]{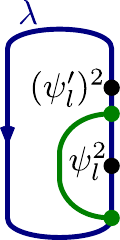} =
\pic[1.25]{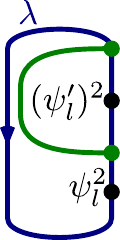} =
\pic[1.25]{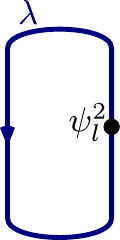} \, ,
\end{equation}
where in the second equality one uses the symmetry property of $A$, the associativity of the action and a combination of the identities~\eqref{eq:psi_comm_identities}.
\end{proof}

We will find that working with a symmetric separable Frobenius algebra $A\in\mcC$ is easier if the section $\psi$ has a certain invertibility property, which one can impose without loss of generality as stated in the following
\begin{prp}
\label{prp:FA_psi_inv}
The section $\psi\colon\opid\ra A$ can be chosen so that it has a multiplicative inverse, i.e.\ a morphism $\psi^{-1}\colon\opid\ra A$ such that
\begin{equation}
\mu \circ (\psi \otimes \psi^{-1}) = \mu \circ (\psi^{-1} \otimes \psi) = \eta \, ,
\end{equation}
where $\mu\colon A\otimes A \ra A$, $\eta\colon\opid\ra A$ are the multiplication and the unit morphisms~of~$A$.
\end{prp}
\begin{proof}
One has an explicit isomorphism of the vector spaces $\mcC(\opid,A)$ and $\mcC_A(A,A)$ sending $[\varphi\colon\opid\ra A]$ to a morphism $[\varphi_l^A:A\ra A]$ defined similarly as in~\eqref{eq:psi_lr_defs}.
The multiplicative inverse of $\varphi$ then corresponds to the inverse of $\varphi_l^A$ with respect to composition.
We can therefore construct an invertible right $A$-module morphism $\psi_l^A \colon A \ra A$ satisfying~\eqref{eq:FA_sep_cond_ito_psi} instead.

Let $A\cong\bigoplus_p \mu_p$ be the decomposition of $A$ into simple $A$-$A$-bimodules provided by the sets of morphisms $\{b_p\colon\mu_p\ra A\}$, $\{\overline{b_p}\colon A \ra \mu_p\}$, i.e.
\begin{equation}
\label{eq:psi-inv_proof_mu_split}
\sum_p b_p\circ\overline{b_p} = \id_A \quad , \qquad
\overline{b_{p'}} \circ b_p = \d_{pp'} \cdot \id_{\mu_p} \, .
\end{equation}
Similarly, for each index $p$, let $\mu_q \cong \bigoplus_q \kappa^p_q$ be the decomposition into simple right $A$-modules provided by the sets of morphisms $\{b^p_q\colon \kappa^p_q\ra \mu_p\}$, $\{\overline{b^p_q}\colon \mu_p \ra \kappa^p_q\}$, i.e.\
\begin{equation}
\label{eq:psi-inv_proof_kappa_split}
\sum_q b^p_q \circ \overline{b^p_q} = \id_{\mu_p} \quad , \qquad
\overline{b^p_{q'}} \circ b^p_q = \d_{qq'} \cdot \id_{\kappa^p_q} \, .
\end{equation}

Let $\psi'\colon \opid\ra A$ be any section (which exists by Proposition~\ref{prp:FA_sep_cond}).
A similar computation as in~\eqref{eq:tr_psi_calc} yields the identity
\begin{equation}
\pic[1.25]{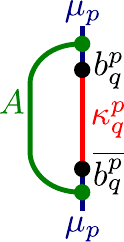} = \frac{\tr_\mcC (\psi'^{\,\kappa^p_q}_r)^2}{\tr_\mcC (\omega'_{\mu_a})^2} \pic[1.25]{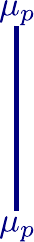} \, ,
\end{equation}
where the factor on the right-hand side is well defined because of Lemma~\ref{lem:tr-psi_non_zero}.
We can now define
\begin{equation}
\psi_l^A := \sum_{p,q} \left( \frac{\tr_\mcC (\psi'^{\,\kappa^p_q}_r)^2}{\tr_\mcC (\omega'_{\mu_a})^2} \right)^{-1/2} \cdot \left(\Lambda^p_q\right)^{1/2} b_p \circ b^p_q \circ \overline{b^p_q} \circ \overline{b_p} \, ,
\end{equation}
where $\Lambda^p_q\in\opk$ are arbitrary non-zero scalars such that $\sum_q \Lambda^p_q = 1$ for each index $p$ (one can for example take $\Lambda^p_q = (\dim \End_{\mcC_A}(\mu_p))^{-1}$ if this dimension is non-zero in $\opk$).
The relations~\eqref{eq:psi-inv_proof_mu_split} and~\eqref{eq:psi-inv_proof_kappa_split} then imply that $\psi_l^A$ is invertible with
\begin{equation}
(\psi_l^A)^{-1} := \sum_{p,q} \left( \frac{\tr_\mcC (\psi'^{\,\kappa^p_q}_r)^2}{\tr_\mcC (\omega'_{\mu_a})^2} \right)^{1/2} \cdot \left(\Lambda^p_q\right)^{-1/2} b_p \circ b^p_q \circ \overline{b^p_q} \circ \overline{b_p} \, .
\end{equation}
Moreover one has
\begin{equation}
\pic[1.25]{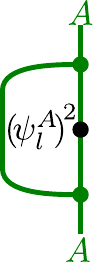} = \sum_{p,q} \left(\frac{\tr_\mcC (\psi'^{\,\kappa^p_q}_r)^2}{\tr_\mcC (\omega'_{\mu_a})^2}\right)^{-1} \hspace{-8pt}\cdot \Lambda^p_q
\pic[1.25]{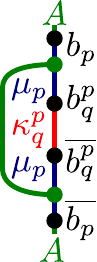} = \sum_{p,q}  \Lambda^p_q
\pic[1.25]{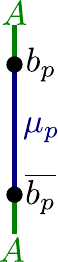} = \sum_p 
\pic[1.25]{1_psi_inv_calc_3.pdf} = 
\pic[1.25]{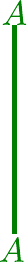} \, ,
\end{equation}
so the condition~\eqref{eq:FA_sep_cond_ito_psi} is satisfied as needed.
\end{proof}

\begin{conv}
Slightly abusing the terminology, we will sometimes call a pair $(A,\psi)$ a symmetric separable Frobenius algebra in a pivotal category $\mcC$, in which case $\psi\colon\opid\ra A$ is understood to be a fixed choice of a section with a multiplicative inverse in the sense of Proposition~\ref{prp:FA_psi_inv} (for $\mcC$ a spherical fusion category this can be assumed without loss of generality).
If $\psi$ is given by the unit $\eta\colon\opid\ra A$, we say that $A$ is \textit{$\D$-separable}.
In this case the section of the product $\mu\colon A \otimes A \ra A$ in $\mcACA$ is given by the coproduct $\Delta:A \ra A \otimes A$.
\end{conv}

Most of the properties of symmetric $\D$-separable Frobenius algebras, as developed and exploited e.g.\ in \cite{FRS1}, apply also to the separable ones. 
For example, for two left modules $L, L' \in {_A\mcC}$ the map
\begin{equation}
\label{eq:A-mod-morphs_avg}
\mcC(L,L') \ni f \mapsto \pic[1.25]{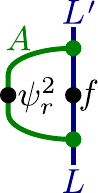} \in {_A\mcC}(L,L')
\end{equation}
acts as a projector onto the subspace ${_A\mcC}(L,L') \subseteq \mcC(L,L')$ of left module morphisms.
Similarly, one defines the projectors onto the subspaces right modules and bimodules.

\begin{example}
Let $A\in\mcC$ be a symmetric Frobenius algebra with the structure morphisms $(\mu,\eta,\D,\vareps)$ which is in addition haploid, i.e.\ $\dim\mcC(\opid,A)=1$.
Then it is separable if and only if one has $\mu\circ\D\neq 0$, in which case the section has to be of the form $\xi\cdot\eta$ for some $\xi\in\opk^\times$.
\end{example}

\begin{example}
\label{eg:XAX_alg}
Let $(A,\psi)$ be a symmetric separable Frobenius algebra in $\mcC$ and $X\in\mcC$ an arbitrary non-zero object.
Then $A_X := X^* \otimes A \otimes X$ is a symmetric Frobenius algebra in $\mcC$ with
\begin{align} \nonumber
&  \text{multiplication: } && \pic[1.25]{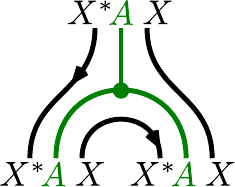} ,
&& \text{comultiplication: } && \pic[1.25]{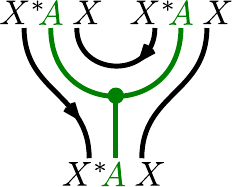},\\ \label{eq:XAX_alg}
&  \text{unit:} && \pic[1.25]{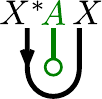} ,
&& \text{counit:} && \pic[1.25]{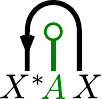} .
\end{align}
It is also separable as evident from picking an endomorphism $\varphi\in\End_\mcC(X)$ such that $\tr_\mcC\varphi^2 = \tr_\mcC(\varphi\circ\varphi) \neq 0$ and setting the section to be
\begin{equation}
\label{eq:XAX_section}
\psi_X = (\tr_\mcC(\varphi\circ\varphi))^{-1/2} \cdot \pic[1.25]{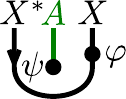} .
\end{equation}
It is also an invertible section if $\varphi$ is invertible.
Later we will encounter this algebra in the case of $\mcC$ being a MFC and $X = C = \bigoplus i$ being the Kirby colour object.
In this case there is a canonical choice of an invertible section $\varphi = d^{1/2}$ since one has $\tr_\mcC d = \Dim\mcC \neq 0$.

Note that in case $A=\opid$ and $\mcC=\Vect_\opk$,  $X^*\otimes X$ is a matrix algebra.
\end{example}

\begin{example}
\label{eg:Euler_completion}
Let $A\in\mcC$ be a symmetric $\D$-separable Frobenius algebra with the structure morphisms $(\mu,\eta,\D,\vareps)$ and $\widetilde{\psi}\colon A \ra A$ an $A$-$A$-bimodule isomorphism.
This yields a symmetric separable Frobenius algebra $(A,\psi)$ having the same underlying object, the structure morphisms $(\widetilde{\psi}^{-1} \circ \mu, \widetilde{\psi} \circ \eta, \widetilde{\psi}^{-1} \circ \D, \widetilde{\psi} \circ \vareps)$ and the section $\psi := \widetilde{\psi}\circ \eta$.

The setting of symmetric $\D$-separable Frobenius algebras was used in the works \cite{CRS3, MR1} whose results we seek to expand in this paper.
In~\cite{CRS3} in particular, a pair like $(A,\widetilde{\psi})$ was used in the construction of surface defects in the Reshetikhin--Turaev TQFT, where the entry $\widetilde{\psi}$ was used in the procedure called Euler completion (see~\cite[Sec.\,2.5]{CRS1}).
We noted in the introduction that symmetric separable Frobenius algebras can be used for construction of surface defects as well.
In fact, one can check that using the algebra $(A,\psi)$ in the present setting is equivalent to using the pair $(A,\widetilde{\psi})$ in the constructions of~\cite{CRS3, MR1}.
Our approach is however more general, as e.g.\ the section~\eqref{eq:XAX_section} is not obtained from a bimodule morphism.
\end{example}

\subsection{Relative tensor products}
\label{subsec:relative_tensor_products}
For an algebra $A$ in a (multi)tensor category $\mcC$, the relative tensor product of a right module $K\in\mcC_A$ and a left module $L\in {}_A\mcC$ is given by the difference cokernel~\cite[Def.\,7.8.21]{EGNO}
\begin{equation}
\label{eq:relative_tensor_prod}
K \otimes_A L := \coker[ K \otimes A \otimes L \xra{\rho\otimes\id_K - \id_L \otimes \l} K \otimes L] \, ,
\end{equation}
where $\rho\colon K\otimes A \ra K$, $\l\colon A \otimes L \ra L$ are the action morphisms.

\medskip

In case $\mcC$ is a pivotal multifusion category and $(A,\psi)$ is a symmetric separable Frobenius algebra in $\mcC$, the relative tensor product can be computed as follows:
Recall that for an object $X\in\mcC$ the image of an idempotent $p\in\End_\mcC X$, $p\circ p = p$, is an object $\im p\in\mcC$ together with projection and inclusion morphisms $\pi\colon X \lra \im p : \imath$ such that $\imath\circ\pi=p$ and $\pi\circ\imath = \id_{\im p}$ ($\pi$ and $\imath$ are then said to \textit{split} the idempotent $p$).
One then has $K\otimes_A L \cong \im P_{K,L}$ where the idempotent $P_{K,L}$ is defined by
\begin{equation}
\label{eq:PKL_idempotent}
P_{K,L} :=
\pic[1.25]{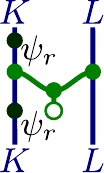} =
\pic[1.25]{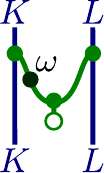} =
\pic[1.25]{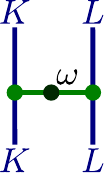} =
\pic[1.25]{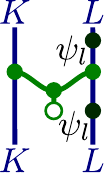}.
\end{equation}
The horizontal line after the third equality can be read as a composition of $A$-coaction~\eqref{eq:Frob_coaction} and $A$-action).
The respective inclusion/projection morphisms will be denoted by
\begin{equation}
\pi = \pic[1.25]{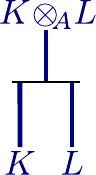} \, , \quad
\imath = \pic[1.25]{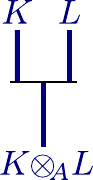} \, .
\end{equation}
To show that one indeed has $\im P_{K,L}\cong K\otimes_A L$ one uses the splitting property to check that the pair $(\im P_{K,L}, q\colon K \otimes L \ra \im P_{K,L})$ where $q := \pi\circ((\psi_r^K)^{-1}\otimes\id_L)$ constitutes a choice of cokernel for the morphism $\rho\otimes\id_K - \id_L \otimes \l$.

Note that the morphisms~\eqref{eq:PKL_idempotent} might fail to be equal if $A$ is not symmetric as this is shown by a computation like~\eqref{eq:symm_moves_psi}.
Still, each of them gives an idempotent which can be used to define the relative tensor product, in this case $\mcC$ also need not be pivotal.

\medskip

For arbitrary objects $X,Y\in\mcC$, when writing down an explicit morphism $f\colon K \otimes_A L \ra X $ or $g\colon Y \ra K \otimes_A L$ we will often find it more convenient to give the morphisms
\pagebreak
\begin{equation}
\label{eq:hatted_morphisms}
\widehat{f} := \pic[1.25]{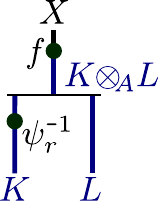} \, , \qquad
\widehat{g} := \pic[1.25]{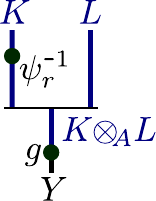} 
\end{equation}
instead, which are defined so that the balancing property holds, i.e.\
\begin{equation}
\pic[1.25]{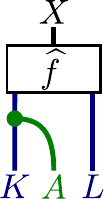} =
\pic[1.25]{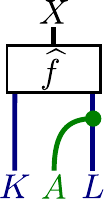} \, , \qquad
\pic[1.25]{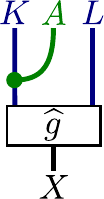} =
\pic[1.25]{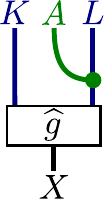} \, .
\end{equation}
Note that the compositions of balanced maps does not necessarily correspond to the composition of maps into/out of the relative tensor product:
\begin{equation}
\label{eq:extra_psi}
\pic[1.25]{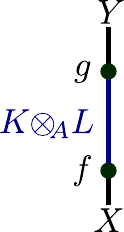} =
\pic[1.25]{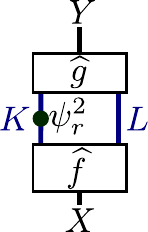}.
\end{equation}
The composition is however preserved properly if $A$ is $\D$-separable.
\begin{conv}
In what follows we will often encounter objects defined as images of idempotents similar to the one in \eqref{eq:PKL_idempotent}.
The corresponding projection/inclusion morphisms will also be denoted by horizontal lines.
The overhats as in \eqref{eq:hatted_morphisms} will be omitted.
If necessary, we will refer to the equation \eqref{eq:extra_psi} to explain the extra $\psi$-insertions due to this computational nuisance.
\end{conv}

For any algebra $A\in\mcC$, the relative tensor product~\eqref{eq:relative_tensor_prod} equips the category $\mcACA$ of bimodules with a monoidal structure with the tensor unit $\opid_{\mcACA} := A$.
If $\mcC$ is pivotal multifusion and $(A,\psi)$ a symmetric separable Frobenius algebra, the monoidal structure is easy to describe in terms of the idempotents as in~\eqref{eq:PKL_idempotent}.
In particular, for $M,N\in\mcACA$, $M\otimes_A N$ has the left/right actions
\begin{equation}
\pic[1.25]{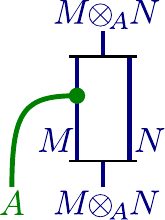} \, \qquad
\pic[1.25]{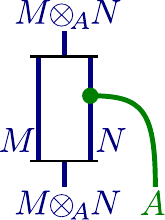} \, ,
\end{equation}
while for $M,N,K\in\mcACA$, the associator, as well as the left/right unitors and their inverses are obtained from the balanced maps
\begin{equation}
\label{eq:ACA_assoc_unitors}
a_{M,N,K} = \hspace{-4pt}\pic[1.25]{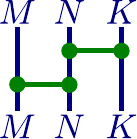}         \hspace{-4pt}, ~
l_M       = \hspace{-4pt}\pic[1.25]{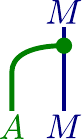}     \hspace{-4pt}, ~
r_M       = \hspace{-4pt}\pic[1.25]{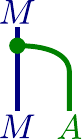}    \hspace{-4pt}, ~
l_M^{-1}  = \hspace{-4pt}\pic[1.25]{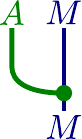} \hspace{-4pt}, ~
r_M^{-1}  = \hspace{-4pt}\pic[1.25]{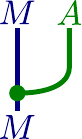}\hspace{-4pt}.
\end{equation}
We emphasise one more time, that obtaining the actual morphisms out of the balanced maps is obtained by using the relation~\eqref{eq:hatted_morphisms}, e.g.\ the associator is given by
\begin{equation}
a_{M,N,K} = 
\pic[1.25]{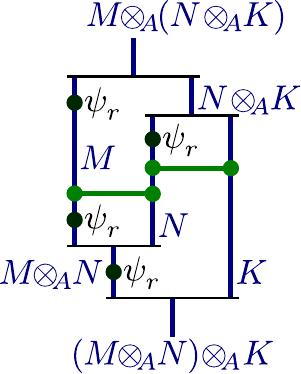} =
\pic[1.25]{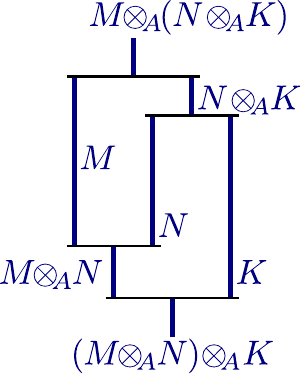} \, .
\end{equation}
Note that when following these definitions closely, $\mcACA$ should not be treated as a strict monoidal category even if $\mcC$ is assumed to be strict to make the use of graphical calculus easier.

\medskip

If $\mcC$ is pivotal multifusion and $(A,\psi)$ is symmetric separable, $\mcACA$ is pivotal multifusion as well, with the dual of an object $M\in\mcACA$ being the dual $M^*$ of the underlying object in $\mcC$ with the actions as in~\eqref{eq:lambda_dual_action} and the (co)evaluation morphisms given by the balanced maps
\begin{equation}
\ev_M    = \hspace{-4pt}\pic[1.25]{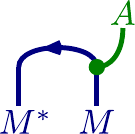}   \hspace{-4pt}, ~
\coev_M  = \hspace{-4pt}\pic[1.25]{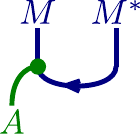} \hspace{-4pt}, ~
\evt_M   = \hspace{-4pt}\pic[1.25]{33_ev-M.pdf}   \hspace{-4pt}, ~
\coevt_M = \hspace{-4pt}\pic[1.25]{33_coev-M.pdf} \hspace{-4pt},
\end{equation}
(note that $A$ indeed needs to be symmetric for the coevaluation maps to be balanced).
As was shown in~\eqref{eq:extra_psi}, the balanced maps do not compose exactly as the maps in $\mcACA$, so the expressions for the left/right traces of $f\in\End_\mcACA(M)$ in $\mcACA$ have additional $\psi^2$-insertions, for example
\begin{equation}
\label{eq:tr_in_ACA}
\tr_l f = \pic[1.25]{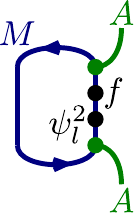} \, , \qquad
\tr_r f = \pic[1.25]{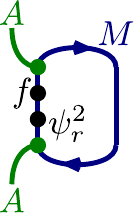} \, .
\end{equation}
We note that in general $\mcACA$ need not be spherical even if $\mcC$ is.
If however $\mcC$ is spherical fusion and $A$ is haploid, $\mcACA$ is spherical as well (the argument for this is used in the proof of Lemma~\ref{lem:ZCAi-ZF_equiv} below).

\subsection{Morita equivalence}
\label{subsec:Morita_eq}
Let $A$, $B$ be two algebras in a monoidal category $\mcC$.
Then the categories of right modules $\mcC_A$, $\mcC_B$ are left $\mcC$-module categories, where the action is $X \triangleright (K,\rho) := (X\otimes K, \id_X \otimes \rho)$ for all $X\in\mcC$ and $K=(K,\rho)\in\mcC_A$ (or $\in\mcC_B$).
The algebras $A$, $B$ are called \textit{Morita equivalent} if $\mcC_A$, $\mcC_B$ are equivalent as left $\mcC$-module categories.
If $\mcC$ is a pivotal tensor category, this is equivalent to the existence of a \textit{Morita module}, i.e.\ an $A$-$B$-bimodule $R\in\mcC$, such that $R^* \otimes_A R \cong B$ as $B$-$B$-bimodules and $R \otimes_B R^* \cong A$ as $A$-$A$-bimodules.
The equivalence $\mcC_A \simeq \mcC_B$ of module categories is then provided by the functor $-\otimes_A R\colon \mcC_A \ra \mcC_B$ with the inverse $-\otimes_B R^*\colon\mcC_B\ra\mcC_A$.

\medskip

Let us now investigate the case when $(A,\psi)$, $(B,\psi')$ are symmetric separable Frobenius algebras in a pivotal fusion category $\mcC$.
The $\mcC$-module categories $\mcC_A$, $\mcC_B$ are then equipped with the following additional structure, see~\cite{Schm}:
\begin{defn}
\label{def:mod_trace}
Let $\mcC$ be a pivotal fusion category.
\begin{itemize}
\item
A \textit{module trace} on a (left) $\mcC$-module category $\mcM$ is a collection of linear maps $\{\Theta_M\colon\End_\mcM(M)\ra\opk\}_{M\in\mcM}$, such that for all $X\in\mcC$, $M,N\in\mcM$, $f\in\mcM(M,N)$, $g\in\mcM(N,M)$ and $h\in\End_\mcM(X\triangleright M)$ one has
\begin{enumerate}[i)]
\item $\Theta_M(g\circ f) = \Theta_N(f\circ g)$,
\item the pairing $(f,g)\mapsto \Theta(g\circ f)$ is non-degenerate,
\item $\Theta_{X\triangleright M}(h) = \Theta_M(\overline{h})$, where $\overline{h}$ is the partial trace
\begin{equation}
\overline{h} := [M \xra{\coevt_X \triangleright \id_M} X^* \triangleright X \triangleright M \xra{\id_{X^*}\triangleright h} X^* \triangleright X \triangleright M \xra{\ev_X \triangleright \id_M} M] \, .
\end{equation}
\end{enumerate}
\item
A module functor $F\colon\mcM\ra\mcN$ between two $\mcC$-module categories with module traces $\mcM=(\mcM,\Theta)$ and $\mcN=(\mcN,\Theta')$ is called \textit{isometric} if one has $\Theta'_{F(M)}(F(f)) = \Theta_M(f)$ for all $M\in\mcM$, $f\in\End_\mcM(M)$.
$\mcM$ and $\mcN$ are called (isometrically) equivalent if there is an isometric module equivalence between them.
\end{itemize}
\end{defn}
\begin{prp}
For $(A,\psi)$ a symmetric separable Frobenius algebra in a spherical fusion category $\mcC$, the collection of maps 
\vspace{-12pt}
\begin{equation}
\label{eq:CA_mod_tr}
\vspace{-14pt}
\{\Theta^\psi_K\colon \End_{\mcC_A}(K)\ra\opk\}_{K\in\mcC_A} \, , \quad
\Theta^\psi_K(f) := \pic[1.25]{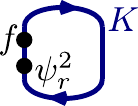}
\end{equation}
is a module trace on $\mcC_A$ as $\mcC$-module category.
\end{prp}
\begin{proof}
The condition i) in Definition~\ref{def:mod_trace} is implied by the cyclicity of the trace in $\mcC$, while the condition iii) is holds since $\mcC$ is spherical.
To check ii) it is enough to check that for a simple object $\kappa\in\mcC_A$ its dimension $\Theta_\kappa(\id_\kappa)$ does not vanish, which holds by Lemma~\ref{lem:tr-psi_non_zero}.
\end{proof}

\begin{rem}
\begin{enumerate}[i), wide, align=left]
\item
It was shown in~\cite[Prop.\,4.4]{Schm} that for an indecomposable $\mcC$-module category $\mcM$, any two module traces $\Theta$, $\Theta'$ are proportional, i.e.\ there is a $\zeta\in\opk^\times$ such that $\Theta' = \zeta\cdot\Theta$.
Similarly, for an indecomposable symmetric separable Frobenius algebra $A\in\mcC$ one can rescale the module trace~\eqref{eq:CA_mod_tr} by rescaling the Frobenius structure of $A$.
In particular, if $(\mu,\eta,\Delta,\vareps)$ are the structure morphisms of $A$, define the Frobenius algebra $A_\zeta$ with the structure morphisms $(\mu,\eta,\zeta^{-1}\cdot\Delta, \zeta\cdot\vareps)$.
Then if $\psi\colon\opid\ra A$ is a section of $A$, $\zeta^{1/2}\cdot\psi$ is a section of $A_{\zeta}$ and the module category with module trace $(\mcC_{A_\zeta}, \Theta^{\zeta^{1/2}\cdot\psi})$ is equivalent to $(\mcC_A, \zeta\cdot\Theta^\psi)$.
Note that by Lemma~\ref{lem:tr-psi_non_zero}, choosing other section for $A_\zeta$ does not change the module trace.
\item
For decomposable module categories/algebras, one can similarly relate the module traces/Frobenius structures by performing a rescaling on each direct summand: if $A = \bigoplus_i A_i$, let $\zeta = \bigoplus_i \zeta_i\cdot \id_{A_i}$ for $\zeta_i\in\opk^\times$ and define the rescaling $A_\zeta$ to have the structure morphisms $(\mu,\eta,(\id\otimes\zeta^{-1})\circ\D, \vareps\circ\zeta)$ and the section $\psi\circ\zeta^{1/2}$.
This can also be generalised to rescalings of $A$ with an arbitrary invertible $A$-$A$-bimodule morphism $\zeta\colon A \ra A$ (see Example~\ref{eg:Euler_completion}), but taking a diagonal $\zeta$ as before is enough for the purpose of changing the module trace of $\mcC_A$ into any other one.
\end{enumerate}
\end{rem}

\begin{defn}
\label{def:isometric_Morita_mod}
We call two symmetric separable Frobenius algebras $(A,\psi)$, $(B,\psi')$ in a spherical fusion category $\mcC$ \textit{Morita equivalent} if $(\mcC_A,\Theta^\psi)$, $(\mcC_B,\Theta^{\psi'})$ are isometrically equivalent and a Morita module ${}_A R_B$ \textit{isometric} if the functor $-\otimes_A R\colon\mcC_A\ra\mcC_B$ is isometric.
\end{defn}

\begin{rem}
Note that it is possible for $(A,\psi)$ and $(B,\psi')$ to be Morita equivalent as algebras, but not as symmetric separable Frobenius algebras - take e.g.\ $A$ and its rescaling $A_\zeta$ as above.
However, as the resulting module traces can be rescaled to each other, in this case it is possible to find $\zeta\in\End_{\mcBCB}(B)$ such that $(A,\psi)$ and $(B_\zeta, \psi'\circ \zeta^{1/2})$ are Morita equivalent.
Similarly, if ${}_A R_B$ is an arbitrary Morita module, ${}_A R_{B_\zeta}$ is an isometric Morita module if the rescaling $\zeta$ is chosen suitably.
\end{rem}

\begin{example}
\label{eg:RX_Morita_mod}
For a symmetric separable Frobenius algebra $(A,\psi)$ in a spherical fusion category $\mcC$, an object $X\in\mcC$, and an invertible morphism $\varphi\in\End_\mcC(X)$ such that $\tr_\mcC(\varphi^2)\neq 0$, let $(A_X,\psi_X)$ be the algebra in Example~\ref{eg:XAX_alg}.
Then the object $R_X := A \otimes X$ can be equipped with an $A$-$A_X$-bimodule structure and $A$-$A$- bimodule maps $R_X \otimes_{A_X} R_X^* \lra A$ defined by
\begin{equation}
\label{eq:RX_morphs}
\pic[1.25]{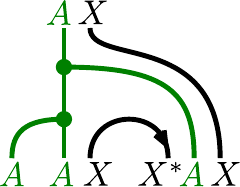} \quad , \quad
\pic[1.25]{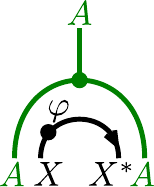} \quad , \quad
\pic[1.25]{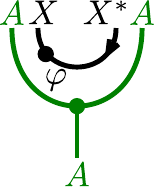},
\end{equation}
where due to $A$ being symmetric one has $R_X^* \cong X^* \otimes A^* \cong X^* \otimes A$.
The bimodule $R_X$ is evidently a Morita module as one has
\begin{equation}
R_X \otimes_{A_X} R_X^* \cong \im
\pic[1.25]{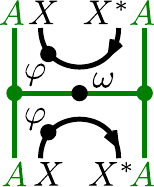} \cong
A \, , \quad
R_X^* \otimes_A R_X \cong X^* \otimes (A \otimes_A A) \otimes X \cong A_X \, ,
\end{equation}
where the first isomorphism is given explicitly by the maps in~\eqref{eq:RX_morphs}.
$R_X$ is also isometric since for an arbitrary simple module $\kappa\in\mcC_A$ one has $\kappa \otimes_A R_X \cong \kappa \otimes X$ and
\begin{equation}
\Theta^{\psi_X}_{\kappa \otimes_A R_X} (\id_{\kappa\otimes_A R_X}) =
\frac{1}{\tr_\mcC(\varphi^2)}\cdot
\tr_\mcC (\psi_r^\kappa)^2 \cdot \tr_\mcC(\varphi^2) = \tr_\mcC(\psi_r^\kappa)^2 = \Theta^\psi_{\kappa}(\id_\kappa) \, .
\end{equation}
\end{example}

In what follows we will find the following identity useful for computations:
\begin{lem}
\label{lem:tr-R_id}
Let $(A,\psi)$, $(B,\psi')$ be symmetric separable Frobenius algebras in a spherical fusion category $\mcC$.
Then a Morita module ${}_A R_B$ is isometric if and only if one has
\begin{equation}
\label{eq:tr-R_id}
\pic[1.25]{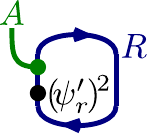} =
\pic[1.25]{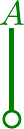} \, , \quad
\pic[1.25]{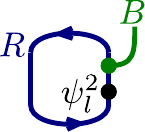} =
\pic[1.25]{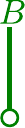} \, .
\end{equation}
\end{lem}
\begin{proof}
We focus on the first identity only.
Let $\kappa\in\mcC_A$ be an arbitrary simple right $A$-module.
Then one has a scalar $X_\kappa$, such that
\begin{equation}
\label{eq:tr-R_id_calc}
\pic[1.25]{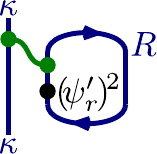} \hspace{-4pt}= X_\kappa \cdot
\pic[1.25]{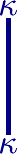} ~\Lra~
\pic[1.25]{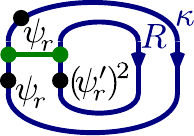} \hspace{-4pt}=
\pic[1.25]{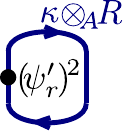} \hspace{-8pt}= X_\kappa \cdot
\pic[1.25]{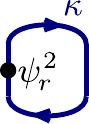} ,
\end{equation}
where the first equality holds because the morphism on the left-hand side is a morphism of right $A$-modules and the implied identity was obtained by pre- and postcomposing both sides with $\psi_r^\kappa$, taking the trace in $\mcC$ and noting that the action/coaction pair together with the $\psi_r^\kappa$ insertions compose into the idempotent $P_{\kappa,R}$ as in~\eqref{eq:PKL_idempotent} projecting onto the relative tensor product $\kappa \otimes_A R$.

If~\eqref{eq:tr-R_id} holds, one has $X_\kappa = 1$ so that the last equality implies that $-\otimes_A R$ preserves the module trace~\eqref{eq:CA_mod_tr} and ${_AR_B}$ is therefore isometric.

Conversely, if ${_AR_B}$ is isometric, the last equality of~\eqref{eq:tr-R_id_calc} again implies $X_\kappa=1$.
To show~\eqref{eq:tr-R_id}, one decomposes $A$ into simple right $A$-modules and precomposes both sides of the first equality in~\eqref{eq:tr-R_id_calc} with the unit $\eta\colon\opid\ra A$ for each direct summand.
\end{proof}

In the remainder of the section we look at how the Morita equivalence of two symmetric separable Frobenius algebras $(A,\psi)$, $(B,\psi')$ in a spherical fusion category $\mcC$ translates to the equivalence of the categories of bimodules $\mcACA$ and $\mcBCB$.
To this end, let ${_AR_B}$ be a (not necessarily isometric) Morita module and fix a $B$-$B$-bimodule isomorphism $\zeta\colon B \ra B$, such that ${_AR_{B_\zeta}}$ is isometric.
We define an isomorphism of $B$-$B$-bimodules $B\xra{\sim}R^*\otimes_A R$ and its inverse by the balanced maps
\begin{equation}
R_0      = \pic[1.25]{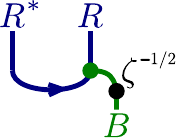} \, , \quad
R_0^{-1} = \pic[1.25]{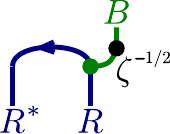} \, .
\end{equation}
That they are indeed inverses follows from the computation
\begin{equation}
\pic[1.25]{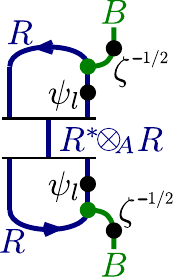} =
\pic[1.25]{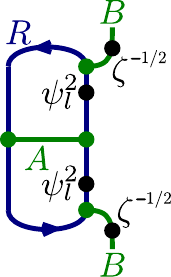} =
\pic[1.25]{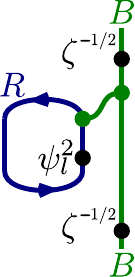} =
\pic[1.25]{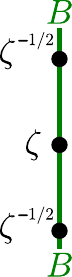} =
\pic[1.25]{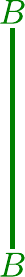} \, ,
\end{equation}
where in the third equality we have used the second identity in~\eqref{eq:tr-R_id} for $A$ and $B_\zeta$.
Since $R^* \otimes_A R$ and $B$ are isomorphic objects in $\mcBCB$ and $\mcBCB$ is finitely semisimple, this is enough for $R_0^{-1}$ to be the two-sided inverse.
Similarly, one defines a natural isomorphism $R_\otimes\colon (-\otimes_A R) \otimes_B (R^* \otimes_A -) \Ra  - \otimes_A -$ by setting for all $K\in\mcC_A$ and $L\in{_A\mcC}$ the balanced maps
\pagebreak
\begin{equation}
\label{eq:Rtensor_morphs}
R_\otimes(K,L)      = \pic[1.25]{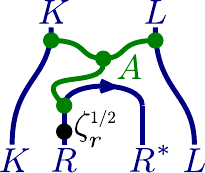} \, , \quad
R^{-1}_\otimes(K,L) = \pic[1.25]{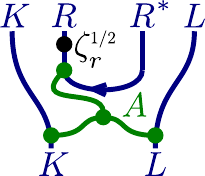} \, .
\end{equation}

\medskip

Recall from Section~\ref{subsec:relative_tensor_products} that $\mcACA$ and $\mcBCB$ are pivotal (multi)fusion categories.
By the functor $R\colon\mcACA\ra\mcBCB$ we will mean the functor $R^* \otimes_A - \otimes_A R$.
One has:
\begin{prp}
\label{prp:R_pivotal}
$R=(R,R_2,R_0)$, where $R_2(-,-) := R_\otimes(R^* \otimes_A -, - \otimes_A R)$, is a monoidal equivalence $\mcACA\xra{\sim}\mcBCB$.
If the Morita module ${_AR_B}$ is in addition isometric (i.e.\ $\zeta = \id_B$), $R$ is a pivotal equivalence.
\end{prp}
\begin{proof}
Checking the coherence identities for $(R,R_2,R_0)$ is a straightforward exercise, whose only difficulty is the occurrence of the non trivial associator/unitor morphisms in $\mcACA$ and $\mcBCB$, which are by definition obtained from the balanced maps~\eqref{eq:ACA_assoc_unitors}.

Let us check the pivotality of $R$, a priori without assuming ${_AR_B}$ to be isometric.
For $M\in\mcACA$, we denote by $\pi_M\colon R(M) \lra R^* \otimes M \otimes R :\imath_M$ the idempotent splitting morphisms in $\mcBCB$.
The four morphisms $R(\delta_M)$, $R_1(M^*)$, $\delta_{R(M)}$, $R_1^*(M)^*$ in the diagram~\ref{eq:F-pivotal} (where we set $X=M$ and $F=R$) are given by pre- and postcomposing
\begin{equation}
\label{eq:R-pivotal_aux_morphs}
\pic[1.25]{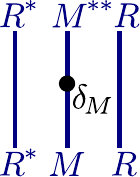}        \, , ~
\pic[1.25]{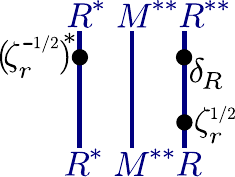}      \, , ~
\pic[1.25]{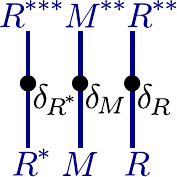}       \, , ~
\pic[1.25]{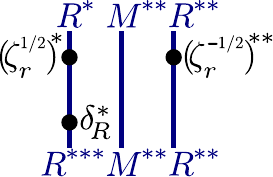}
\end{equation}
respectively with
$\imath_M$ and $\pi_{M^{**}}$,
$\imath_{M^{**}}$ and $\imath_{M^*}^*$,
$\imath_M$ and $\pi_M^{**}$,
$\imath_M^{**}$ and $\imath_{M^*}^*$.
By symmetry~\eqref{eq:Frob_symmetric} and the action~\eqref{eq:lambda_dual_action} on dual objects, the compositions $\imath_{M^{**}}\circ\pi_{M^{**}}$ and $\imath_M^{**}\circ\pi_{M}^{**}$ in particular yield the idempotents projecting onto the relative tensor products $R^* \otimes_A M^{**} \otimes_A R$ and $R^{***} \otimes_A M^{**} \otimes_A R^{**}$.
Since the action of $A$ commutes with the pivotal structure, upon composing $R_1(M^*) \circ R(\delta_M)$ and $R_1^*(M) \circ \delta_{R(M)}$ one can absorb these idempotents into $\imath_M$, so that they correspond to stacking the first-second and third-fourth diagrams in~\eqref{eq:R-pivotal_aux_morphs} and pre- and postcomposing with $\imath_M$ and $\imath_{M^*}^*$.
We see that if ${_AR_B}$ is isometric, i.e.\ $\zeta=\id_B$, the two compositions are equal since one has $\delta_R^* = \delta_{R^*}^{-1}$.

That $R$ is an equivalence follows from it being a Morita module, the inverse is given by $R^{-1} := R \otimes_B - \otimes_B R^*$.
\end{proof}

\begin{rem}
\label{rem:R_pivotal_nuisances}
Note that in Proposition~\ref{prp:R_pivotal} we do not state the converse: it is possible for the functor $R$ to be pivotal even if ${_AR_B}$ is not isometric.
As an example, suppose $\zeta$ is given by a global scaling, i.e.\ is proportional to $\id_A$.
Then the $\zeta^{\pm 1/2}$ factors in~\eqref{eq:R-pivotal_aux_morphs} cancel and so $R$ is pivotal by the same argument.
This can also be seen from the expression~\eqref{eq:tr_in_ACA} for the categorical traces: they remain invariant upon changing $A$ with its rescaling $A_\zeta$, since the additional $\zeta$ factor due to changing the $\psi^2$-insertion is compensated by the factor due to changing the coproduct (and therefore the coaction~\eqref{eq:Frob_coaction}).
In particular, for indecomposable and Morita equivalent symmetric separable Frobenius algebras $A, B\in\mcC$, the pivotal categories $\mcACA$ and $\mcBCB$ do not depend on the module trace.

In general however $R$ is not a pivotal functor.
Take for example an algebra of the form $A=\bigoplus_i A_i$ where each $A_i = (A_i,\psi_i)$ is an indecomposable symmetric separable Frobenius algebra in $\mcC$.
An arbitrary object $M\in\mcACA$ has the decomposition $\bigoplus_{ij} {_iM_j}$ where ${_iM_j}\in{_{A_i} \mcC_{A_j}}$.
In particular, a simple object is of the form $\mu={_i\mu_j}$ and its right categorical dimension in $\mcACA$ is a morphism $\dim_\mu = \imath_i \circ \tr_r\id_{\mu} \circ \pi_i \in \End_{\mcACA}(A)$ where $\pi_i\colon A \lra A_i :\imath_i$ are the projection/inclusion morphisms with respect to the decomposition of $A$ and $\tr_r \id_{\mu}\in \End_{_{A_i}\mcC_{A_i}}(A_i)$ is as in~\eqref{eq:tr_in_ACA}, but with the $\psi_j^2$-insertion inside the $\mu$-loop instead.
Scaling $A$ with $\zeta = \bigoplus \zeta_i \cdot \id_{A_i}$, $\zeta_i\in\opk^\times$, then changes $\dim_\mu$ into $\zeta_i^{-1}\zeta_j \dim_\mu$.
Since in general this delivers a different set of categorical dimensions of simple objects, the categories $\mcACA$ and ${_{A_\zeta}\mcC_{A_\zeta}}$ are not pivotal-equivalent.
\end{rem}

\section{Monoidal Frobenius functors}
\label{sec:monoidal_FFs}
The proofs of the main results of this paper will be based on mapping a Frobenius algebra $A\in\mcC$ to another Frobenius algebra $F(A)\in\mcD$ across a functor $F\colon\mcC \ra \mcD$ between monoidal categories.
It turns out that for this purpose one does not need $F$ to be monoidal, a weaker structure of a \textit{monoidal Frobenius} functor on $F$ is enough and in fact provides more interesting examples of such mappings.
We list the definitions related to and general properties of monoidal Frobenius functors in Section~\ref{subsec:FF_generalities}, describe how they preserve Frobenius algebras in Section~\ref{subsec:FF_preservation_of_FAs} and list some examples in Section~\ref{subsec:FF_examples}.

The notion of a monoidal Frobenius functor was introduced and explored in~\cite{Sz1, Sz2, DP, MS}, here we adapt it also to pivotal and ribbon categories.

\subsection{Generalities}
\label{subsec:FF_generalities}
Let $(\mcC,\otimes,\opid)$, $(\mcD,\otimes',\opid')$ be monoidal categories.
One can make an analogy between a (not necessarily strong) monoidal structure on a functor $F\colon\mcC\ra\mcD$ and an algebra in a monoidal category, according to which the natural transformation $F_2\colon F(-)\otimes F(-)\Ra F(-\otimes -)$ corresponds to the multiplication, the morphism $F_0\colon\opid'\ra F(\opid)$ to the unit and the coherence identities~\eqref{eq:monoidal_funct_ids} to associativity and left/right unitality.
Similarly, one has an analogy between (weak) comonoidal structures and coalgebras.
Extending this analogy even further, the following structure on $F$ is designed to correspond to a Frobenius algebra.
\begin{defn}
\label{def:FF}
A \textit{monoidal Frobenius functor} from $\mcC$ to $\mcD$ is a functor $F\colon\mcC\ra\mcD$ equipped with a weak monoidal structure
\begin{equation}
F_2\colon F(-) \otimes' F(-) \Ra F(- \otimes -)\, , \quad F_0\colon \opid' \ra F(\opid)
\end{equation}
and a weak comonoidal structure
\begin{equation}
\overline{F_2}\colon F(- \otimes -) \Ra F(-) \otimes' F(-) \, , \quad \overline{F_0} \colon F(\opid) \ra \opid'
\end{equation}
such that the following diagrams commute for all $X,Y,Z\in\mcC$
\begin{equation}
\label{eq:FF_Frob1}
\begin{tikzpicture}[baseline={([yshift=-.5ex]current bounding box.center)}]
\node (F1) at (-8,1.5)    {$F(X\otimes Y) \otimes' F(Z)$};
\node (F2) at  (0,1.5)    {$F((X\otimes Y) \otimes Z)$};
\node (F3) at  (-8,0)   {$(F(X) \otimes' F(Y)) \otimes' F(Z)$};
\node (F4) at  (0,0) {$F(X\otimes (Y \otimes Z))$};
\node (F5) at  (-8,-1.5) {$F(X) \otimes' (F(Y) \otimes' F(Z))$};
\node (F6) at  (0,-1.5) {$F(X) \otimes' F(Y \otimes Z)$};

\path[commutative diagrams/.cd,every arrow,every label]
(F1) edge node {$F_2(X \otimes Y, Z)$} (F2)
(F2) edge node {$F(a_{X,Y,Z})$} (F4)
(F3) edge node[swap] {$a'_{F(X),F(Y),F(Z)}$} (F5)
(F4) edge node {$\overline{F_2}(X,  Y \otimes Z)$} (F6)
(F1) edge node[swap] {$\overline{F_2}(X,Y) \otimes' \id_{F(Z)}$} (F3)
(F5) edge node {$\id_{F(X)} \otimes' F_2(Y,Z)$} (F6);
\end{tikzpicture} \, ,
\end{equation}
\begin{equation}
\label{eq:FF_Frob2}
\begin{tikzpicture}[baseline={([yshift=-.5ex]current bounding box.center)}]
\node (F1) at (-8,1.5)    {$F(X)\otimes' F(Y \otimes Z)$};
\node (F2) at  (0,1.5)    {$F(X\otimes (Y \otimes Z))$};
\node (F3) at  (-8,0)   {$F(X) \otimes' (F(Y) \otimes' F(Z))$};
\node (F4) at  (0,0) {$F((X \otimes Y) \otimes Z)$};
\node (F5) at  (-8,-1.5) {$(F(X) \otimes' F(Y)) \otimes' F(Z)$};
\node (F6) at  (0,-1.5) {$F(X \otimes Y) \otimes' F(Z)$};

\path[commutative diagrams/.cd,every arrow,every label]
(F1) edge node {$F_2(X, Y \otimes Z)$} (F2)
(F2) edge node {$F(a^{-1}_{X,Y,Z})$} (F4)
(F4) edge node {$\overline{F_2}(X \otimes Y,  Z)$} (F6)
(F1) edge node[swap] {$\id_{F(X)} \otimes' \overline{F_2}(Y,Z)$} (F3)
(F3) edge node[swap] {$a'^{-1}_{F(X),F(Y),F(Z)}$} (F5)
(F5) edge node {$F_2(X,Y) \otimes' \id_{F(Z)}$} (F6);
\end{tikzpicture} \, .
\end{equation}
\end{defn}

\begin{figure}
	\captionsetup{format=plain}
	\centering
    \begin{subfigure}[b]{1.0\textwidth}
    \begin{equation*}
        F(f) \hspace{-2pt}=\hspace{-8pt}\pic[1.25]{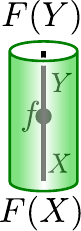} \hspace{-4pt},~
		F_2(X,Y) \hspace{-2pt}=\hspace{-16pt} \pic[1.25]{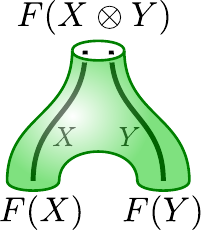} \hspace{-12pt},~
		F_0 \hspace{-2pt}=\hspace{-8pt} \pic[1.25]{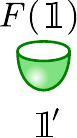} \hspace{-4pt},~
		\overline{F_2}(X,Y) \hspace{-2pt}=\hspace{-16pt} \pic[1.25]{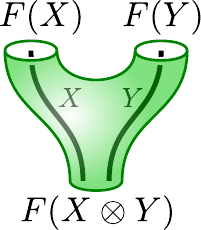} \hspace{-12pt},~
		\overline{F_0} \hspace{-2pt}=\hspace{-8pt} \pic[1.25]{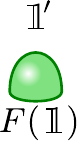} \hspace{-4pt},            
    \end{equation*}
	\end{subfigure}\\
	\vspace{-8pt}
	\begin{subfigure}[b]{1.0\textwidth}
    \begin{equation}
    \label{eq:F_graph_calc:assoc-unitality}
    \tag{F1}
            \pic[1.25]{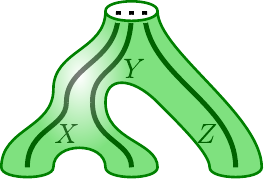} \hspace{-8pt}=\hspace{-8pt}
		    \pic[1.25]{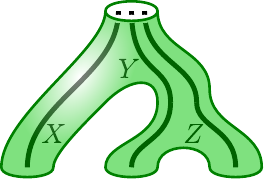} \hspace{-8pt},
		    \pic[1.25]{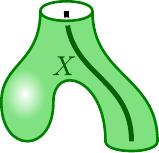} \hspace{-8pt}=\hspace{-8pt}
		    \pic[1.25]{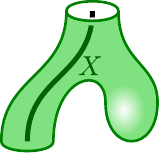} \hspace{-8pt}=\hspace{-4pt}
		    \pic[1.25]{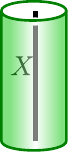}
    \end{equation}
	\end{subfigure}\\
	\vspace{-8pt}
	\begin{subfigure}[b]{1.0\textwidth}
	\begin{equation}
	\label{eq:F_graph_calc:coassoc-counitality}
	\tag{F2}
	    \pic[1.25]{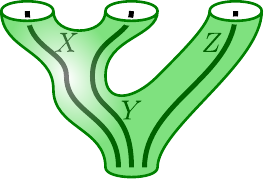} \hspace{-8pt}=\hspace{-8pt}
		\pic[1.25]{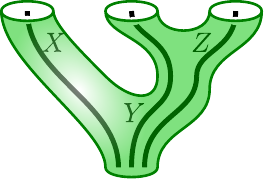} \hspace{-8pt},
		\pic[1.25]{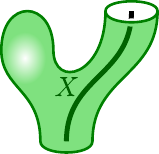} \hspace{-8pt}=\hspace{-8pt}
		\pic[1.25]{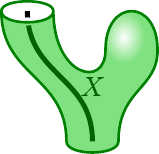} \hspace{-8pt}=\hspace{-4pt}
		\pic[1.25]{41_F_graph_calc_id.pdf}
	\end{equation}
	\end{subfigure}\\
	\vspace{-8pt}
	\begin{subfigure}[b]{1.0\textwidth}
	\begin{equation}
	\label{eq:F_graph_calc:Frob}
	\tag{F3}
	    \pic[1.25]{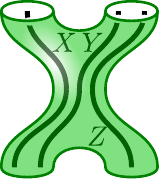} \hspace{-8pt}=\hspace{-2pt}
		\pic[1.25]{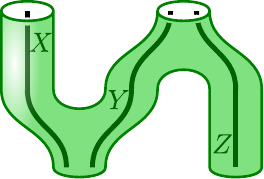} ,
		\pic[1.25]{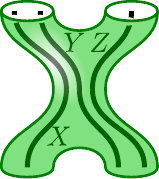} \hspace{-8pt}=\hspace{-2pt}
		\pic[1.25]{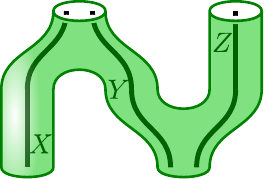}
	\end{equation}
	\end{subfigure}\\
	\vspace{-8pt}
    \begin{subfigure}[b]{0.45\textwidth}
	\begin{equation}
	\label{eq:F_graph_calc:braided}
	\tag{F4}
	    \pic[1.25]{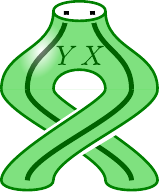} \hspace{-6pt}=\hspace{-6pt}
		\pic[1.25]{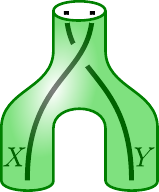}
	\end{equation}
	\end{subfigure}
	\begin{subfigure}[b]{0.45\textwidth}
	\begin{equation}
	\label{eq:F_graph_calc:cobraided}
	\tag{F4'}
		\pic[1.25]{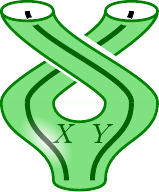} \hspace{-6pt}=\hspace{-6pt}
		\pic[1.25]{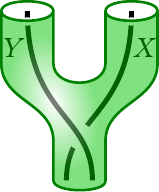}
	\end{equation}
	\end{subfigure}\\
	\vspace{-8pt}
	\begin{subfigure}[b]{0.45\textwidth}
	\begin{equation}
	\label{eq:F_graph_calc:twist}
	\tag{F5}
	    \pic[1.25]{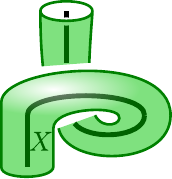} =
		\pic[1.25]{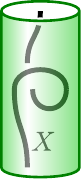}
	\end{equation}
	\end{subfigure}
	\caption{
	Graphical calculus for a monoidal Frobenius functor $F\colon\mcC\ra\mcD$.
	The identities~\eqref{eq:F_graph_calc:assoc-unitality}--\eqref{eq:F_graph_calc:Frob} correspond to the coherence of the (co)monoidal structure and the commutative diagrams~\eqref{eq:FF_Frob1}, \eqref{eq:FF_Frob2}.
	If $\mcC$ and $\mcD$ are braided, \eqref{eq:F_graph_calc:braided} (resp.\ \eqref{eq:F_graph_calc:cobraided}) is the identity~\eqref{eq:F_braided} (resp.\ \eqref{eq:F_cobraided}) which holds if $F$ is in addition braided (resp.\ cobraided).
	If $\mcC$ and $\mcD$ are ribbon, \eqref{eq:F_graph_calc:twist} is the identity~\eqref{eq:F_twist} in case $F$ is in addition ribbon.
	For a ribbon Frobenius functor~\eqref{eq:F_graph_calc:braided} and~\eqref{eq:F_graph_calc:cobraided} imply each other.
	}
	\label{fig:F_graph_calc}
\end{figure}
To make the properties of a monoidal Frobenius functor $F\colon\mcC\ra\mcD$ more obvious, for a morphism $[f\colon X \ra Y]\in\mcC$ we will display the morphism $F(f)\in\mcD$ in the string diagrams of $\mcD$ by a shaded cylindrical tube containing the string diagram in $\mcC$ corresponding to $f$, see Figure~\ref{fig:F_graph_calc}.
The functoriality of $F$ allows one to perform the graphical calculus of $\mcC$ inside such tubes\footnote{Admittedly, since the graphical calculus for monoidal categories is `planar', it would be more appropriate to use flat shaded regions instead of tubes in such notation.
Our goal is however to eventually use Frobenius functors for ribbon categories, for which the tube notation will be more appropriate.}.
The weak (co)monoidal structure morphisms $F_0$, $\overline{F_0}$ will be depicted by cups/caps, which open/close an empty tube, while for $X,Y\in\mcC$ the morphisms $F_2(X,Y)$, $\overline{F_2}(X,Y)$ will be depicted by merging/splitting of two tubes.
The (co)monoidality of $F$ then translates to the identities~\eqref{eq:F_graph_calc:assoc-unitality}, \eqref{eq:F_graph_calc:coassoc-counitality} in Figure~\ref{fig:F_graph_calc} while the commutative diagrams~\eqref{eq:FF_Frob1} and~\eqref{eq:FF_Frob2} to the ones in~\eqref{eq:F_graph_calc:Frob}.

\medskip

As expected, the conditions~\eqref{eq:F_graph_calc:assoc-unitality} and~\eqref{eq:F_graph_calc:coassoc-counitality} resemble the associativity/unitality \eqref{eq:alg_assoc-unitality} and coassociativity/counitality~\eqref{eq:alg_coassoc-counitality} identities of an algebra in a monoidal category, while~\eqref{eq:F_graph_calc:Frob} the Frobenius property~\eqref{eq:Frob_cond}.
Further developing the analogy with Frobenius algebras, the following property of monoidal Frobenius functors is formulated to correspond to separability
\begin{defn}
\label{def:FF_sep_cond}
A monoidal Frobenius functor $F\colon\mcC\ra\mcD$ is called \textit{separable} if there exists a morphism $\psi_F\colon \opid' \ra F(\opid)$ such that for all $X,Y\in\mcC$ one has
\begin{equation}
\label{eq:FF_sep_cond}
\pic[1.25]{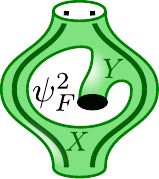} =
\pic[1.25]{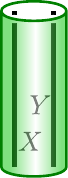} \, ,
\end{equation}
where $\psi^2_F = F_2(\opid,\opid)\circ (\psi_F \otimes \psi_F)$.
We call such morphism $\psi_F$ a \textit{section of $F$}.
\end{defn}
We note that the use of the squared section $\psi^2_F$ in~\eqref{eq:FF_sep_cond} is just a convention aimed at resembling our preferred condition~\eqref{eq:FA_sep_cond_ito_psi} for separable Frobenius algebras.
In principle the separability of $F$ can also be defined in terms of a morphism $\zeta_F\colon\opid'\ra F(\opid)$ by analogy with~\eqref{eq:FA_sep_cond}; the requirement that $\zeta_F = \psi_F^2$ for some section $\psi_F$ can be relaxed.
We also note that the separability condition~\eqref{eq:FF_sep_cond} is more general than the one usually found in the literature (see e.g.\ \cite{MS}), which corresponds to having $\psi_F = F_0$ (i.e.\ analogous to the condition for a Frobenius algebra to be $\D$-separable).

\medskip

It was noted in~\cite{DP} that a monoidal Frobenius functor $F\colon\mcC\ra\mcD$ preserves dualities.
Indeed, if an object $X\in\mcC$ has a left dual $X^*$ with evaluation/coevaluation morphisms $\ev_X\colon X^* \otimes X \ra \opid$, $\coev_X\colon\opid \ra X \otimes X^*$, one has the following pairing/copairing morphisms between $F(X)$ and $F(X^*)$:
\begin{equation}
\pic[1.25]{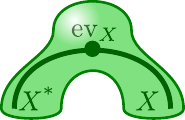} \, , \quad
\pic[1.25]{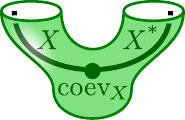} \, ,
\end{equation}
which are inverses of each other (by which we mean the zig-zag identities hold, see~\cite[Sec.\,1.5.1]{TV}) as shown by the following computation
\begin{equation}
\pic[1.25]{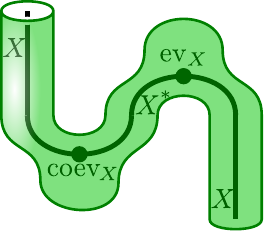} =
\pic[1.25]{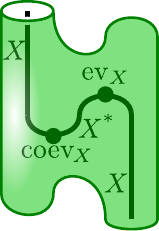} =
\pic[1.25]{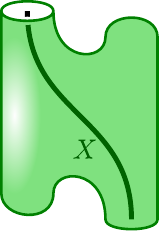} =
\pic[1.25]{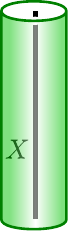}
\end{equation}
(similarly one shows the other zig-zag identity).
In particular, if $\mcC$, $\mcD$ are left rigid, there exists a unique family of isomorphisms $\{F_1(X) \colon F(X^*) \xra{\sim} F(X)^* \}_{X\in\mcC}$ such that
\begin{equation}
\label{eq:Frob_F1_morphs}
\pic[1.25]{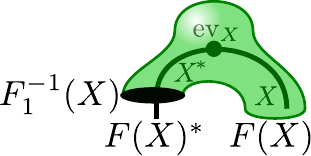} \hspace{-4pt}=\hspace{-4pt}
\pic[1.25]{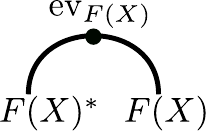} \hspace{-4pt},
\pic[1.25]{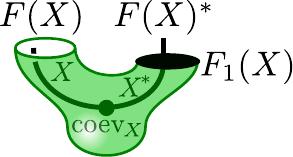} \hspace{-4pt}=\hspace{-4pt}
\pic[1.25]{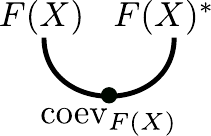} \hspace{-4pt}.
\end{equation}
\begin{defn}
\label{def:pivotal_braided_ribb_FFs}
Let $\mcC$, $\mcD$ be monoidal categories and $F\colon\mcC\ra\mcD$ a monoidal Frobenius functor.
\begin{itemize}
\item
If $\mcC = (\mcC,\d)$, $\mcD = (\mcD,\d')$ are pivotal, then $F$ is said to be \textit{pivotal} if it preserves the pivotal structure, i.e.\ if for all $X\in\mcC$ diagram~\eqref{eq:F-pivotal} commutes.
\item
If $\mcC = (\mcC,c)$, $\mcD = (\mcD,c')$ are braided, then $F$ is called \textit{braided} if $F_2$ preserves the braidings, i.e.\ for all $X,Y\in\mcC$ one has
\begin{equation}
\label{eq:F_braided}
F_2(Y,X) \circ c'_{F(X),F(Y)} = F(c_{X,Y}) \circ F_2(X,Y) \, .
\end{equation}
Similarly, $F$ is called \textit{cobraided} if $\overline{F_2}$ preserves the braidings, i.e.\ one has
\begin{equation}
\label{eq:F_cobraided}
c'_{F(X),F(Y)}\circ\overline{F_2}(X,Y) = \overline{F_2}(Y,X)\circ F(c_{X,Y}) \, .
\end{equation}
\item
If $\mcC$, $\mcD$ are ribbon with the twist morphisms $\{\theta_X\}_{X\in\mcC}$, $\{\theta'_{X'}\}_{X'\in\mcD}$, then $F$ is called \textit{ribbon} if it is braided and preserves the twists, i.e.\ for all $X\in\mcC$ one has
\begin{equation}
\label{eq:F_twist}
\theta'_{F(X)} = F(\theta_X) \, .
\end{equation}
\end{itemize}
\end{defn}
Like in the case of strong monoidal functors, one shows that if $F$ is pivotal then the same family of morphisms $\{F_1(X)\}_{X\in\mcC}$ preserves the right duals:
\begin{equation}
\label{eq:Frob_F1_morphs_pivotal}
\pic[1.25]{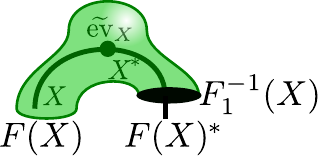} \hspace{-4pt}=\hspace{-4pt}
\pic[1.25]{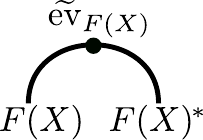} \hspace{-4pt},
\pic[1.25]{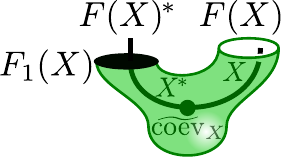} \hspace{-4pt}=\hspace{-4pt}
\pic[1.25]{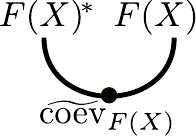} \hspace{-4pt}.
\end{equation}

Graphically the identities \eqref{eq:F_braided}, \eqref{eq:F_cobraided} and \eqref{eq:F_twist} are depicted respectively by the equations~\eqref{eq:F_graph_calc:cobraided}, \eqref{eq:F_graph_calc:cobraided} and~\eqref{eq:F_graph_calc:twist} in Figure~\ref{fig:F_graph_calc}.
For a ribbon Frobenius functor $F\colon\mcC\ra\mcD$, \eqref{eq:F_graph_calc:braided} and~\eqref{eq:F_graph_calc:twist} hold by definition, while~\eqref{eq:F_graph_calc:cobraided} holds as a consequence of the following
\begin{prp}
\label{prp:FF_ribb-piv-cobr}
A braided Frobenius functor $F\colon\mcC\ra\mcD$ between ribbon categories $\mcC$, $\mcD$ is ribbon if and only if it is pivotal and cobraided.
\end{prp}
\begin{proof}
Given that $F$ is ribbon, one shows that it is pivotal by adapting a similar argument for strong monoidal ribbon functors.
For an arbitrary $X\in\mcC$, one has by definition
\begin{equation}
\theta'_{F(X)} =
\pic[1.25]{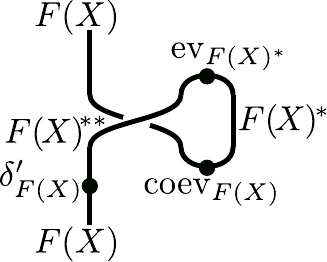} \quad\Ra\quad
\delta'_{F(X)} =
\pic[1.25]{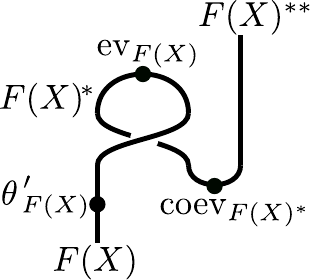} \, .
\end{equation}
Using the canonical isomorphism $F_1(X)\colon F(X^*) \ra F(X)^*$ as in~\eqref{eq:Frob_F1_morphs} (which along with its inverse is not always labelled explicitly in this proof) as well as~\eqref{eq:F_graph_calc:braided}, \eqref{eq:F_graph_calc:twist} one gets:
\begin{align} \nonumber
&\delta'_{F(X)} = \hspace{-4pt}
\pic[1.25]{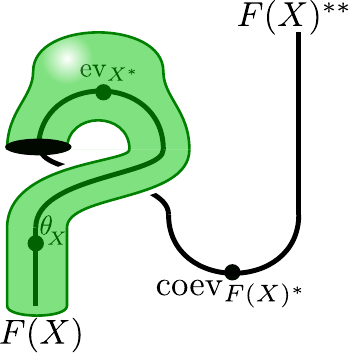} \hspace{-15pt}=\hspace{-4pt}
\pic[1.25]{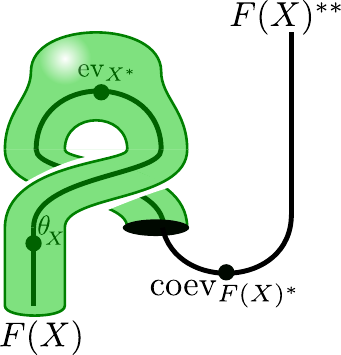}  \hspace{-15pt}=\hspace{-4pt}
\pic[1.25]{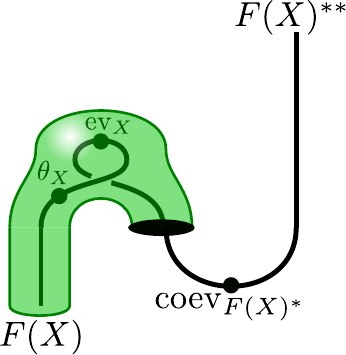}\\ \nonumber
&=
\pic[1.25]{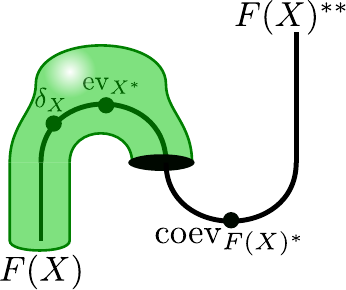} \hspace{-10pt}=\hspace{-4pt}
\pic[1.25]{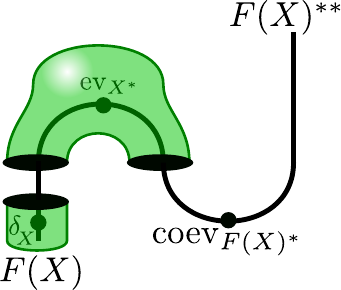} \hspace{-10pt}=\hspace{-4pt}
\pic[1.25]{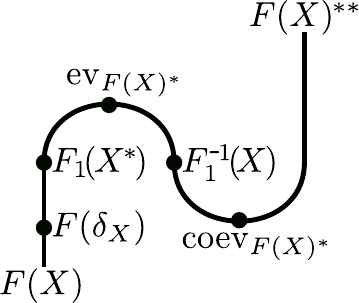}\\
&=
(F_1^*(X))^{-1}\circ F_1(X^*)\circ F(\delta_X) \, ,
\end{align}
\goodbreak
\noindent{}which is exactly the condition for $F$ being pivotal.

Having shown that $F$ is pivotal, we can use the identities~\eqref{eq:Frob_F1_morphs_pivotal} to show that $F$ is also cobraided.
The argument follows from~\eqref{eq:F_graph_calc:assoc-unitality}--\eqref{eq:F_graph_calc:twist} and resembles the one that a symmetric commutative Frobenius algebra is automatically cocommutative:
\begin{align} \nonumber
&
\pic[1.25]{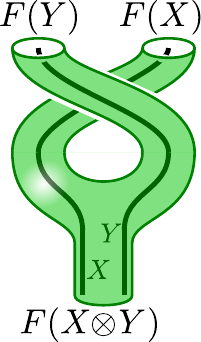} \hspace{-4pt}=\hspace{-4pt}
\pic[1.25]{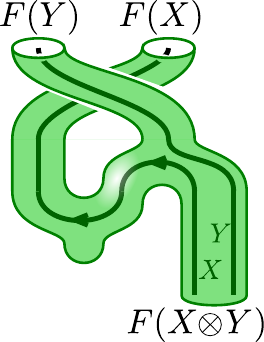} \hspace{-4pt}=\hspace{-4pt}
\pic[1.25]{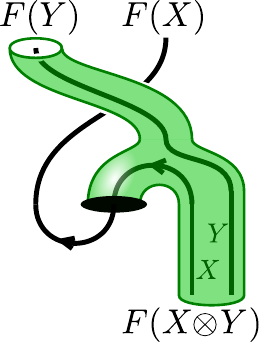} \hspace{-4pt}=\hspace{-4pt}
\pic[1.25]{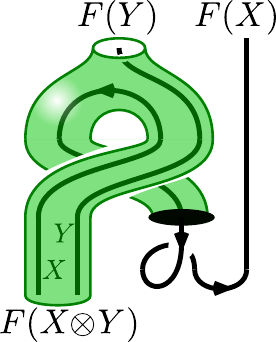}\\
&
=\hspace{-4pt}
\pic[1.25]{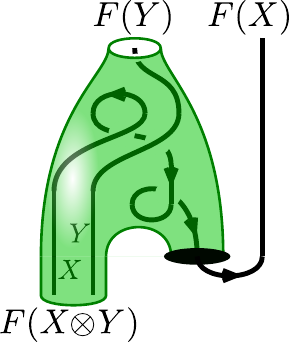} \hspace{-4pt}=\hspace{-4pt}
\pic[1.25]{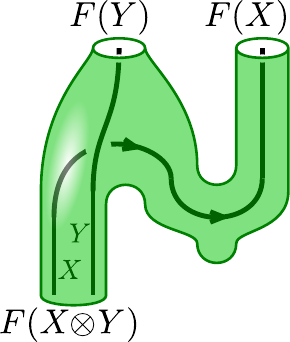} \hspace{-2pt}=\hspace{-4pt}
\pic[1.25]{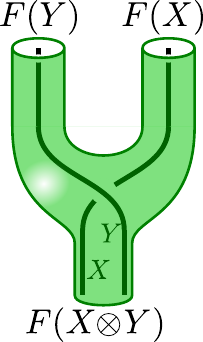} \, .
\end{align}

The other implication is shown by similar computations.
\end{proof}
\begin{rem}
If $F\colon\mcC\ra\mcD$ is a ribbon Frobenius functor, the identities \eqref{eq:F_graph_calc:assoc-unitality}--\eqref{eq:F_graph_calc:twist}, as well as the pivotality of $F$, allow one to deform the tubes in the graphical calculus of $\mcD$ up to an isotopy preserving the string diagrams in $\mcC$ inside the tubes.
This makes the notation very intuitive to use.
\end{rem}

\subsection{Preservation of Frobenius algebras}
\label{subsec:FF_preservation_of_FAs}
Let $F\colon\mcC\ra\mcD$ be a monoidal Frobenius functor and let $A = (A,\mu,\eta,\Delta,\vareps)$ a Frobenius algebra in $\mcC$.
Using the relations~\eqref{eq:F_graph_calc:assoc-unitality}--\eqref{eq:F_graph_calc:Frob} it is easy to show that $F(A)$ is a Frobenius algebra in $\mcD$ (see~\cite[Cor.\,5]{DP}) with multiplication $F(\mu) \circ F_2(A,A)$, unit $F(\eta)\circ F_0$, comultiplication $\overline{F_2}(A,A) \circ F(\Delta)$ and counit $\overline{F_0} \circ F(\vareps)$, or graphically:
\begin{equation}
\pic[1.25]{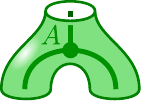}   \, , \quad
\pic[1.25]{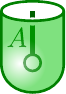}   \, , \quad
\pic[1.25]{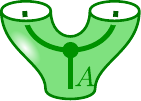} \, , \quad
\pic[1.25]{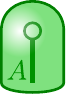} \, .
\end{equation}
Similarly, for a left module $L = (L,\l)\in{_A\mcC}$ and a right module $K = (K,\rho)\in\mcC_A$, the images $F(L)$ and $F(K)$ are left and right $F(A)$-modules respectively with the actions $F(\l)\circ F_2(A,L)$ and $F(\rho)\circ F_2(K,A)$, or
\begin{equation}
\label{eq:FA_actions}
\pic[1.25]{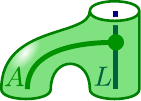} \quad , \qquad
\pic[1.25]{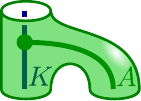} \, .
\end{equation}
The property~\eqref{eq:F_graph_calc:Frob} guarantees that the $F(A)$-comodule structures on $F(L)$ and $F(K)$, obtained by similarly transporting the coactions on $L$ and $K$, coincide with the ones induced by $F(A)$ as in~\eqref{eq:Frob_coaction}.

In this section we look at variations of this observation when $F$ is in addition separable, pivotal, (co)braided or ribbon.

\medskip

We start by noting that in general $F(A)$ need not be separable even if $A$ is.
Unsurprisingly one has instead:
\begin{prp}
If $F\colon\mcC\ra\mcD$ is a separable monoidal Frobenius functor and $A\in\mcC$ a separable Frobenius algebra, then $F(A)\in\mcD$ is a separable Frobenius algebra as well.
\end{prp}
\begin{proof}
If $\psi_F\colon\opid' \ra F(\opid)$ is a section of $F$ in the sense of Definition~\ref{def:FF_sep_cond} and $\psi: \opid \ra A$ is a section of $A$ in the sense that~\eqref{eq:FA_sep_cond_ito_psi} holds, then $\opid' \xra{\psi_F} F(\opid) \xra{F(\psi)} F(A)$ is a section of $F(A)$.
\end{proof}
\begin{rem}
\label{rem:F-F1_sections}
By the above proposition, for a separable monoidal Frobenius functor $F\colon\mcC\ra\mcD$, $F(\opid)\in\mcD$ is always a separable Frobenius algebra.
If $\mcD$ is pivotal and $F(\opid)$ symmetric, a similar computation as in~\eqref{eq:tr-psi_indep_calc} shows that $\psi_F\colon\opid'\ra F(\opid)$ is a section of the functor $F$ if and only if it is a section of the algebra $F(\opid)$.
We note that the assumption for $F$ to be separable is necessary in this case: in general $F$ need not be separable even if $F(\opid)$ is a separable Frobenius algebra and so a section of $F(\opid)$ need not be a section of $F$.
\end{rem}

As expected, the definitions of pivotal and (co)braided Frobenius functors are designed to preserve the corresponding properties of Frobenius algebras:
\begin{prp}
\label{prp:FF_sym-br-algs}
Let $F\colon\mcC\ra\mcD$ be a monoidal Frobenius functor and $A\in\mcC$ a Frobenius algebra.
\begin{enumerate}[i)]
\item If $\mcC$, $\mcD$, $F$ are pivotal and $A$ is symmetric, then $F(A)$ is symmetric as well.
\item If $\mcC$, $\mcD$ are braided, $F$ is (co)braided and $A$ is (co)commutative, then $F(A)$ is (co)commutative as well.
\end{enumerate}
\end{prp}
\begin{proof}
For i) one uses~\eqref{eq:F_graph_calc:assoc-unitality}--\eqref{eq:F_graph_calc:Frob} and~\eqref{eq:Frob_F1_morphs},\eqref{eq:Frob_F1_morphs_pivotal} to compute
\begin{equation}
\pic[1.25]{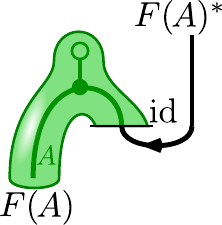} \hspace{-10pt}=\hspace{-4pt}
\pic[1.25]{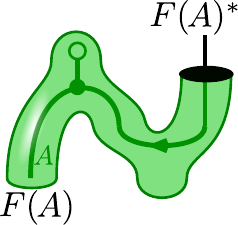} \hspace{-4pt}=\hspace{-4pt}
\pic[1.25]{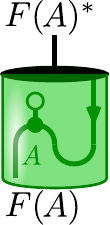} \hspace{-4pt}=\hspace{-4pt}
\pic[1.25]{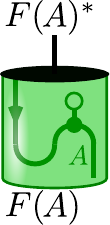} \hspace{-4pt}=\dots =\hspace{-10pt}
\pic[1.25]{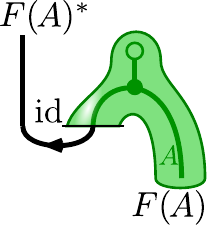}.
\end{equation}
For ii), in case of $F$ being braided and $A$ commutative, \eqref{eq:F_graph_calc:braided} yields
\begin{equation}
\pic[1.25]{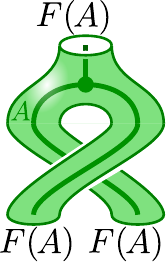} =
\pic[1.25]{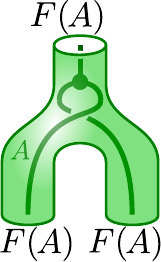} =
\pic[1.25]{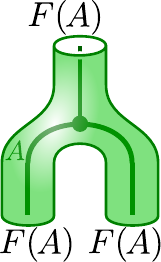} .
\end{equation}
Similarly, in case of $F$ being cobraided and $A$ cocommutative one uses~\eqref{eq:F_graph_calc:cobraided}.
\end{proof}

\begin{conv}
When referring to a pair $(F,\psi_F)$ as a separable Frobenius functor, the entry $\psi_F\colon\opid'\ra F(\opid)$ will mean a fixed choice of an invertible section for $F$.
If $F$ is pivotal with $\mcD$ spherical and fusion, the invertibility can be assumed without loss of generality (see Proposition~\ref{prp:FA_psi_inv}, Remark~\ref{rem:F-F1_sections}).
\end{conv}

For the rest of the section, let $\mcC$ and $\mcD$ be pivotal multifusion categories, $F\colon\mcC\ra\mcD$ a pivotal Frobenius functor and $A=(A,\psi)\in\mcC$ a symmetric separable Frobenius algebra.
We will compare the categories of bimodules ${_A\mcC_A}$ and ${_{F(A)}\mcD_{F(A)}}$ under various separability assumptions.
\begin{defn}
\label{def:F-A_weakly_sep}
We call $F$ \textit{weakly separable with respect to $A$} (or simply the pair $(F,A)$ weakly separable) if the symmetric Frobenius algebra $F(A)\in\mcD$ is separable.
\end{defn}

Let $(F,A)$ be weakly separable and $\psi_F\colon \opid' \ra F(A)$ an invertible section of $F(A)$.
We define natural transformations $F^A_\otimes\colon F(-) \otimes_{F(A)} F(-) \Ra F( - \otimes_A -)$ and $\overline{F}^A_\otimes\colon F(-\otimes_A-)\Ra F(-) \otimes_{F(A)} F(-)$ by setting the following balanced maps for all $K\in\mcC_A$, $L\in{_A\mcC}$:
\begin{equation}
\label{eq:Ftensor_morphs}
F^A_\otimes (K,L) := \pic[1.25]{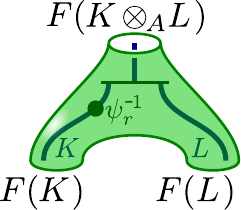} \, , \quad
\overline{F}^A_\otimes (K,L) := \pic[1.25]{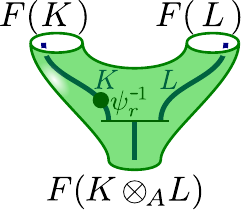} \, .
\end{equation}
For an arbitrary bimodule $M\in\mcACA$, the object $F(M)$ is an $F(A)$-$F(A)$-bimodule with the actions~\eqref{eq:FA_actions} (for $K=L=M$).
This induces from $F$ a functor $F^{A}\colon{_A\mcC_A}\ra{_{F(A)}\mcD_{F(A)}}$.
\begin{prp}
\label{prp:F-A_is_FF}
If $(F,A)$ is weakly separable, $F^A = (F^A, F^A_2, F^A_0, \overline{F}^A_2, \overline{F}^A_0)$ where $F^A_2(-,-) := F^A_\otimes(-,-)$, $\overline{F}^A_2(-,-) := F^A_\otimes(-,-)$ and $F^A_0 = \overline{F}^A_0 = \id_{F(A)}$ is a pivotal Frobenius functor.
\end{prp}
\begin{proof}
Checking the identities~\eqref{eq:F_graph_calc:assoc-unitality}--\eqref{eq:F_graph_calc:Frob} is straightforward if one keeps track of the non-trivial associators and unitors of the categories $\mcACA$ and ${_{F(A)}\mcC_{F(A)}}$ which are given by the balanced maps analogous to the ones in~\eqref{eq:ACA_assoc_unitors}.
We sketch the first identity of~\eqref{eq:F_graph_calc:Frob} only, i.e. that the diagram~\eqref{eq:FF_Frob1} commutes.

Starting with the down-down-right path in~\eqref{eq:FF_Frob1} one computes:
\begin{equation}
\pic[1.25]{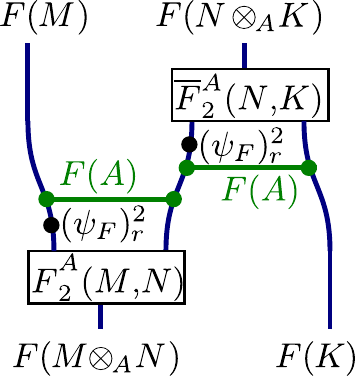} \hspace{-4pt}=\hspace{-4pt}
\pic[1.25]{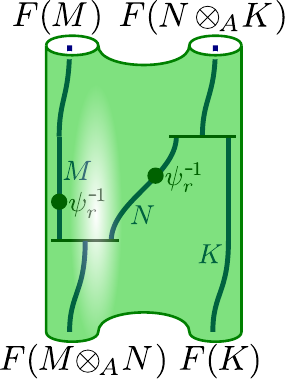} \hspace{-4pt}=\hspace{-4pt}
\pic[1.25]{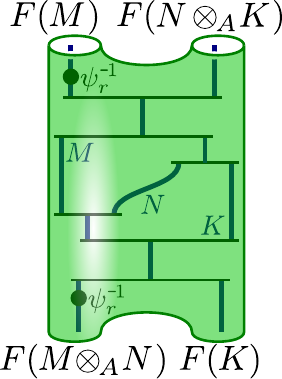} ,
\end{equation}
where the $(\psi_F)^2_r$-insertions on the left-hand side are due to composing the balanced maps as in~\eqref{eq:extra_psi}, in the first step we used that $F$ is a monoidal Frobenius functor and in the second step we used the idempotent splitting property.
On the right-hand side one recognises the associator of $\mcACA$~\eqref{eq:ACA_assoc_unitors} and so it yields the right-down-down path in~\eqref{eq:FF_Frob1}.

That $F^A$ is pivotal follows from $F$ being pivotal and the morphisms $F_1(M)$ as in~\eqref{eq:Frob_F1_morphs}, \eqref{eq:Frob_F1_morphs_pivotal} for $X=M\in\mcACA$ being $F(A)$-$F(A)$-bimodule morphisms.
\end{proof}

Note that for $(F,A)$ weakly separable and $K\in\mcC_A$, $L\in{_A\mcC}$, the idempotent splitting property of $F(A)$ implies that $\overline{F}^A_\otimes(K,L) \circ F^A_\otimes(K,L) = \id_{F(K) \otimes_{F(A)} F(L)}$.
One can check that the other composition $F^A_\otimes(K,L) \circ \overline{F}^A_\otimes(K,L)$ yields an idempotent on $F(K \otimes_A L)$, which in general is not identity.
Trying to circumvent this leads to a stronger separability condition on the pair $(F,A)$, which is easy to formulate noting that the forgetful functor $U\colon\mcACA\ra\mcC$ has a natural structure of a pivotal Frobenius functor (see Example~\ref{eg:FF_forgetful} below) and the composition of monoidal Frobenius functors is again a monoidal Frobenius functor.
\begin{defn}
\label{def:F-A_strongly_sep_full}
Let $F\colon\mcC\ra\mcD$ be a pivotal Frobenius functor between pivotal multifusion categories and $A\in\mcC$ a symmetric separable Frobenius algebra.
We call $F$
\begin{itemize}
\item
\textit{strongly separable with respect to $A$} (or simply the pair $(F,A)$ strongly separable) if the pivotal Frobenius functor $F \circ U\colon \mcACA \ra \mcD$ is separable.
\item
\textit{full with respect to $A$} (or simply the pair $(F,A)$ full) if the functor $F^A\colon\mcACA\ra{_{F(A)}\mcD_{F(A)}}$ is surjective on the spaces of morphisms.
\end{itemize}
\end{defn}
The condition for $(F,A)$ to be strongly separable is equivalent to the existence of a section $\psi_F\colon \opid' \ra F(A)$ such that for all $M,N\in \mcC_A$ one has
\begin{equation}
\label{eq:F_strong_sep_cond}
\pic[1.25]{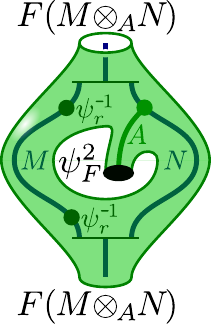} = \hspace{-12pt}
\pic[1.25]{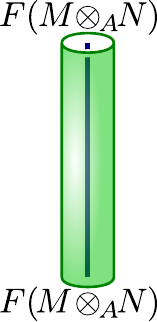} \hspace{-6pt} ,  \text{ or equivalently }
\pic[1.25]{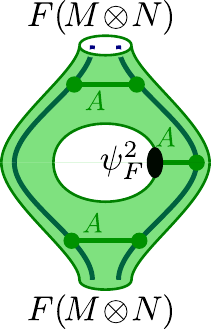} = \hspace{-12pt}
\pic[1.25]{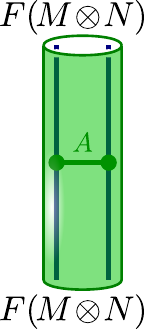} 
\end{equation}
(we will sometimes denote $\psi_F$-insertions as in the second equality in~\eqref{eq:F_strong_sep_cond}).
It is then also a section for the symmetric Frobenius algebra $F(A)$, so naturally $(F,A)$ is also weakly separable.
Since the morphisms~\eqref{eq:Ftensor_morphs} for $K=M$ and $L=N$ are (as balanced morphisms $F(M)\otimes_{F(A)}F(N)\lra F(M\otimes_A N)$) inverses of each other in this case, from Propositions~\ref{prp:surj_pivotal_functs} and~\ref{prp:F-A_is_FF} one obtains
\begin{prp}
\label{prp:F-A_pivotal_embed}
If $(F,A)$ is strongly separable, $F^A = (F^A, F_2^A, F_0^A)$ where $F_2^A(-,-) := F_\otimes^A(-,-)$, $F_0^A := \id_{F(A)}$ is a (strong monoidal) pivotal functor $\mcACA \ra {_{F(A)}\mcD_{F(A)}}$.
If in addition $(F,A)$ is full, $F^A$ is a full embedding of pivotal categories.
\end{prp}

\subsection{Examples}
\label{subsec:FF_examples}
The following examples serve to illustrate the use of monoidal Frobenius functors.
We will rely on some of them in the later sections.

\begin{example}
\label{eg:Vect-C_functor}
For a (multi)tensor category $\mcC$, a linear monoidal Frobenius functor $F\colon\Vect_\opk\ra\mcC$ is determined by the Frobenius algebra $F(\opk)$, which is separable exactly when $F$ is separable.
Moreover, if $\mcC$ is pivotal/braided, then by Proposition~\ref{prp:FF_sym-br-algs} $F$ is pivotal/(co)braided precisely when $F(\opk)$ is symmetric/(co)commuta-tive.
\end{example}

In the following examples we let $\mcC$ be a pivotal multifusion category and $(A,\psi)$ a symmetric separable Frobenius algebra in $\mcC$.
The qualifiers `pivotal' and `symmetric' can in principle be dropped, but this is the setting which will be the most relevant for us in later sections.

\begin{example}
\label{eg:FF_forgetful}
The forgetful functor $U\colon\mcACA\ra\mcC$ is a separable pivotal Frobenius functor with the section $\psi\colon \opid \ra A = U(\opid_\mcACA)$ and the structure morphisms
\begin{equation}
U_2(M,N) := \hspace{-4pt}\pic[1.25]{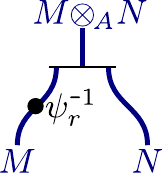} \hspace{-6pt} , \quad
U_0 := \eta \, , \quad
\overline{U_2}(M,N) := \hspace{-4pt}\pic[1.25]{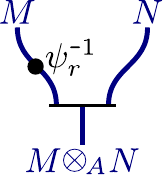} \hspace{-6pt} , \quad
\overline{U_0} := \vareps \, ,
\end{equation}
where $M,N\in\mcACA$ and $\eta\colon\opid\ra A$ and $\vareps\colon A\ra\vareps$ are the unit and the counit of $A$.
Indeed, similar computations as the one in the proof of Proposition~\ref{prp:F-A_is_FF} show that the diagrams~\eqref{eq:FF_Frob1}, \eqref{eq:FF_Frob2} commute (where one takes $F=\Id$ and does not regard the structure maps $U_2$, $\overline{U_2}$ as defining morphisms with the domain/codomain $M \otimes_A N$ via balanced maps).

Note that as $U$ is separable, it is strongly separable with respect to any symmetric separable Frobenius algebra $(A',\psi')$ in $\mcACA$.
By Proposition~\ref{prp:FF_sym-br-algs} $(U(A'),\opid\xra{\eta}A\xra{\psi'}A')$ is a symmetric separable Frobenius algebra in $\mcC$, whose (co)multiplication morphisms are exactly the balanced maps in $\mcC$ that would define the (co)multiplication morphisms of $A'$ in $\mcACA$ (in particular, $U(\opid_{\mcACA})$ gives back the algebra $(A,\psi)$ in $\mcC$).
As the $A$-actions on $U(A')$ can be used to equip each $U(A')$-$U(A')$-bimodule with a structure of an $A$-$A$-bimodule, $U$ is also full with respect to $A'$.
\end{example}

\begin{example}
\label{eg:FF_Ind_functor}
The induced bimodule functor $[\Ind\colon\mcC\ra\mcACA] := A \otimes - \otimes A$ is a pivotal Frobenius functor with structure morphisms
\begin{align}\nonumber
&
\Ind_2(X,Y) := \hspace{-4pt}\pic[1.25]{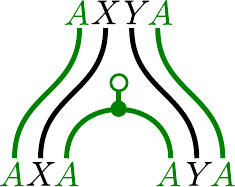} \hspace{-6pt} ,
&&
\overline{\Ind_2}(X,Y) := \hspace{-4pt}\pic[1.25]{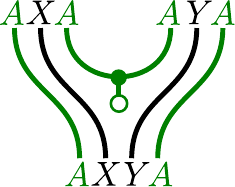} \hspace{-6pt} ,\\
&
\Ind_0 := \D \, , 
&&
\overline{\Ind_0} := \mu \, ,
\end{align}
where $X,Y\in\mcC$ and $\mu\colon A \otimes A \ra A$ and $\D\colon A \ra A \otimes A$ are the multiplication and comultiplication of $A$.
Again, if $\mcC$ is pivotal and $A$ symmetric, $\Ind$ is also pivotal.

Note that if $\mcC$ is multifusion and $\opid=\bigoplus\opid_i$ is the decomposition of the monoidal unit into simple objects, each $\opid_i$ is trivially a $\D$-separable Frobenius algebra in $\mcC$.
The category of bimodules ${_{\opid_i}\mcC_{\opid_i}}$ is then equivalent to the component category $\mcC_{ii} = \opid_i \otimes \mcC \otimes \opid_i$ and the projection functor $\opid_i \otimes - \otimes \opid_i$ is exactly the bimodule induction functor.
In particular, if $\mcC$ is pivotal then $\opid_i$ is also symmetric, and if $\mcC$ is in addition indecomposable, the equivalence $\mcZ(\mcC)\simeq \mcZ(\mcC_{ii})$ as in Proposition~\eqref{prp:ZA-ZAii_equiv} is pivotal.
\end{example}

\begin{example}
\label{eg:FF_EAl_functor}
Let $\mcC$, $(A,\psi)$ be as in the previous example with $\mcC$ in addition ribbon.
The (left) local induction functor $E_A^l\colon\mcC\ra\mcC$ (see~\cite[Sec.\,3]{FFRS}) is defined on an arbitrary object $X\in\mcC$ and a morphism $f\colon X \ra Y$ by
\begin{equation}
\label{eq:EAl_functor}
E_A^l(X) := \im P_A^l(X) = \im\pic[1.25]{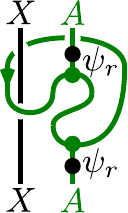} \, , \qquad
E_A^l(f) := \pic[1.25]{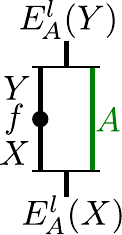} \, ,
\end{equation}
where one can check that the morphism $P_A^l(X)\in\End_\mcC(X\otimes A)$ is an idempotent (when $A$ is $\D$-separable this is done in~\cite[Lem.\,5.2]{FRS1}).
$E_A^l$ is a ribbon Frobenius functor with the structure morphisms
\begin{align} \nonumber
&
(E_A^l)_2(X,Y) := \pic[1.25]{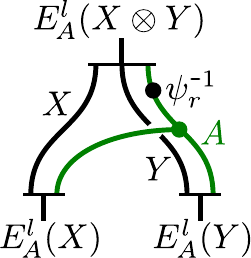} ,
&&
\overline{(E_A^l)_2}(X,Y) := \pic[1.25]{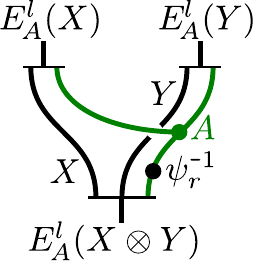}, \\ \label{eq:EAl_struct_morphs}
&
(E_A^l)_0 := \pi_A^l(\opid) \circ \psi_r^A \circ \eta \, , 
&&
\overline{(E_A^l)_0} := \vareps \circ \psi_r^A \circ \imath_A^l(\opid) \, ,
\end{align}
where for arbitrary $X\in\mcC$, $\pi_A^l(X)\colon X\otimes A \lra E_A^l(X) : \imath_A^l(X)$ are the projection/inclusion morphisms, splitting the idempotent $P_A^l(X)$ (as usual depicted in the graphical calculus by horizontal lines).
The relations~\eqref{eq:F_graph_calc:assoc-unitality}--\eqref{eq:F_graph_calc:twist} are straightforward to check using the idempotent splitting property, for example~\eqref{eq:F_graph_calc:braided} and~\eqref{eq:F_graph_calc:twist} follow from the computations
\begin{equation}
\pic[1.25]{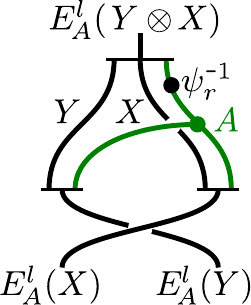} \hspace{-8pt}=\hspace{-8pt}
\pic[1.25]{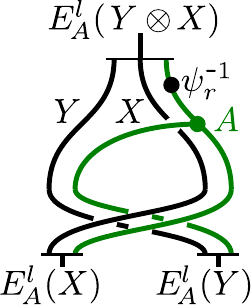}  \hspace{-8pt}=\hspace{-8pt}
\pic[1.25]{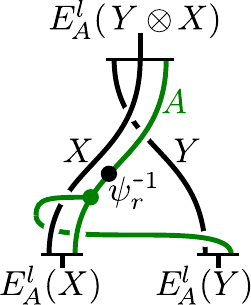}  \hspace{-8pt}\stackrel{(*)}{=}\hspace{-8pt}
\pic[1.25]{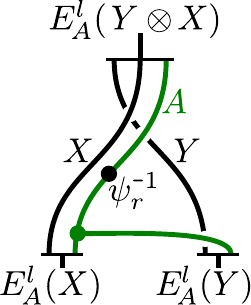}\hspace{-6pt}.
\end{equation}
and
\begin{equation}
\pic[1.25]{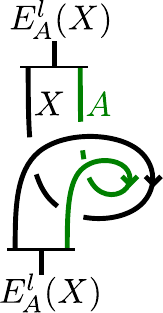} =
\pic[1.25]{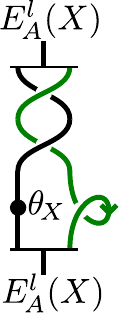} \stackrel{(**)}{=}
\pic[1.25]{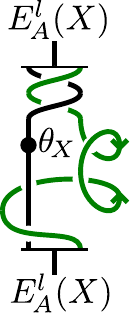} =
\pic[1.25]{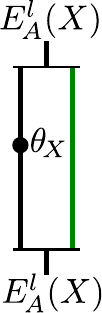} ,
\end{equation}
where in the steps $(*)$ and $(**)$ one uses the identity $\imath_A^l(X) = \imath_A^l(X) \circ  P_A^l(X)$ and the properties of the algebra $A$ (cf.\ \cite[Lem.\,3.5, Lem.\,3.11)]{FFRS}).

Flipping the undercrossings and overcrossings in the definition of $P_A^l(X)$ in~\eqref{eq:EAl_functor} yields a different idempotent $P_A^r(X)$.
They can be used to define the analogous right local induction functor $E_A^r\colon\mcC\ra\mcC$, which is also a ribbon Frobenius functor.
The Frobenius algebras $E^{l/r}_A(\opid)$ are symmetric and commutative by Proposition~\ref{prp:FF_sym-br-algs} and are known as the left/right centres of $A$, see~\cite[Def.\,5.8]{FRS1}, \cite[Lem.\,3.13]{FFRS}.
A priori they need not be separable and so $E^{l/r}_A$ are not in general separable Frobenius functors.
\end{example}

\begin{rem}
Example~\ref{eg:FF_EAl_functor} is best understood in terms of 3-dimensional TQFTs obtained from a modular fusion category (MFC) $\mcC$ via the Reshetikhin--Turaev construction.
In it, the objects of $\mcC$ are used to label the framed line defects, whereas a symmetric separable Frobenius algebra $(A,\psi)$ in a MFC $\mcC$ constitutes a datum for a surface defect, constructed using the internal state-sum procedure (This was briefly explained in Section~\ref{sec:introduction}, see~\cite{KSa,FSV,CRS2} for more details).
For a line defect $X\in\mcC$, the images $E^{l/r}_A(X)$ have an interpretation of a line defect obtained by wrapping $X$ with a cylindrical surface defect with the label $(A,\psi)$ (the two functors $E^{l/r}_A$ correspond to the two possible orientations of this defect).
The structure morphisms~\eqref{eq:EAl_struct_morphs} correspond to merging/splitting two cylindrical defects and opening/closing an empty cylindrical defect.
The conditions~\eqref{eq:F_graph_calc:assoc-unitality}--\eqref{eq:F_graph_calc:twist} follow automatically, since the surface defect is topological, i.e.\ can be deformed.

An interesting observation is that by Example~\ref{eg:Vect-C_functor} a ribbon Frobenius functor $F\colon\Vect_\opk\ra\Vect_\opk$ is the same as a 2-dimensional TQFT (i.e.\ a commutative Frobenius algebra in $\Vect_\opk$), and so can be equivalently interpreted as a surface defect in the trivial 3-dimensional TQFT.
In this simple case such surface defects go beyond the aforementioned internal state-sum procedure as $F(\opk)$ need not be semisimple.
We leave it for the future work to determine under what conditions a ribbon Frobenius functor $F\colon\mcC\ra\mcD$ between the MFCs $\mcC$, $\mcD$ can be used to define a surface defect and whether it is possible to obtain non-semisimple surface defects in case $\mcC$ or $\mcD$ is not equivalent to $\Vect_\opk$.
\end{rem}

\section{Generalised orbifolds}
\label{sec:gen_orbifolds}
The notion of an orbifold datum is of central importance to the results of this paper.
The motivation to introduce it stems from the generalised orbifold construction of TQFTs, which given a defect TQFT $Z^{\operatorname{def}}$ and an orbifold datum $\opA$ produces a new (ordinary) TQFT $Z^{\operatorname{orb}\opA}$, see~\cite{CRS1}.
When specialised to the 3-dimensional defect TQFTs of Reshetikhin--Turaev (RT) type~\cite{CRS2}, the generalised orbifolds can be shown to be of RT type as well~\cite{CRS3,CMRSS2} and so an orbifold datum $\opA$ can be seen as an algebraic input in a MFC $\mcC$ which produces another MFC $\mcC_\opA$~\cite{MR1}.

In this section we review the notion of an orbifold datum $\opA$ in a MFC $\mcC$, the construction of the MFC $\mcC_\opA$ associated to it, as well as two ways to map $\opA$ to other orbifold data.

\subsection{Orbifold data and the associated MFCs}
\label{subsec:orb_data_and_assoc_cats}
Throughout this section, let $\mcC$ be a MFC.
Recall that if $A=(A,\psi)$ is a symmetric separable Frobenius algebra in $\mcC$, so is $A^{\otimes n}$ where the structure morphisms are as in~\eqref{eq:AxB_alg_morphs}.
In the following definition we will need some flexibility when working with an $A$-$A\otimes A$-bimodule $T$: by adapting~\eqref{eq:AxB_multimodul_action} such bimodule can be seen as an object $T\in\mcC$ having one left $A$-action and two right $A$-actions satisfying
\begin{equation}
\label{eq:right_T_actions_commute}
\pic[1.25]{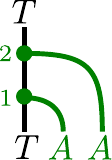} =
\pic[1.25]{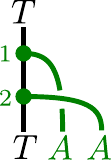} \, , \qquad
\pic[1.25]{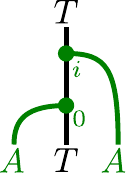} =
\pic[1.25]{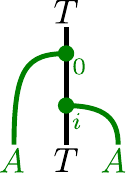} \, , \quad i\in\{1,2\} \, ,
\end{equation}
where the indices $1$, $2$ are used to distinguish the two right actions and similarly the left action is indicated by $0$ whenever we find it necessary to avoid ambiguity.
For a left module $L\in{}_A\mcC$ and a right module $K\in\mcC_A$, the relative tensor products $K \otimes_0 T$ and $T\otimes_i L$, $i\in\{1,2\}$ are defined as images of idempotents like the one in~\eqref{eq:PKL_idempotent} using the corresponding action of $T$.
In particular, $T \otimes_i T$, $i\in\{1,2\}$ are $A$-$A^{\otimes 3}$-bimodules.
We also adapt the notations~\eqref{eq:psi_lr_defs} as follows:
\begin{equation}
\label{eq:psis_on_T}
\pic[1.25]{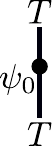} = \pic[1.25]{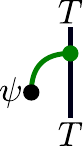} \, , \qquad
\pic[1.25]{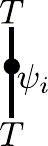} = \pic[1.25]{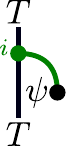} \, , \quad
i\in\{1,2\} \, .
\end{equation}

\begin{defn}
\label{def:orb_datum}
An \textit{orbifold datum} in $\mcC$ is a tuple $\opA = (A,T,\a,\abar,\psi,\phi)$ where
\begin{itemize}
\item
$(A,\psi)$ is a symmetric separable Frobenius algebra in $\mcC$;
\item
$T$ is an $A$-$A^{\otimes 2}$-bimodule in $\mcC$;
\item
$\a\colon T \otimes_2 T \ra T \otimes_1 T$ is an $A$-$A^{\otimes 3}$-bimodule isomorphism with the inverse $\abar$;
\item
$\phi\in\opk^\times$;
\end{itemize}
which satisfies the conditions \eqrefO{1}--\eqrefO{8} in Figure~\ref{fig:orb_datum_conds}.
\end{defn}
\begin{rem}
\label{rem:Delta-sep_orb_data}
The definition of an orbifold datum $\opA=(A,T,\a,\abar,\psi,\phi)$ is slightly different in our main source references~\cite{CRS3, MR1, MR2, CMRSS2}.
There the algebra $A$ was required to be $\D$-separable and the $\psi$ entry was instead taken to be an $A$-$A$-bimodule isomorphism $A\ra A$.
The reason for this difference is our use of the more general separability condition~\eqref{eq:FA_sep_cond_ito_psi}.
All examples in the references can be rewritten in our setting by using the rescaled algebras as in Example~\ref{eg:Euler_completion}.
This also explains minor differences in the identities \eqrefO{1}--\eqrefO{8} in between the two settings as rescaling also changes the action/coaction on modules.
\end{rem}

That $\a$, $\abar$ are $A$-$A^{\otimes 3}$-bimodule morphisms means that the corresponding balanced maps commute with the various $A$-(co)actions as follows:
\begin{equation}
\label{eq:aabar-A_commute}
\pic[1.25]{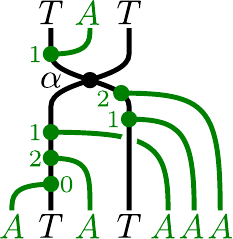}\hspace{-4pt}=\hspace{-4pt}
\pic[1.25]{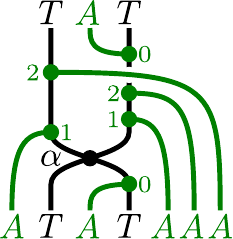}, \,
\pic[1.25]{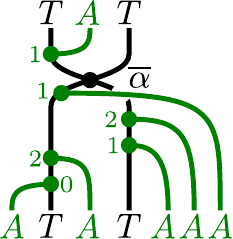}\hspace{-4pt}=\hspace{-4pt}
\pic[1.25]{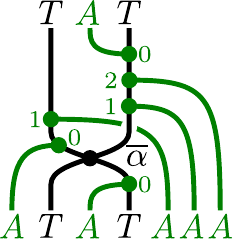}.
\end{equation}

The main results of this paper explore three further categories which one can construct given an orbifold datum in a MFC $\mcC$ as described in~\cite[Sec.\,3]{MR1}.
\begin{defn}
Let $\opA$ be an orbifold datum in $\mcC$.
Define the category $\mcC_\opA$ to have
\begin{itemize}[leftmargin=*]
\item \textit{objects:} tuples $M=(M,\tau_1,\tau_2,\taubar{1},\taubar{2})$, where
\begin{itemize}
\item $M\in\mcACA$ is an $A$-$A$-bimodule;
\item $\tau_1\colon M \otimes_0 T \ra T \otimes_1 M$, $\tau_2\colon M \otimes_0 T \ra T \otimes_2 M$ are $A$-$A^{\otimes 3}$-bimodule isomorphisms with inverses $\taubar{1}$, $\taubar{2}$, denoted by
\begin{equation}
\hspace{-20pt}
\tau_1 = \pic[1.25]{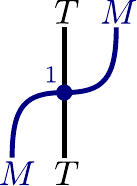} \hspace{-8pt}, \quad
\tau_2 = \pic[1.25]{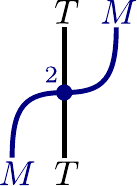} \hspace{-8pt}, \quad
\taubar{1} = \pic[1.25]{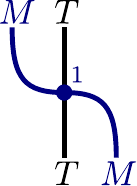} \hspace{-8pt}, \quad
\taubar{2} = \pic[1.25]{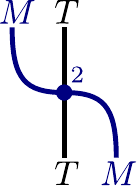} \hspace{-8pt},
\end{equation}
such that the identities in \eqrefT{1}--\eqrefT{7} in Figure~\ref{fig:CA_identities} are satisfied; they will be referred to as the \textit{$T$-crossings} of $M$;
\end{itemize}
\item
\textit{morphisms:} for $M,N\in\mcC_\opA$, a morphism $f\colon M \ra N$ is a morphism of the underlying $A$-$A$-bimodules such that
\begin{equation}
\label{eq:M}
\pic[1.25]{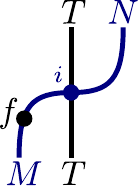} = \pic[1.25]{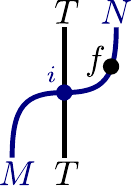} \, , \quad i = 1,2 \, .
\end{equation}
\end{itemize}
Similarly one defines two categories $\mcC_\opA^i$, $i=1,2$, whose objects are triples $(M,\tau_i,\taubar{i})$, with $M\in\mcACA$ and $\tau_i$, $\taubar{i}$ a $T$-crossing and its inverse, i.e.\ satisfying~\eqrefT{1} and \eqrefT{4}--\eqrefT{7} for $i=1$ and \eqrefT{3}--\eqrefT{7} for $i=2$, and a morphism $f\in\mcC_\opA^i(M,N)$ being a bimodule morphism satisfying~\eqref{eq:M} for the given value of $i$.
\end{defn}

For an orbifold datum $\opA$, the categories $\mcC_\opA$, $\mcC^1_\opA$, $\mcC^2_\opA$ are multifusion and pivotal, with the tensor product and dualities inherited from $\mcACA$.
In particular, for two objects $M$, $N$ in either of these categories, the $T$-crossings of $M\otimes_A N$, the monoidal unit $A$ and a dual $M^*$ are defined by
\begin{equation}
\label{eq:CA_MN_A_T-crossings}
\pic[1.25]{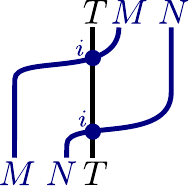} ,
\pic[1.25]{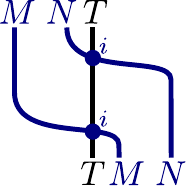} , \,
\pic[1.25]{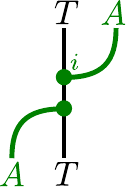} ,
\pic[1.25]{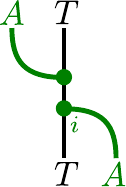} , \,
\pic[1.25]{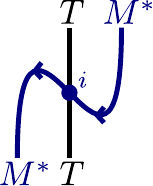} ,
\pic[1.25]{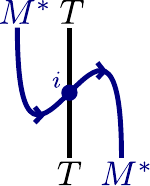} ,
\end{equation}
with $i$ having values in $\{1,2\}$, $\{1\}$ and $\{2\}$ correspondingly for $\mcC_\opA$, $\mcC_\opA^1$ and $\mcC_\opA^2$.
The identities~\eqrefT{6}, \eqrefT{7} imply that the evaluation/coevaluation morphisms of $\mcACA$ satisfy~\eqref{eq:M}.
We use the symbols $\otimes_\opA$, $\otimes_\opA^1$, $\otimes_\opA^2$ to denote the monoidal product in the categories $\mcC_\opA$, $\mcC_\opA^1$, $\mcC_\opA^2$ respectively.

\begin{defn}
We call an orbifold datum $\opA$ \textit{simple} if $\mcC_\opA$ is fusion (i.e.\ if $\opid_{\mcC_\opA} := A$ is a simple object of $\mcC_\opA$) and \textit{haploid} if both $\mcC_\opA^1$ and $\mcC_\opA^2$ are fusion.
\end{defn}

Additionally, $\mcC_\opA$ is ribbon with the braiding and twist morphisms for arbitrary $M,N\in\mcC_\opA$ given by
\begin{equation}
\label{eq:CA_br_twist}
c^\opA_{M,N}  = \phi^2 \cdot \pic[1.25]{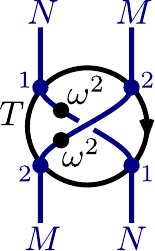} \, , \quad
\theta^\opA_M = \phi^2 \cdot \pic[1.25]{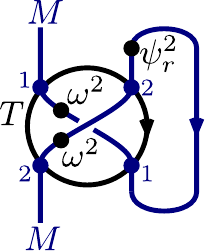} \, .
\end{equation}

\begin{figure}[!tbp]
	\captionsetup[subfigure]{labelformat=empty}
	\centering
	\begin{subfigure}[b]{0.5\textwidth}
		\centering
		\pic[1.25]{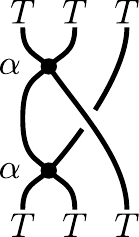}$=$\pic[1.25]{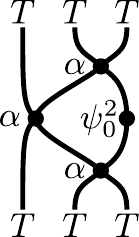}
		\caption{}
		\label{eq:O1}
	\end{subfigure}\hspace{-2em}\raisebox{5.5em}{(O1)}\\
	\begin{subfigure}[b]{0.4\textwidth}
		\centering
		\pic[1.25]{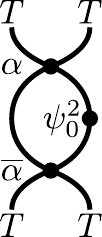}$=$\pic[1.25]{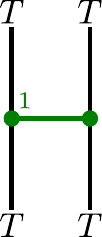}
		\caption{}
		\label{eq:O2}
	\end{subfigure}\hspace{-2em}\raisebox{5.5em}{(O2)}
	\begin{subfigure}[b]{0.4\textwidth}
		\centering
		\pic[1.25]{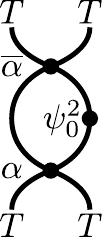}$=$\pic[1.25]{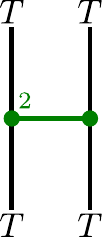}
		\caption{}
		\label{eq:O3}
	\end{subfigure}\hspace{-2em}\raisebox{5.5em}{(O3)}\\
	\begin{subfigure}[b]{0.4\textwidth}
		\centering
		\pic[1.25]{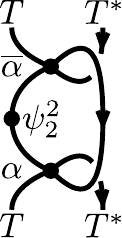}$=$\pic[1.25]{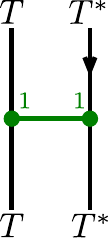}
		\caption{}
		\label{eq:O4}
	\end{subfigure}\hspace{-2em}\raisebox{5.5em}{(O4)}
	\begin{subfigure}[b]{0.4\textwidth}
		\centering
		\pic[1.25]{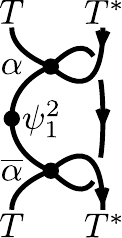}$=$\pic[1.25]{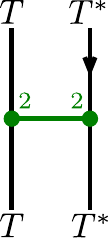}
		\caption{}
		\label{eq:O5}
	\end{subfigure}\hspace{-2em}\raisebox{5.5em}{(O5)}\\
	\begin{subfigure}[b]{0.4\textwidth}
		\centering
		\pic[1.25]{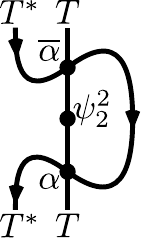}$=$\pic[1.25]{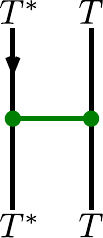}
		\caption{}
		\label{eq:O6}
	\end{subfigure}\hspace{-2em}\raisebox{5.5em}{(O6)}
	\begin{subfigure}[b]{0.4\textwidth}
		\centering
		\pic[1.25]{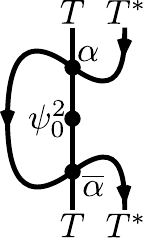}$=$\pic[1.25]{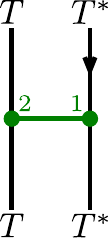}
		\caption{}
		\label{eq:O7}
	\end{subfigure}\hspace{-2em}\raisebox{5.5em}{(O7)}\\
	\begin{subfigure}[b]{0.8\textwidth}
		\centering
		\pic[1.25]{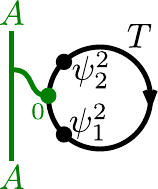}$=$\pic[1.25]{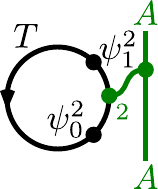}$=$\pic[1.25]{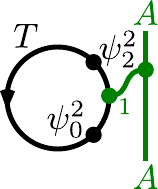}$=$\pic[1.25]{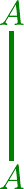}$\cdot \, \phi^{-2}$
		\caption{}
		\label{eq:O8}
	\end{subfigure}\hspace{-1em}\raisebox{4.25em}{(O8)}

\vspace*{-1em}

	\caption{Conditions on an orbifold datum $\opA = (A,T,\alpha,\overline{\alpha}, \psi,\phi)$.}
	\label{fig:orb_datum_conds}
\end{figure}

\begin{figure}[!tbp]
	\captionsetup[subfigure]{labelformat=empty}
	\centering
	\begin{subfigure}[b]{0.45\textwidth}
		\centering
		\pic[1.25]{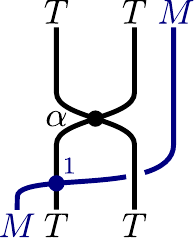}$=$\pic[1.25]{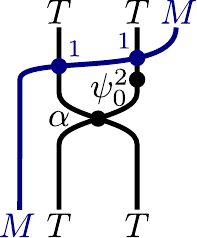}
		\caption{}
		\label{eq:T1}
	\end{subfigure}\hspace{-2em}\raisebox{5.5em}{(T1)}
	\begin{subfigure}[b]{0.45\textwidth}
		\centering
		\pic[1.25]{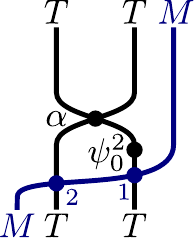}$=$\pic[1.25]{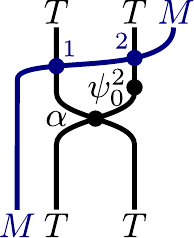}
		\caption{}
		\label{eq:T2}
	\end{subfigure}\hspace{-2em}\raisebox{5.5em}{(T2)}\\
	\begin{subfigure}[b]{0.45\textwidth}
		\centering
		\pic[1.25]{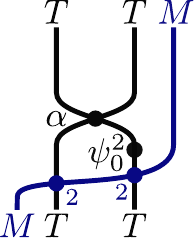}$=$\pic[1.25]{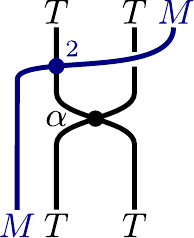}
		\caption{}
		\label{eq:T3}
	\end{subfigure}\hspace{-2em}\raisebox{5.5em}{(T3)}\\
	\begin{subfigure}[b]{0.35\textwidth}
		\centering
		\pic[1.25]{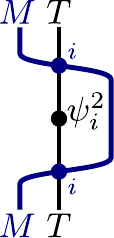}$=$\pic[1.25]{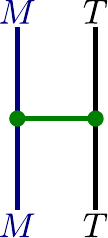}
		\caption{}
		\label{eq:T4}
	\end{subfigure}\hspace{-1em}\raisebox{5.5em}{(T4)}
	\begin{subfigure}[b]{0.35\textwidth}
		\centering
		\pic[1.25]{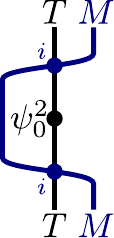}$=$\pic[1.25]{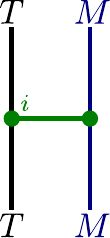}
		\caption{}
		\label{eq:T5}
	\end{subfigure}\hspace{-1em}\raisebox{5.5em}{(T5)}\hspace{4em}\raisebox{5.5em}{$i \in \{1,2\}$}\\
	\begin{subfigure}[b]{0.35\textwidth}
		\centering
		\pic[1.25]{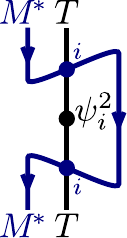}$=$\pic[1.25]{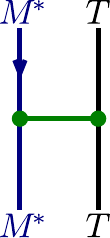}
		\caption{}
		\label{eq:T6}
	\end{subfigure}\hspace{-1em}\raisebox{5.5em}{(T6)}
	\begin{subfigure}[b]{0.35\textwidth}
		\centering
		\pic[1.25]{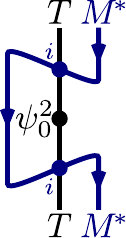}$=$\pic[1.25]{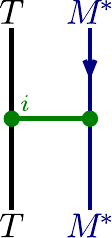}
		\caption{}
		\label{eq:T7}
	\end{subfigure}\hspace{-1em}\raisebox{5.5em}{(T7)}\hspace{4em}\raisebox{5.5em}{$i \in \{1,2\}$}
	\caption{Identities for $\mcC_\opA$.} 
	\label{fig:CA_identities} 
\end{figure}

One has (originally shown in~\cite[Thm.\,3.17]{MR1} and adapted in~\cite{Mul} for symmetric separable Frobenius algebras, see Remark~\ref{rem:psi_insertions}):
\begin{thm}
\label{thm:CA_is_MFC}
For a simple orbifold datum $\opA$, the category $\mcC_\opA$ is a MFC whose global dimension is
\begin{equation}
\label{eq:Dim-CA_formula}
\Dim\mcC_\opA = \frac{\Dim\mcC}{\phi^8\cdot(\tr_\mcC\omega_A^2)^2} \, ,
\end{equation}
with $\tr_\mcC\omega_A^2\neq 0$ holding automatically.
\end{thm}

\medskip

For an arbitrary orbifold datum $\opA=(A,T,\a,\abar,\psi,\phi)$ in a MFC $\mcC$, let us look at some constructions of objects/morphisms in the associated categories $\mcC_\opA$, $\mcC_\opA^1$, $\mcC_\opA^2$.
Firstly, for any left module $L\in{_A\mcC}$ the $A$-$A$-bimodules $T \otimes_1 L$ and $T \otimes_2 L$ are naturally objects of $\mcC_\opA^1$ and $\mcC_\opA^2$ respectively with the $T$-crossings given by
\begin{equation}
\label{eq:TxL_T-crossings}
\pic[1.25]{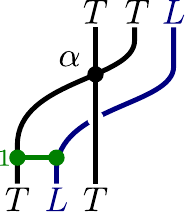},
\pic[1.25]{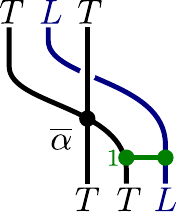} \qquad, \qquad
\pic[1.25]{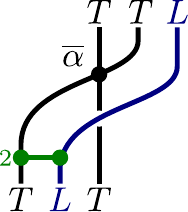},
\pic[1.25]{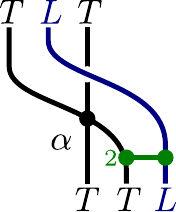}.
\end{equation}
Indeed, in both cases the identities~\eqrefT{4}, \eqrefT{5} are implied by~\eqrefO{2}, \eqrefO{3} and the identities~\eqrefT{6}, \eqrefT{7} by~\eqrefO{6}, \eqrefO{7}.
Then for the first pair~\eqrefT{1} is obtained from~\eqrefO{1}, while for the second pair~\eqrefT{3} is obtained from an analogous identity as~\eqrefO{1} but involving $\abar$, which is obtained by combining~\eqrefO{1}, \eqrefO{2} and \eqrefO{3}.
Of particular importance later will be such kind of objects obtained by setting $L=A$.
In this case one has canonical $A$-$A$-bimodule isomorphisms $T \otimes_1 A \cong T_2$, $T \otimes_2 A \cong T_1$ where $T_2 = T$ seen as a bimodule by forgetting the first right $A$-action and similarly $T_1=T$ by forgetting the second right $A$-action.
Simplifying the $T$-crossings~\eqref{eq:TxL_T-crossings} accordingly, one gets:
\begin{equation}
\label{eq:T1_T2_objs}
(T_2, \a, \abar) \in \mcC_\opA^1 \qquad , \qquad
(T_1, \abar, \a) \in \mcC_\opA^2 \, .
\end{equation}

\medskip

Another family of objects is given by the so called `pipe functor' $P_\opA\colon\mcACA\ra\mcC_\opA$, see~\cite[Sec.\,3.3]{MR1}.
It assigns to a bimodule $M\in\mcACA$ an object of $\mcC_\opA$ whose underlying bimodule is
\begin{equation}
\label{eq:pipe_funct_def}
P_\opA(M) := \im \pic[1.25]{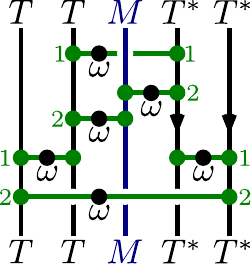} \, ,
\end{equation}
and the $T$-crossings $\tau_1$ and $\tau_2$ are
\begin{equation}
\tau_1 := \pic[1.25]{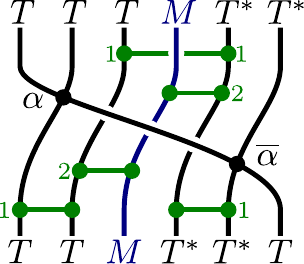} \, , \quad
\tau_2 := \pic[1.25]{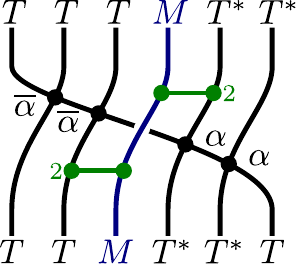} \, ,
\end{equation}
with the inverses $\taubar{1}$, $\taubar{2}$ defined similarly.
The functor $P_\opA$ is both left and right adjoint to the forgetful functor $U\colon\mcC_\opA\ra\mcACA$ (\cite[Prop.\,3.11,\,Rem.\,3.12]{MR1}) and so can be treated as the `free construction' of an object in $\mcC_\opA$.
Moreover one has:
\begin{prp}
\label{prp:pipe_objs_generate}
Any object $M\in\mcC_\opA$ is a subobject of a pipe object $P_\opA(M')$ for some $M'\in\mcACA$.
\end{prp}
\begin{proof}
The pair $(U,P_\opA)$ forms what in~\cite[App.\,A.2]{MR1} was called a separable biadjunction: for the unit $\widetilde{\eta}\colon \Id_{\mcC_\opA} \Ra P_\opA\circ U$ of $U\dashv P_\opA$ and the counit $\vareps\colon P_\opA\circ U \Ra \Id_{\mcC_\opA}$ of $P_\opA \dashv U$ one has $\vareps\circ\widetilde{\eta} = \id$.
Hence $M$ is a subobject of $P_\opA(U(M))$.
\end{proof}
\begin{rem}
\label{rem:pipe_objs_on_induced_bimods}
Abusing the notation, we will also use the symbol $P_\opA$ to denote the functor $P_\opA \circ \Ind_A = P_\opA(A \otimes - \otimes A)\colon\mcC\ra\mcC_\opA$, i.e.\ the pipe functor applied to an induced $A$-$A$-bimodule.
Since induced bimodules generate $\mcACA$, by Proposition~\ref{prp:pipe_objs_generate} the objects $\{P_\opA(X)\}_{X\in\mcC}$ generate $\mcC_\opA$.
\end{rem}

\medskip

Finally, let us look at a way to obtain morphisms in $\mcC_\opA$.
In particular we provide an analogue of the map~\eqref{eq:A-mod-morphs_avg} projecting onto the $A$-module morphisms.
Let $M$, $N$ be arbitrary objects of $\mcC_\opA$ and $f\in\mcC(M,N)$ an arbitrary morphism of the underlying objects.
Then using the identities~\eqrefO{1}--\eqrefO{8} and~\eqrefT{1}--\eqrefT{7} one can show that the `averaged' morphism
\begin{equation}
\label{eq:avg_map}
\operatorname{avg} f := \phi^4\cdot\pic[1.25]{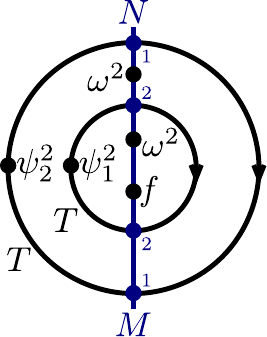}
\end{equation}
commutes with the $A$-actions and satisfies~\eqref{eq:M} and therefore is a morphism in $\mcC_\opA(M,N)$.
Moreover, the map $f\mapsto\operatorname{avg}f$ is indeed the idempotent projecting onto the subspace $\mcC_\opA(M,N)\subseteq\mcC(M,N)$.
To show this note that if $f$ is already in $\mcC_\opA(M,N)$, one can use~\eqref{eq:M}, \eqrefT{4} and~\eqrefO{8} to remove the two $T$-lines in~\eqref{eq:avg_map}.

\subsection{Morita transports}
\label{subsec:Morita_transports}
Let $\mcC$ be a MFC and $\opA=(A,T,\a,\abar,\psi,\phi)$ an orbifold datum in it.
In the following two sections we discuss two ways to obtain another orbifold datum out of $\opA$.
The first one was introduced in~\cite{CRS3} and involves changing (in the present setting isometrically, see Definitions~\ref{def:mod_trace} and~\ref{def:isometric_Morita_mod}) the symmetric separable Frobenius algebra $(A,\psi)$ for a Morita equivalent one $(B,\psi')$.

\medskip

Suppose $(B,\psi')$ is a symmetric separable Frobenius algebra in $\mcC$ and ${}_AR_B$ an isometric Morita module.
We use the terminology as in~\cite{CRS3}, where $R$ was said to `transport $\opA$ along the Morita equivalence'.
One then has (cf.\, \cite[Def.\,3.7, Prop.\,3.8]{CRS3}):
\begin{defnprp}
\label{defnprp:RA_orb_datum}
Let $\opA=(A,T,\a,\abar,\psi,\phi)$ be an orbifold datum in a MFC $\mcC$, $(B,\psi')$ a symmetric separable Frobenius algebra in $\mcC$ and ${}_A R_B$ an isometric Morita module.
\begin{enumerate}[i)]
\item
The \textit{Morita transport of $\opA$ along $R$} is the tuple $R(\opA) := (B, T^R, \a^R,$ $\abar^R, \psi', \phi)$
where
\begin{equation}
\label{eq:RA_orb_datum}
\hspace{-32pt}
T^R     := \im \hspace{-0.25cm}\pic[1.25]{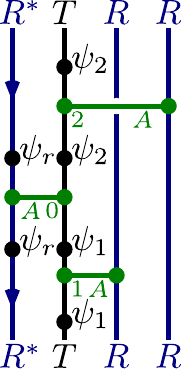}\hspace{-0.1cm}, \,
\a^R    := \hspace{-12pt}\pic[1.25]{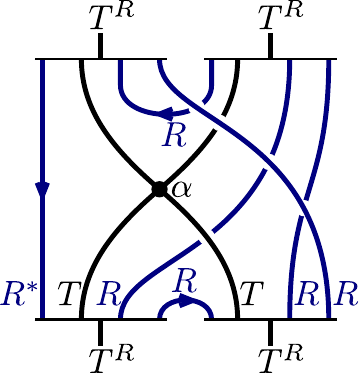}\hspace{-0.25cm}, \,
\abar^F := \hspace{-12pt}\pic[1.25]{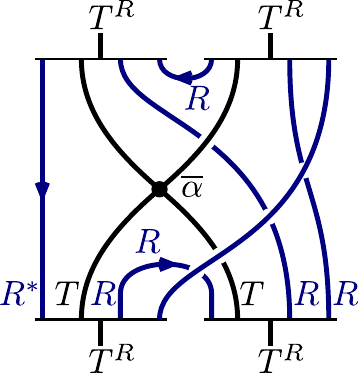}\hspace{-10pt}.
\end{equation}
\item
The \textit{Morita transport of an object $M\in\mcC_\opA$} is defined to be the tuple $R(M) := (M^R, \tau_1^R, \tau_2^R, \taubar{1}^R, \taubar{2}^R)$ where $M^R = R(M) = R^* \otimes_A M \otimes_A R$ and
\begin{align} \nonumber
\hspace{-32pt}
M^R     := \im \hspace{-0.25cm}\pic[1.25]{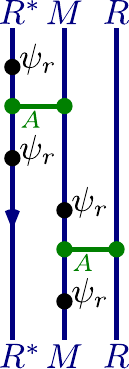}, \quad
&\tau_1^R        := \hspace{-12pt}\pic[1.25]{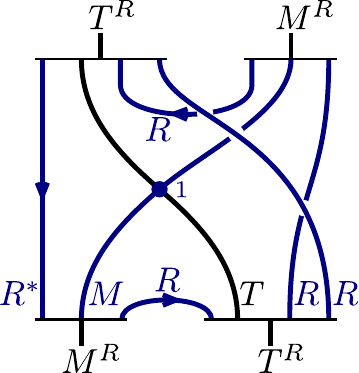}\hspace{-0.25cm}, \,
 \tau_2^R        := \hspace{-12pt}\pic[1.25]{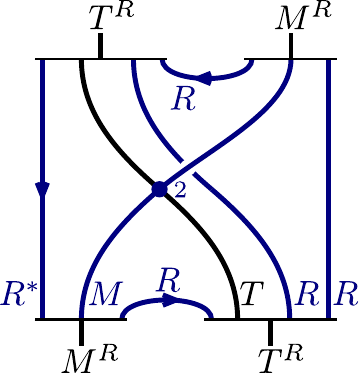}\hspace{-10pt},\\ \label{eq:RM_obj}
&\taubar{1}^R    := \hspace{-12pt}\pic[1.25]{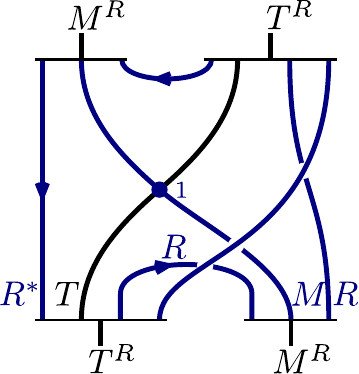}\hspace{-0.25cm}, \,
 \taubar{2}^R    := \hspace{-12pt}\pic[1.25]{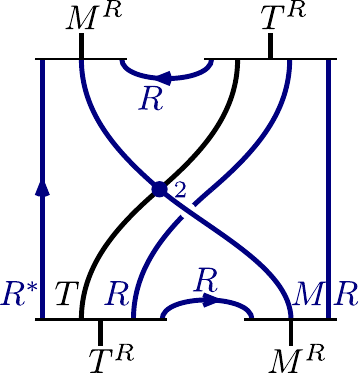}\hspace{-10pt}.
\end{align}
Similarly one defines the transport $R(N)$ of an object $N\in\mcC_\opA^i$, $i\in\{1,2\}$.
\item
$R(\opA)$ is an orbifold datum in $\mcC$ and $R(M)$, $R(N)$ for $M\in\mcC_\opA$, $N\in\mcC_\opA^i$, $i\in\{1,2\}$ are objects of $\mcC_{R(\opA)}$ and $\mcC_{R(\opA)}^i$ respectively.
\item
The functor $R\colon\mcACA\ra\mcBCB$ induces a ribbon equivalence $R\colon\mcC_\opA\ra\mcC_{R(\opA)}$ and pivotal equivalences $R\colon\mcC_\opA^i\ra\mcC_{R(\opA)}^i$, $i=1,2$.
\end{enumerate}
\end{defnprp}
\begin{proof}
iii) Since $T^R \cong R^* \otimes_A T \otimes_{A^{\otimes 2}} R^{\otimes 2}$, the $B$-action on $R$ induces a $B$-$B^{\otimes 2}$-bimodule structure on $T^R$.
The braidings in the definition~\eqref{eq:RA_orb_datum} of $\a^R$, $\abar^R$ are such that the identities~\eqref{eq:aabar-A_commute} hold, i.e.\ so that $\a^R\colon T^R \otimes_2 T^R \lra T^R \otimes_1 T^R :\abar^R$ are $B$-$B^{\otimes 3}$-bimodule morphisms as needed.

The identities~\eqrefO{1}--\eqrefO{8} for $R(\opA)$ can be checked by straightforward computations using the same identities for $\opA$, the splitting of the idempotent defining $T^R$ and the identities~\eqref{eq:tr-R_id}.
For example, let us check~\eqrefO{1}.
In this case both sides are $B$-$B^{\otimes 4}$-bimodule morphisms.
Let us simplify them by taking the relative tensor product $\id_R \otimes_B - \otimes_{B^{\otimes 4}} \id_{(R^*)^{\otimes 4}}$ and applying the $A$-$A$-bimodule isomorphism $R \otimes_B R^* \cong A$.
The left-hand side of~\eqrefO{1} for $R(\opA)$ can then be rewritten as the balanced map
\begin{align}\nonumber
&
\pic[1.25]{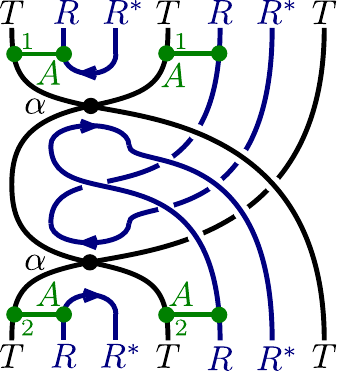} =
\pic[1.25]{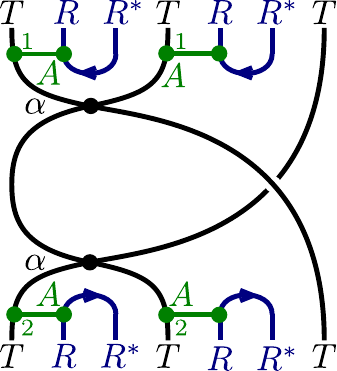}\\ \label{eq:Morita_transport_O1calc}
&\stackrel{\text{\eqrefO{1} for $\opA$}}{=}
\pic[1.25]{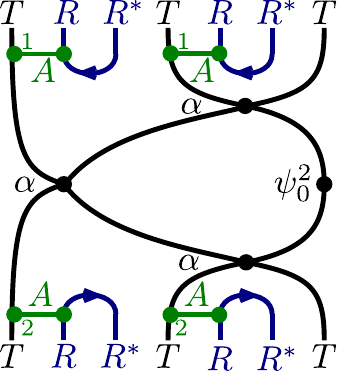} \stackrel{\eqref{eq:tr-R_id}}{=}
\pic[1.25]{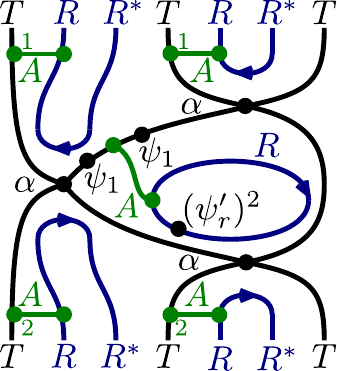}.
\end{align}
One recognises the right-hand side of~\eqref{eq:Morita_transport_O1calc} as the right-hand side of~\eqrefO{1} for $R(\opA)$ (or rather the morphism obtained after the relative tensor product $\id_R \otimes_B - \otimes_{B^{\otimes 4}} \id_{(R^*)^{\otimes 4}}$ as before).
Note that the two $\psi_1$-insertions together with the $A$-line between them constitute the projector onto $T \otimes_1 R$.

The identities~\eqrefT{1}--\eqrefT{7} for $R(M)$, $R(N)$ are shown in the same way as the identities~\eqrefO{1}--\eqrefO{8} for the orbifold datum $R(\opA)$.

iv) The pivotal structures in $\mcC_\opA$, $\mcC_\opA^1$, $\mcC_\opA^2$ are inherited from $\mcACA$ and therefore preserved by the functor $R\colon\mcACA\ra\mcBCB$ (see Proposition~\ref{prp:R_pivotal}) since by definition one has $M^R = R(M)$ as $A$-$A$-bimodules.
Since the functor $R$ on bimodules is fully faithful, so is the induced functor on $\mcC_\opA$, $\mcC_\opA^1$, $\mcC_\opA^2$.
It is also an equivalence, since it has the inverse $R \otimes_B - \otimes_B R^*$.
That $R\colon\mcC_\opA \ra \mcC_{\opA}$ preserves braidings can be shown by a similar computation as in iii) using the expressions~\eqref{eq:CA_br_twist}.
\end{proof}

\begin{rem}
A more natural way to prove the above proposition is the following:
Let $\mcA{}lg_\mcC$ be the monoidal bicategory of algebras in $\mcC$, their bimodules and bimodule morphisms (with the monoidal product being given by the tensor product of algebras in $\mcC$) and let $\mcA{}lg_\mcC(A)$ be the subcategory generated by the objects $A$ and $A^{\opp}$ (the opposite algebra).
Then the identities~\eqrefO{1}--\eqrefO{8}, \eqrefT{1}--\eqrefT{7} can be perceived as identities of 2-morphisms in $\mcA{}lg_\mcC(A)$.
The bicategory $\mcA{}lg_\mcC(A)$ is pivotal in the sense of~\cite[Sec.\,2.2]{Ca}, i.e.\ the 1-morphisms have coinciding left/right adjoints which satisfy the analogous identities as those for duals in a pivotal category.
The isometric Morita module ${_AR_B}$ provides one with an equivalence $R\colon\mcA{}lg_\mcC(A) \xra{\sim} \mcA{}lg_\mcC(B)$ preserving these structures, which is defined analogously as the pivotal equivalence $R\colon\mcACA\ra\mcBCB$ in Proposition~\ref{prp:R_pivotal}.
The 1-morphisms $T^R$, $M^R$ and the 2-morphisms $\a^R$, $\abar^R$, $\tau_i^R$, $\taubar{i}^R$, $i\in\{1,2\}$ in~\eqref{eq:RA_orb_datum}, \eqref{eq:RM_obj} are obtained by mapping $T$, $M$, $\a$, $\abar$, $\tau_i$, $\taubar{i}$ along the functor $R$ and therefore they preserve~\eqrefO{1}--\eqrefO{8}, \eqrefT{1}--\eqrefT{7}.
\end{rem}

\subsection{Transports along ribbon Frobenius functors}
\label{subsec:transports_along_FFs}
The second way to transport an orbifold datum $\opA$ in a MFC $\mcC$ that we will consider is along a ribbon Frobenius functor (see Definitions~\ref{def:FF} and~\ref{def:pivotal_braided_ribb_FFs}) $F\colon\mcC\ra\mcD$ to another MFC $\mcD$.
Although $F$ need not always preserve orbifold data, we will see that it does so under a compatibility assumption which we now explain.

Let $F\colon\mcC\ra\mcD$ be a ribbon Frobenius functor and $\opA=(A,T,\a,\abar,\psi,\phi)$ an orbifold datum in $\mcC$.
Then $F(A)\in\mcD$ is a symmetric Frobenius algebra and $F(T)$ together with the induced $F(A)$-actions is automatically an $F(A)$-$F(A)^{\otimes'2}$-bimodule since one has
\begin{equation}
\pic[1.25]{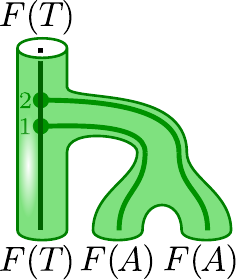} =
\pic[1.25]{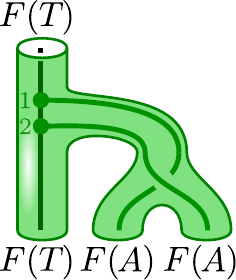} =
\pic[1.25]{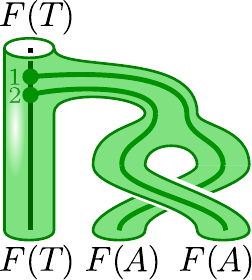} .
\end{equation}
Recall that if $(F,A)$ is strongly separable (see Definition~\ref{def:F-A_strongly_sep_full}), one can fix a section $\psi_F\colon\opid'\ra F(A)$ which serves as a section to the Frobenius algebra $F(A)$ and without loss of generality can be assumed to be invertible.
\begin{defn}
We call a ribbon Frobenius functor $F\colon\mcC\ra\mcD$ between two MFCs $\mcC$, $\mcD$ \textit{compatible with an orbifold datum $\opA = (A,T,\a,\abar,\psi,\phi)$} (or simply \textit{$(F,\opA)$ compatible}) if 
\begin{enumerate}[i)]
\item $(F,A)$ is strongly separable with an invertible section $\psi_F\colon\opid'\ra F(A)$;
\item there is a scalar $\phi_F\in\opk^\times$ such that the identity~\eqrefO{8} holds for $(F(A),F(T),$ $\psi_F,\phi_F)$, i.e.\ one has
\begin{equation}
\label{eq:F-A_compatible_cond}
\hspace{-32pt}
\pic[1.25]{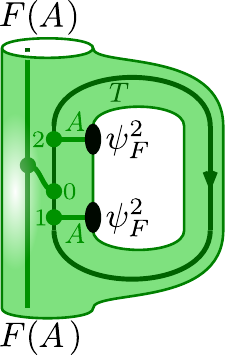} =
\pic[1.25]{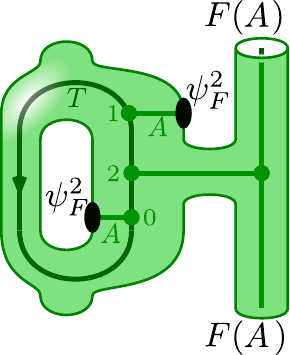} =
\pic[1.25]{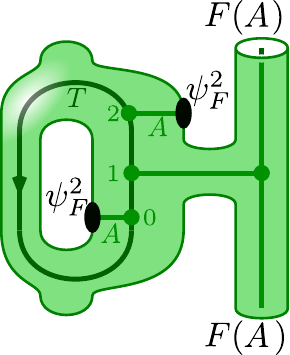} =
\phi_F^{-2} \cdot \id_{F(A)} \, .
\end{equation}
\end{enumerate}
Furthermore, we call $(F,\opA)$ \textit{full} if $(F,A)$ is full.
\end{defn}
\goodbreak
\begin{rem}
\label{rem:F-A_special}
In our main example below, the section $\psi_F\colon\opid'\ra F(A)$ will have the form $\xi \cdot [\opid'\xra{F_0} F(\opid) \xra{F(\psi)} F(A)]$, $\xi\in\opk^\times$.
In this case the condition ii) in Definition~\ref{eq:F-A_compatible_cond} holds automatically with $\phi_F = \phi / \xi$.
\end{rem}

The transport of orbifold data along a ribbon Frobenius functor is then defined in an obvious way:
\begin{defnprp}
\label{defnprp:FA_orb_datum}
Let $F\colon\mcC\ra\mcD$ be a ribbon Frobenius functor between two MFCs $\mcC$ and $\mcD$, $\opA=(A,T,\a,\abar,\psi,\phi)$ an orbifold datum in $\mcC$ and suppose $(F,\opA)$ is compatible with an invertible section $\psi_F\colon\opid'\ra F(A)$ of $F$ and the scalar $\phi_F\in\opk^\times$ as in~\eqref{eq:F-A_compatible_cond}.
\begin{enumerate}[i)]
\item
The \textit{transport of $\opA$ along $F$} is the tuple $F(\opA) := (F(A), F(T), \a^F, \abar^F, \psi_F, \phi_F)$, where
\begin{equation}
\label{eq:FA_orb_datum}
\a^F    := \pic[1.25]{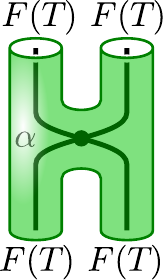} \, , \quad 
\abar^F := \pic[1.25]{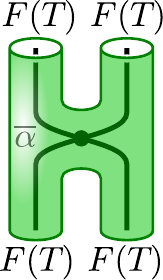} \, .
\end{equation}
\item
The \textit{transport of an object $M=(M,\tau_1,\tau_2,\taubar{1},\taubar{2})\in\mcC_\opA$} is the tuple $F^\opA(M) := (F(M), \tau_1^F, \tau_2^F, \taubar{1}^F, \taubar{2}^F)$ where
\begin{equation}
\label{eq:FM_obj}
\tau_i^F     := \pic[1.25]{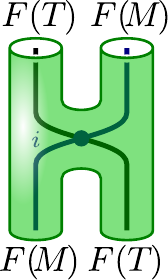} \, , \quad 
\taubar{i}^F := \pic[1.25]{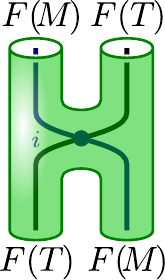} \, , \quad i\in\{1,2\} \, .
\end{equation}
Similarly one defines the transport $F^{\opA,i}(N)$ of an object $N\in\mcC_\opA^i$, $i\in\{1,2\}$.
\item
$F(\opA)$ is an orbifold datum in $\mcD$ and $F(M)$, $F(N)$ for $M\in\mcC_\opA$, $N\in\mcC_\opA^i$, $i\in\{1,2\}$ are objects of $\mcD_{F(\opA)}$ and $\mcD_{F(\opA)}^i$ respectively.
\item
The pivotal functor $F^A\colon\mcACA\ra{_{F(A)}\mcD_{F(A)}}$ induces a ribbon functor $F^\opA\colon\mcC_\opA\ra\mcD_{F(\opA)}$ and pivotal functors $F^{\opA,i}\colon\mcC_\opA^i\ra\mcD_{F(\opA)}^i$, $i\in\{1,2\}$.
If in addition $(F,\opA)$ is full, the functors $F^\opA$ are full embeddings.
\end{enumerate}
\end{defnprp}

The proof below essentially argues that for an orbifold datum $\opA=(A,T,\a,$ $\abar,\psi,\phi)$ in a MFC $\mcC$ and a ribbon Frobenius functor $F\colon\mcC\ra\mcD$ to another MFC $\mcD$ which is separable with respect to $A$, the identities \eqrefO{1}--\eqrefO{7} for $F(\opA)$ hold automatically, and \eqrefO{8} is imposed by the compatibility condition~\eqref{eq:F-A_compatible_cond}.
 The need to treat~\eqrefO{8} separately can be explained as follows: $F$ does not preserve tensor products of algebras, i.e.\ in general one has $F(A^{\otimes 2}) \ncong F(A)^{\otimes' 2}$ and the identity~\eqrefO{8} is the only one in where the tensor product relative to $A^{\otimes 2}$ arises.
\begin{proof}
We have already seen that $(F(A), \psi_F)$ is a symmetric separable Frobenius algebra in $\mcD$ and $F(T)$ is a $F(A)$-$F(A)^{\otimes' 2}$-bimodule.
It remains to show that the morphisms in~\eqref{eq:FA_orb_datum}, \eqref{eq:FM_obj} are (balanced maps defining) $F(A)$-$F(A)^{\otimes' 3}$ and $F(A)$-$F(A)^{\otimes' 2}$-bimodule morphisms $\a^F\colon F(T) \otimes_2 F(T) \lra F(T) \otimes_1 F(T) : \abar^F$ and $\tau_i^F\colon F(M) \otimes_0' F(T) \lra F(T) \otimes_i' F(M)$, $i\in\{1,2\}$, and that the tuples $F(\opA)$ and $F^\opA(M)$ satisfy the identities~\eqrefO{1}--\eqrefO{8} and~\eqrefT{1}--\eqrefT{7} respectively.
Apart from~\eqrefO{8}, which holds automatically since $(F,\opA)$ is compatible, all of this can be checked by straightforward computations using the same identities for $\opA$ and $M\in\mcC_\opA$, the identities~\eqref{eq:F_graph_calc:assoc-unitality}--\eqref{eq:F_graph_calc:twist} and the condition for strong separability~\eqref{eq:F_strong_sep_cond}.
For example~\eqrefO{1} is shown as follows:
\begin{align} \nonumber
&
\pic[1.25]{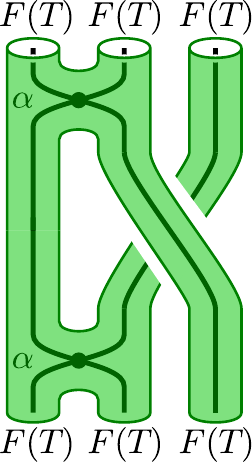} =
\pic[1.25]{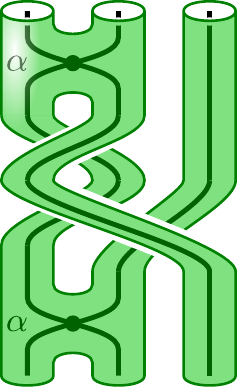} \stackrel{\eqref{eq:F_graph_calc:braided}}{=}
\pic[1.25]{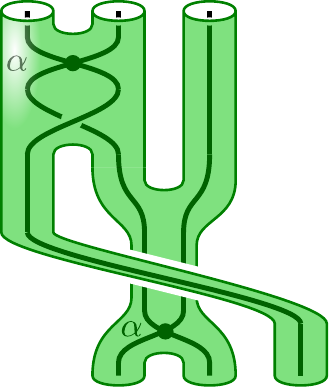}\\ \nonumber
&=
\pic[1.25]{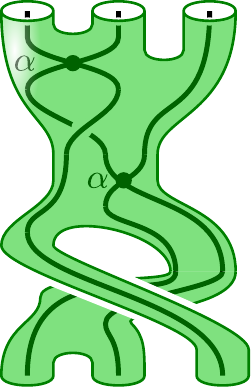} \stackrel{\eqref{eq:F_graph_calc:braided}}{=}
\pic[1.25]{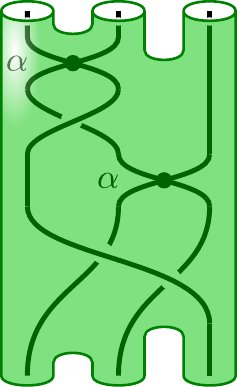} =
\pic[1.25]{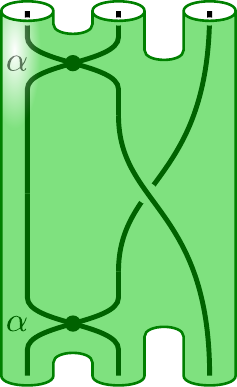}\\
& \hspace{-20pt} \stackrel{\eqrefO{1}}{=}
\pic[1.25]{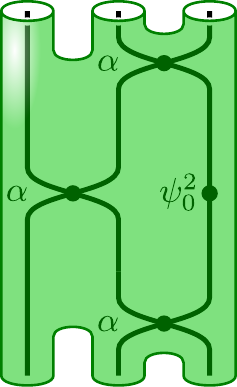}
\stackrel{\eqref{eq:F_strong_sep_cond}}{=}
\pic[1.25]{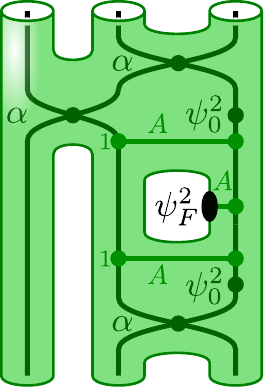} =
\pic[1.25]{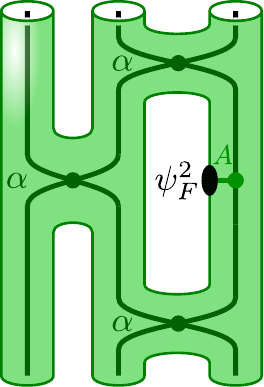} .
\end{align}
We note that in order to show~\eqrefO{4}--\eqrefO{7} and~\eqrefT{6}, \eqrefT{7} for $F(\opA)$ one needs $F$ to be pivotal, hence ribbon.

iv) Since the monoidal and pivotal structures in $\mcC_\opA$, $\mcD_{F(\opA)}$ and $\mcC_\opA^i$, $\mcD_{F(\opA)}^i$, $i\in\{1,2\}$ are inherited from $\mcACA$, ${_{F(A)}\mcD_{F(A)}}$, the induced functors are pivotal as well.
To show that $F^\opA\colon\mcC_\opA\ra\mcD_{F(\opA)}$ is ribbon, i.e.\ preserves the braiding, we perform the following auxiliary computation for arbitrary $M,N\in\mcC_\opA$:
\begin{align} \nonumber
&
\pic[1.25]{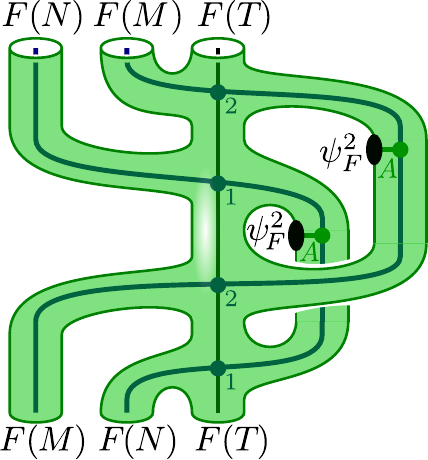} &&\stackrel{\eqref{eq:F_graph_calc:cobraided}}{=}
\pic[1.25]{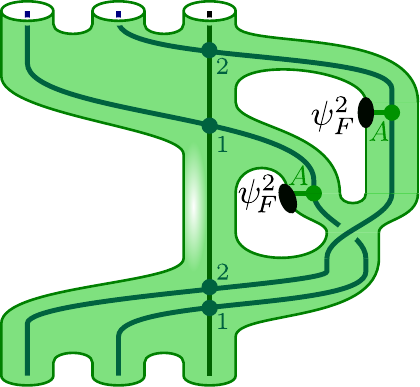}\\ \label{eq:F_braiding_calc}
& \stackrel{\eqref{eq:F_strong_sep_cond}}{=}
\pic[1.25]{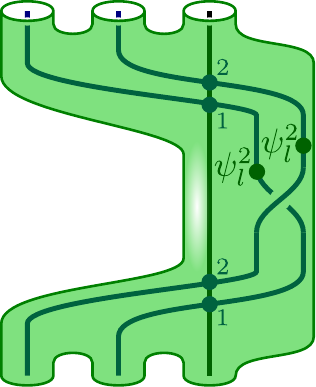} && \stackrel{(*)}{=}
\pic[1.25]{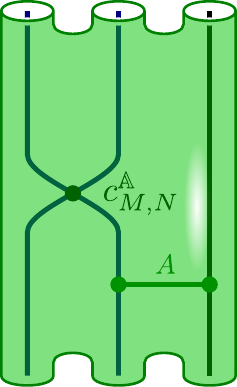} \, ,
\end{align}
where in the step $(*)$ we have used the identity in~\cite[Lem.\,3.4]{MR1}.
Precomposing both sides of~\eqref{eq:F_braiding_calc} with $\phi_F^2 \cdot (\id_{F(N)} \otimes' \id_{F(M)} \otimes' ((\psi_F)_1^2 \circ (\psi_F)_2^2))$ and taking the partial right trace with respect to $F(T)$, the left-hand side yields the expression~\eqref{eq:CA_br_twist} for the braiding of $F(M), F(N)\in\mcD_{F(\opA)}$, while the right-hand side yields (the balanced map corresponding to) $F^\opA(c^\opA_{M,N})$, since by the virtue of $(F,\opA)$ being compatible, the identity~\eqref{eq:F-A_compatible_cond} holds.

Finally, the proof that the functors $F^{\opA,(i)}\colon\mcC^{(i)}_\opA\ra\mcD^{(i)}_{F(\opA)}$, $i\in\{1,2\}$ are full embeddings is the same as that of Proposition~\ref{prp:F-A_pivotal_embed}.
\end{proof}

\begin{rem}
Our main example in Section~\ref{subsec:inv_orb_datum} deals with a ribbon Frobenius functor $F\colon\mcC\ra\mcD$ which is compatible but not full with respect to an orbifold datum $\opA$ in $\mcC$.
In this case the functors $F^{\opA,i}$, $i\in\{1,2\}$ do not need to be surjective on the spaces of morphisms.
However the functor $F^\opA$ still tends to be surjective, in fact if both $\opA$ and $F(\opA)$ are simple orbifold data, $F^\opA\colon\mcC_\opA\ra\mcD_{F(\opA)}$ is automatically fully faithful (see~\cite[Cor.\,3.26]{DMNO}).
We also note that $F(\opA)$ might fail to be simple even if $\opA$ is (take for example the diagonal functor $\mcC \ra \mcC \oplus \mcC$).
\end{rem}

\section{Condensation inversion}
\label{sec:condensation_inversion}
Given a commutative haploid symmetric $\D$-separable Frobenius algebra $B$ (to be called a condensable algebra) in a modular fusion category (MFC) $\mcC$ the category $\mcC_B^\loc$ of its local (or dyslectic) modules is again a MFC~\cite{KO} (to be called the $B$-condensation of $\mcC$).
It was shown in~\cite[Sec.\,4.1]{MR1} that condensations of $\mcC$ are examples of MFCs associated to orbifold data in $\mcC$.
In this section we explore the other direction of this construction, i.e.\ we find an orbifold datum $\opA$ in $\mcC_B^\loc$ whose associated MFC $(\mcC_B^\loc)_\opA$ gives back $\mcC$.

\subsection{Condensable algebras in MFCs}
\label{subsec:condensable_algs}
Let $B$ be a commutative algebra in a braided category $\mcC$ (see~\eqref{eq:comm_algs}).
One can equip a right $B$-module $M$ with two left actions
\begin{equation}
\label{eq:local_mod_cond}
\rho_M^+ := \pic[1.25]{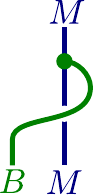}\, , \quad
\rho_M^- := \pic[1.25]{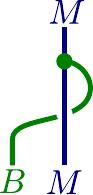} .
\end{equation}
This gives two full embeddings $\mcC_B\ra\mcBCB$ whose images we denote by $\mcC_B^+$ and $\mcC_B^-$ respectively.
For an arbitrary bimodule $N\in\mcBCB$ the morphisms
\begin{equation}
p^+_N := \pic[1.25]{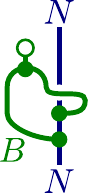}\, , \quad
p^-_N := \pic[1.25]{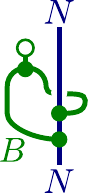}.
\end{equation}
are idempotents which (if they are split) project onto subobjects $N^\pm\in\mcC_B^+$, $N^-\in\mcC_B^\pm$.
Indeed, one has for example:
\begin{equation}
\pic[1.25]{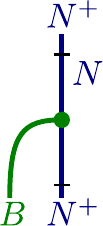} =
\pic[1.25]{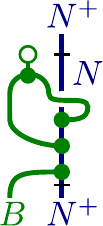} =
\pic[1.25]{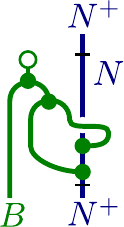} =
\pic[1.25]{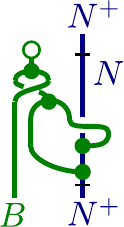} =
\pic[1.25]{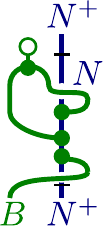} =
\pic[1.25]{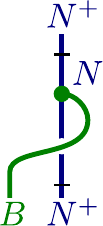} .
\end{equation}
i.e.\ the left action on $N^+$ is exactly $\rho^+_{N^+}$.

\begin{defn}
A (right) $B$-module $M\in\mcC_B$ is called \textit{local} (or \textit{dyslectic}) if one has $\rho_M^+ = \rho_M^-$.
We denote the full subcategory of local modules in $\mcC_B$ by $\mcC_B^\loc$ and identify it with its image in $\mcBCB$.
\end{defn}
For a local module $M\in\mcC_B^\loc$ one has by definition
\begin{equation}
\pic[1.25]{61_rho-plus.pdf} =
\pic[1.25]{61_rho-minus.pdf} =
\pic[1.25]{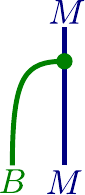}\quad , \qquad
\pic[1.25]{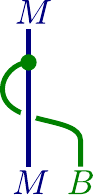} =
\pic[1.25]{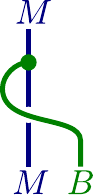} =
\pic[1.25]{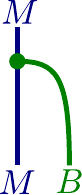}\, .
\end{equation}
Furthermore, for an arbitrary $N\in\mcBCB$ the morphism
\begin{equation}
\label{eq:pz-N_projector}
p^\circ_N := p^-_N\circ p^+_N = p^+_N\circ p^-_N
\end{equation}
is an idempotent projecting onto a subobject $N^\circ$ of $N$ in $\mcBCB$, which is a local module and will be called the \textit{localisation} of $N$.

\begin{defn}
\label{def:condensable_algebra}
A \textit{condensable algebra} in a MFC $\mcC$ is a commutative haploid (i.e.\ $\dim\mcC(\opid,B)=1$) symmetric $\D$-separable Frobenius algebra $B\in\mcC$.
We refer to the category of local modules $\mcC_B^\loc$ as the \textit{$B$-condensation} of $\mcC$.
\end{defn}
\begin{rem}
The term `condensable algebra' is borrowed from applications of fusion categories to condensed matter physics, see e.g.\ \cite{Bur}, \cite[Def.\ 2.6]{Ko}, \cite[Ex.\ 3.2.4]{CZW}.
Sometimes it is also used to refer to an \textit{\'etale algebra} in a non-degenerate braided fusion category, which is a haploid (or connected) commutative separable algebra, i.e.\ without an assigned Frobenius structure (see e.g.\ \cite[Def.\ 3.1 \& Ex.\ 3.3ii)]{DMNO}).
Although it is related to, it should not be considered as a synonym to the more general higher categorical notion of a `condensation algebra' \cite{GJ}.
\end{rem}

The $B$-condensation $\mcC_B^\loc$ is a finitely semisimple ribbon category, with monoidal and pivotal structures inherited from $\mcBCB$, the braiding morphisms $c^\circ_{M,N}$ for all objects $M,N\in\mcC_B^\loc$ are defined by
\begin{equation}
\label{eq:CBloc_braiding}
c^\circ_{M,N}        := \pic[1.25]{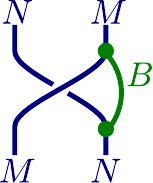} \, , \quad
(c^\circ_{M,N})^{-1} := \pic[1.25]{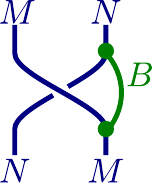} \, ,
\end{equation}
and the twist $\theta^\circ_M$ of an arbitrary object $M\in\mcC_B^\loc$ being equal to the twist of the underlying object in $\mcC$, i.e.\ $\theta^\circ_M := \theta_M$ (note that one automatically has $\theta_B = \id_B$ since $B$ is the tensor unit in $\mcC_B^\loc$).
It was proven in~\cite{KO} and~\cite[Cor.\,3.30]{DMNO} that the $B$-condensation $\mcC_B^\loc$ is in fact a MFC.
Its global dimension is given by
\begin{equation}
\Dim \mcC_B^\loc = \frac{\Dim\mcC}{(\dim_\mcC B)^2} \, ,
\end{equation}
where the categorical dimension $\dim_\mcC B$ of the condensable algebra $B\in\mcC$ is automatically non-zero, see e.g.~\cite[Lem.\,2.7]{KMRS}.
\begin{rem}
\begin{enumerate}[i)]
\item
The $B$-labelled strands in the graphical calculus of $\mcC$ will be regarded as unframed and undirected as $B$ has a trivial twist and a canonical isomorphism $B\cong B^*$.
\item 
Note that altering Definition~\ref{def:condensable_algebra} to require a condensable algebra $B\in\mcC$ to be separable instead of $\D$-separable does not yield different categories $\mcC_B^\loc$, $\mcC_B^+$, $\mcC_B^-$.
Indeed, since $B$ is haploid, any other Frobenius structure on $B$ results in a rescaling $B_\zeta$, $\zeta\in\opk^\times$ so that $\mcBCB$ and ${_{B_\zeta}\mcC_{B_\zeta}}$ (along with the corresponding subcategories) are pivotal-equivalent, see Proposition~\ref{prp:R_pivotal} and Remark~\ref{rem:R_pivotal_nuisances}.
\end{enumerate}
\end{rem}

Using the results of Section~\ref{subsec:transports_along_FFs} one can quickly relate the orbifold data in the condensation $\mcC_B^\loc$ and those in the original MFC $\mcC$.
\begin{prp}
\label{prp:UA_orb_datum}
Let $\opA$ be a candidate orbifold datum in $\mcC_B^\loc$ (i.e.\ a tuple $(A,T,\a,\abar$, $\psi,\phi)$ as in Definition~\ref{def:orb_datum} which does not a priori satisfy~\eqrefO{1}--\eqrefO{8}).
Then
\begin{enumerate}[i)]
\item The forgetful functor $U\colon\mcC_B^\loc\ra\mcC$ is a ribbon Frobenius functor.
\item $\opA$ is an orbifold datum in $\mcC_B^\loc$ if and only if $U(\opA)$ is an orbifold datum in $\mcC$.
\item The induced functor $U^\opA\colon (\mcC_B^\loc)_{\opA} \ra \mcC_{U(\opA)}$ is an equivalence.
\end{enumerate}
\end{prp}
\begin{proof}
i) We know from Example~\ref{eg:FF_forgetful} that the forgetful functor $U\colon\mcBCB\ra\mcC$ is a pivotal Frobenius functor.
That its restriction to $\mcC_B^\loc$ is also ribbon follows from the Definition~\ref{eq:CBloc_braiding} of the braidings in $\mcC_B^\loc$ as shown by the following calculation:
\begin{equation}
\pic[1.25]{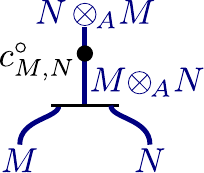} =
\pic[1.25]{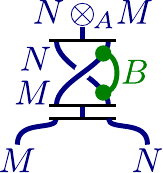} =
\pic[1.25]{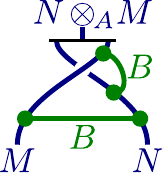} =
\pic[1.25]{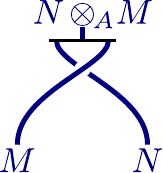} .
\end{equation}

ii) Let $\opA$ be an orbifold datum.
From the argument in Example~\ref{eg:FF_forgetful} follows that $U$ is strongly separable and full with respect to the symmetric separable Frobenius algebra $(A,\psi)$ and has the section $\psi_U = [\opid\xra{\eta=U_0}B\xra{\psi}A]$.
By Remark~\ref{rem:F-A_special} $(U,\opA)$ is compatible, so that $U(\opA)$ is an orbifold datum.

Conversely, if $U(\opA)$ is an orbifold datum, the identities~\eqrefO{1}--\eqrefO{8} for $U(\opA)$ correspond to the same identities for $\opA$ written in terms of $B$-balanced maps.

iii) By Proposition~\ref{prp:ssi_funct_equiv}, we need to show that $U^\opA$ is surjective, i.e.\ that every object of $\mcC_{U(\opA)}$ is a subobject of $U^\opA(M)$ for some $M\in(\mcC_B^\loc)_\opA$.
By Remark~\ref{prp:pipe_objs_generate}, any object of $\mcC_{U(\opA)}$ is a subobject of $P_{U(\opA)}(N)$ where $N$ is a $U(A)$-$U(A)$-bimodule and $P_{U(\opA)}\colon{_{U(A)}\mcC_{U(A)}} \ra \mcC_{U(\opA)}$ is the pipe functor.
A $U(A)$-$U(A)$-bimodule $N\in\mcC$ with the induced $B$-actions in general need not be a local $B$-module, however in case it is, the definition~\eqref{eq:pipe_funct_def} for $P_{U(\opA)}(N)$ can be read as the $B$-balanced map in $\mcC$ defining $P_\opA(N)$ where $P_\opA\colon{_A(\mcC_B^\loc)_A\ra (\mcC_B^\loc)_\opA}$ is the pipe functor for the orbifold datum $\opA$, i.e.\ one has $U^\opA\circ P_\opA (N) \cong P_{U(\opA)}(U(N))$.

We claim that for arbitrary $N\in{_{U(A)}\mcC_{U(A)}}$ one has $P_{U(\opA)}(N) \cong P_{U(\opA)}(U(N^\circ))$, where $N^\circ$ denotes localisation of $N$ with respect to the induced $B$-actions.
Indeed one has:
\begin{align}
&
\pic[1.25]{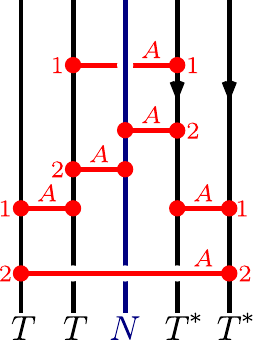} =
\pic[1.25]{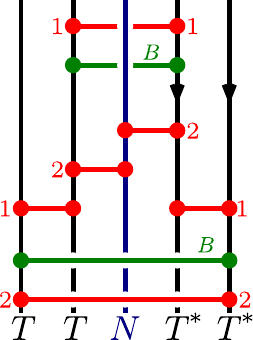} =
\pic[1.25]{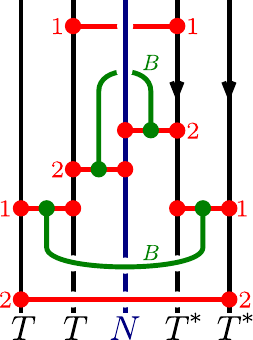}\\
&=
\pic[1.25]{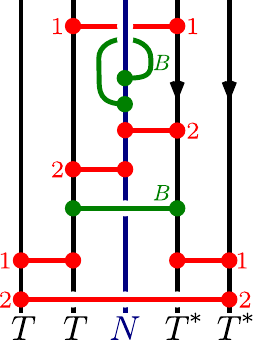} =
\pic[1.25]{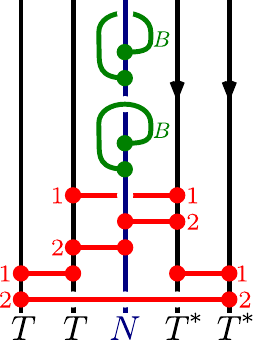},
\end{align}
where we have used that the actions of $U(A)$ on $U(T)$ are $B$-balanced and $U(A)$, $U(T)$ are already local $B$-modules.
On the right-hand side one recognises the idempotent $p^\circ_N$ projecting onto $N^\loc$.
Pre- and postcomposing the both sides with the $\psi_U$-insertions results in the idempotent defining $P_{U(\opA)}$ as needed.

We have shown that the pipe object $P_{U(\opA)}(N)$ for an arbitrary $N\in{}_{U(A)}\mcC_{U(A)}$ is isomorphic to $U^\opA\circ P_\opA(N^\circ)$, i.e.\ in the essential image of $U^\opA$, so that $U^\opA$ is surjective as needed.
\end{proof}
\begin{rem}
\label{rem:condensation_orb_datum}
It was shown in~\cite[Prop.\,3.15]{CRS3} that a condensable algebra $B\in\mcC$ yields an orbifold datum
\begin{equation}
\opB = (B,~ {_BB_{BB}},~ \a=\abar=\pic[1.25]{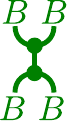} ,~ \psi=\eta,~ \phi=1) \, ,
\end{equation}
where the commutativity of $B$ implies the relations~\eqref{eq:right_T_actions_commute}, \eqref{eq:aabar-A_commute} and hence that $B$ can be treated as a $B$-$B^{\otimes 2}$-bimodule and $\a$, $\abar$ as $B$-$B^{\otimes 3}$-bimodule morphisms.
Furthermore in~\cite[Thm.\,4.1]{MR1} it was shown that the associated MFC $\mcC_\opB$ is equivalent to the condensation $\mcC_B^\loc$.
Both of these results follow from Proposition~\ref{prp:UA_orb_datum}: If $\opid$ is the trivial orbifold datum in $\mcC_B^\loc$, one checks that $\opB = U(\opid)$ and consequently $(\mcC_B^\loc)_{\opid} \simeq \mcC_B^\loc \simeq \mcC_\opB$.
We note that a similar equivalence as in~\ref{prp:UA_orb_datum}iii) does not in general apply for the embeddings $U^{\opA,i}\colon (\mcC_B^\circ)_\opA^{i} \ra \mcC_{U(\opA)}^{i}$, $i\in\{1,2\}$:
It can be inferred from the proof of~\cite[Thm.\,4.1]{MR1} that $\mcC_\opB^1 \simeq \mcC_B^+$, $\mcC_\opB^2 \simeq \mcC_B^-$ as pivotal categories, however one obviously has $(\mcC_B^\loc)_\opid^1 \simeq (\mcC_B^\loc)_\opid^2 \simeq \mcC_B^\loc$.
\end{rem}

\subsection{Inversion orbifold datum}
\label{subsec:inv_orb_datum}
One of the main results of this paper is
\begin{thm}
\label{thm:inv_orb_datum}
Let $\mcC$ be a MFC, $\opA$ an orbifold datum in $\mcC$ and $B\in\mcC$ a condensable algebra.
Then there is an orbifold datum in the condensation $\mcC_B^\loc$ whose associated braided category is equivalent as ribbon multifusion category to $\mcC_\opA$.
\end{thm}
Specialising this statement to $\opA=\opid$ (the trivial orbifold datum in $\mcC$) one immediately gets
\begin{cor}
\label{cor:cond_inv_orb_datum}
There is a simple orbifold datum in the condensation $\mcC_B^\loc$ whose associated MFC is equivalent to $\mcC$.
\end{cor}

In the remainder of this section we prove Theorem~\ref{thm:inv_orb_datum} by constructing the orbifold datum in its statement and then write it out explicitly for the case in Corollary~\ref{cor:cond_inv_orb_datum}.
This is done in several steps where we
\begin{enumerate}[(1)]
\item
introduce a certain ribbon Frobenius functor $\Iz\colon \mcC \ra \mcC_B^\circ$;
\item
show that for an arbitrary orbifold datum $\opA$ in $\mcC$, $\Iz$ is compatible with respect to a certain Morita transport $\opA_C = R_C(\opA)$;
this automatically yields an orbifold datum $\Iz(\opA_C)$ in $\mcC_B^\circ$;
\item
show that the functor $\mcC_{\opA_C} \ra (\mcC_B^\loc)_{\Iz(\opA_C)}$ as in Proposition~\ref{defnprp:FA_orb_datum}iv) is an equivalence so that one has $\mcC_{\opA} \simeq \mcC_{\opA_C} \simeq (\mcC_B^\loc)_{\Iz(\opA_C)}$;
\item
unpack the definitions in for the case of trivial orbifold datum $\opA = \opid$.
\end{enumerate}

\medskip

\subsubsection*{Step (1)}
The functor $\Iz = \Iz_B\colon \mcC \ra \mcC_B^\loc$ is defined on an object $X\in\mcC$ and a morphism $[f\colon X \ra Y]\in\mcC$ by
\begin{equation}
\label{eq:Iz_funct_def}
\Iz(X) = \im \pic[1.25]{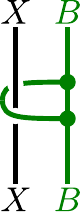} \, , \quad
\Iz(f) = \pic[1.25]{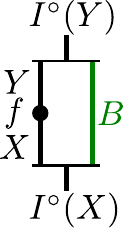} \, .
\end{equation}
Using the commutativity of $B$ and the isomorphisms $B\otimes_B B\cong B$ one can check that $\Iz(X) \cong \im p^\circ_{B \otimes X \otimes B}$, i.e.\ $\Iz(X)$ is obtained by localising the induced bimodule $B\otimes X \otimes B \in \mcBCB$ (cf.\ Example~\ref{eg:FF_EAl_functor} and~\cite[Sec.\,4]{FFRS}).
\begin{lem}
$\Iz\colon\mcC\ra\mcC_B^\loc$ is a ribbon Frobenius functor with the structure morphisms given by $\Iz_0 := \id_B$, $\overline{\Iz_0} := \id_B$ and the $B$-balanced maps
\begin{equation}
\label{eq:Iz2_morphisms}
\Iz_2(X,Y) := \pic[1.25]{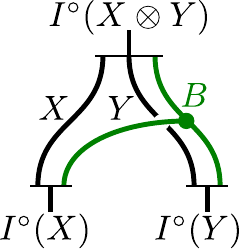} \, , \quad
\overline{\Iz_2}(X,Y) := \pic[1.25]{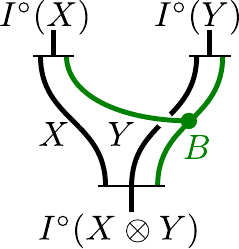} \, .
\end{equation}
\end{lem}
\begin{proof}
Since $\mcC_B^\loc$ inherits the braidings and the twists from $\mcC$, this is shown analogously as for the ribbon Frobenius functor $E^l_A\colon\mcC\ra\mcC$ (see Example~\ref{eg:FF_EAl_functor}).
In fact, $\Iz$ is basically the specialisation of $E^l_A$ to the symmetric separable Frobenius algebra $(A,\psi)=(B,\eta)$.
\end{proof}

Specific to the functor $\Iz$ are the two natural transformations $\rho\colon \Iz(- \otimes B)\Lra \Iz(-) :\gamma$, which for $X\in\mcC$ are defined/denoted by
\begin{equation}
\label{eq:rho-gamma_morphs}
\rho_X   = \pic[1.25]{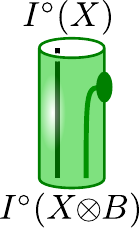}   := \pic[1.25]{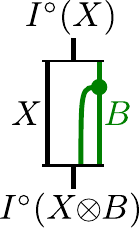} \, , \quad
\gamma_X = \pic[1.25]{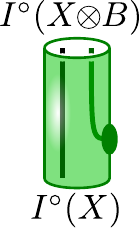} := \pic[1.25]{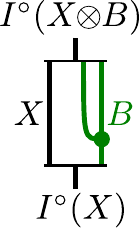} \, .
\end{equation}
It follows immediately that one has the identities
\begin{equation}
\label{eq:rho-gamma_identities}
\pic[1.25]{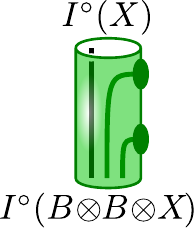} \hspace{-4pt}=\hspace{-4pt}
\pic[1.25]{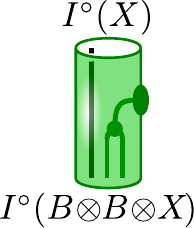}, \quad
\pic[1.25]{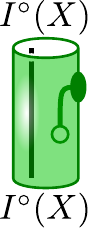} \hspace{-4pt}= \id_{\Iz(X)} \, , 
\pic[1.25]{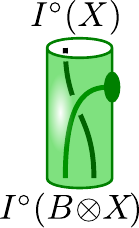} \hspace{-4pt}=\hspace{-4pt}
\pic[1.25]{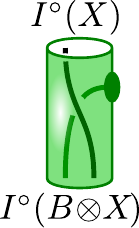},
\end{equation}
(similar identities also hold for $\gamma_X$, comultiplication and counit of $B$).
Using the structure morphisms~\eqref{eq:Iz2_morphisms}, for all $X,Y\in\mcC$ one also gets:
\begin{equation}
\label{eq:Iz_non_separable}
\pic[1.25]{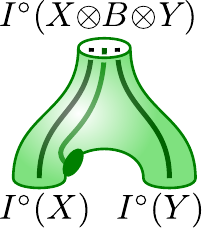} \hspace{-4pt}=\hspace{-4pt} \pic[1.25]{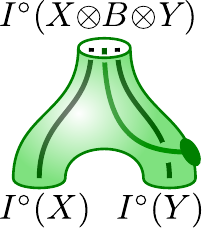} \, , \quad
\pic[1.25]{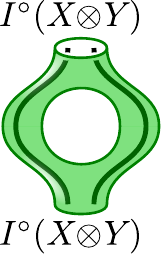} \hspace{-4pt}=\hspace{-4pt} \pic[1.25]{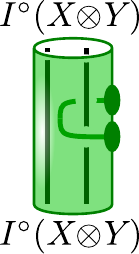} \, .
\end{equation}
Using the definitions~\eqref{eq:rho-gamma_morphs} it is easy to see that the right-hand side of the second identity in~\eqref{eq:Iz_non_separable} is an idempotent, projecting onto $\Iz(X)\otimes_B\Iz(Y)$.
In particular, $\Iz$ need not be separable, since $\Iz(X)\otimes_B\Iz(Y)$ is in general a proper subobject of $\Iz(X\otimes Y)$ in $\mcC_B^\loc$ and so the morphism $\Iz_2(X,Y)$ cannot have a left inverse.
$\overline{\Iz_2}(X,Y)$ is however a right inverse of $\Iz_2(X,Y)$, i.e.\ one has
\begin{equation}
\label{eq:Iz2_inclusion}
\pic[1.25]{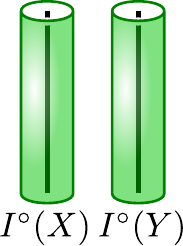} =
\pic[1.25]{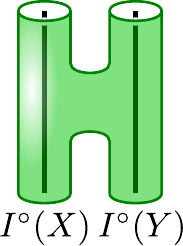} \, .
\end{equation}
This follows from $\Iz(\opid) = B$ being the tensor unit in $\mcC_B^\loc$ and acting on the objects in the image of $\Iz$ by the unitor morphisms (recall that outside the cylinders depicting $\Iz$ in~\eqref{eq:Iz2_inclusion} the graphical calculus is of the category~$\mcC_B^\loc$ and so the tensor unit and the unitor morphisms need not be displayed).

Finally, from definitions~\eqref{eq:Iz_funct_def} and~\eqref{eq:rho-gamma_morphs} follows that for an arbitrary bimodule $M\in\mcC_B^\loc$ one has
\begin{equation}
\label{eq:Iz_simplify}
\im \pic[1.25]{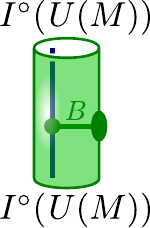}  \cong M  \, , \qquad
\im \pic[1.25]{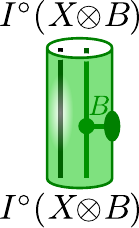} \cong \Iz(X) \, .
\end{equation}
In particular, $M$ is a subobject of $\Iz(U(M))$ in $\mcC_B^\loc$, where $U\colon\mcC_B^\loc\ra\mcC$ is the forgetful functor.

\subsubsection*{Step (2)}

It is easy to find a family of symmetric separable Frobenius algebras with respect to which the functor $\Iz$ is strongly separable.
For that we note that since $\dim\mcC(\opid,B)=1$ and $\vareps\circ\eta = \dim_\mcC B$, the scissors identity~\eqref{eq:scissors_id} implies
\begin{equation}
\label{eq:scissors_on_B}
\pic[1.25]{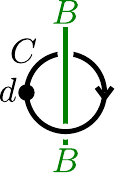} = \frac{\Dim\mcC}{\dim_\mcC B} \cdot
\pic[1.25]{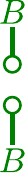} \, .
\end{equation}
Let $(A,\psi)$ be an arbitrary symmetric separable Frobenius algebra in $\mcC$ and let $(A_C, \psi_C)$ be the algebra
\begin{equation}
A_C := C^* \otimes A \otimes C \, , \quad \psi_C := \frac{1}{(\Dim\mcC)^{1/2}} \cdot \pic[1.25]{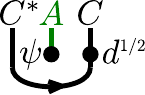} \, .
\end{equation}
and the structure morphisms are as in~\eqref{eq:XAX_alg}.
Recall that according to Example~\ref{eg:RX_Morita_mod} the bimodule $R_C := {_A}(A \otimes C){_{A_C}}$ is an isometric Morita module.

\begin{lem}
For an arbitrary module $L_C\in{_{A_C}\mcC}$ the following identity holds:
\begin{equation}
\label{eq:scissors_on_B_ito_algs}
\pic[1.25]{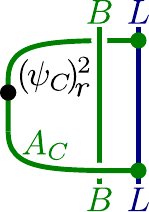} = \frac{1}{\dim_\mcC B} \cdot
\pic[1.25]{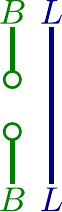} \, .
\end{equation}
\end{lem}
\begin{proof}
As $R_C$ is a Morita module, one can without loss of generality take $L_C = C^* \otimes L$ for some $L\in{_A\mcC}$.
One then has:
\begin{equation}
\frac{1}{\Dim\mcC} \cdot
\pic[1.25]{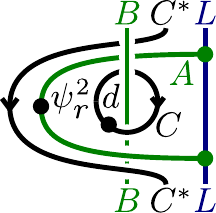} \stackrel{\eqref{eq:scissors_on_B}}{=} \frac{1}{\dim_\mcC B} \cdot
\pic[1.25]{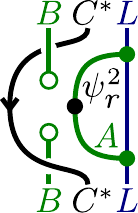} = \frac{1}{\dim_\mcC B} \cdot
\pic[1.25]{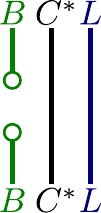}
\end{equation}
as needed.
\end{proof}

\begin{cor}
The ribbon Frobenius functor $\Iz$ is strongly separable with respect to $(A_C,\psi_C)$.
Moreover, the algebra $\Iz(A_C)$ has a section of the form 
\begin{equation}
\label{eq:psi-Iz}
\psi_{\Iz} := (\dim_\mcC B)^{1/2} \cdot \Iz(\psi_C) \, .    
\end{equation}
\end{cor}
\begin{proof}
For all modules $K_C\in{\mcC_{A_C}}$ and $L_C\in{_{A_C}\mcC}$ one computes
\begin{equation}
\label{eq:Iz-ACA_sep_calc}
\pic[1.25]{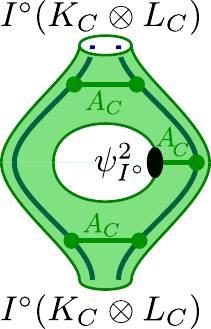} = d_B
\pic[1.25]{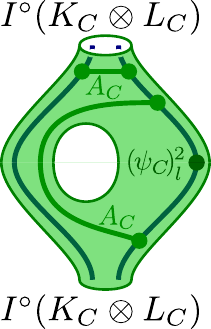} \stackrel{\eqref{eq:Iz_non_separable}}{=} d_B
\pic[1.25]{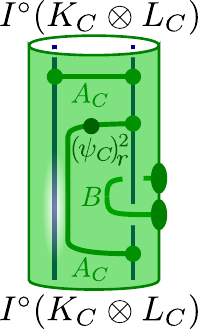} \hspace{-4pt}\stackrel{\eqref{eq:scissors_on_B_ito_algs}}{=}\hspace{-4pt}
\pic[1.25]{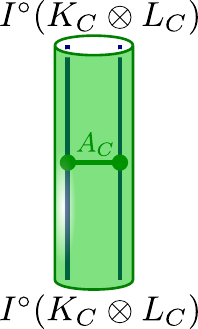} ,
\end{equation}
where we have denoted $d_B = \dim_\mcC B$ for brevity.
The condition~\eqref{eq:F_strong_sep_cond} therefore holds for $(\Iz_0, A_C)$ as needed.
\end{proof}

Now let $\opA = (A,T,\a,\abar,\psi,\phi)$ be an arbitrary orbifold datum in $\mcC$.
By Proposition~\ref{defnprp:RA_orb_datum}, the isometric Morita module $R_C := {_A}(A \otimes C){_{A_C}}$ yields the Morita transport orbifold datum $\opA_C := R_C(\opA)$ in $\mcC$.
The section $\psi_{\Iz}$ in~\eqref{eq:psi-Iz} is of the form as in Remark~\ref{rem:F-A_special}, so that one has
\begin{cor}
\label{cor:inv_orb_datum_constr}
$(\Iz, \opA_C)$ is compatible and therefore by Proposition~\ref{defnprp:FA_orb_datum} produces the orbifold datum $\Iz(\opA_C)$ in $\mcC_B^\loc$, whose last entry is
\begin{equation}
\label{eq:phi-Iz}
\phi_{\Iz} := \frac{\phi}{(\dim_\mcC B)^{1/2}} \, .
\end{equation}
\end{cor}

\subsubsection*{Step (3)}
Let us abbreviate $\Iz_\opA = (\Iz)^{\opA_C}$.
We need to prove the following
\begin{lem}
\label{lem:Iz_is_equiv}
The induced ribbon functor $\Iz_\opA\colon\mcC_{\opA_C} \ra (\mcC_B^\loc)_{\Iz(\opA_C)}$ is an equivalence.
\end{lem}
For that one must show that it is fully faithful and essentially surjective.
Both these steps are achieved by straightforward calculations which are given in Appendix~\ref{appsec:proof_Iz_is_equiv}.
In particular, there we show that
\begin{enumerate}[(a)]
\item
For all $M,N\in\mcC_{\opA_C}$ and an arbitrary $\Iz(A_C)$-$\Iz(A_C)$-bimodule morphism $f\colon\Iz(M)\ra\Iz(N)$, its average morphism $\operatorname{avg}f$ in $(\mcC_B^\loc)_{\Iz(\opA_C)}$ (see~\eqref{eq:avg_map}) is in the image of $\Iz_\opA$.
Since $\operatorname{avg}$ is a projector onto the morphism spaces of $(\mcC_B^\loc)_{\Iz(\opA_C)}$, this implies that $\Iz_\opA$ is surjective on the morphism spaces and hence, since $\Iz_\opA$ is pivotal, automatically fully faithful (see Proposition~\ref{prp:surj_pivotal_functs}).
\item
For all $M\in\mcC_B^\loc$, the pipe object $P_{\Iz(\opA_C)}(M)$ (see Remark~\ref{rem:pipe_objs_on_induced_bimods}) is in the essential image of $\Iz_\opA$.
Since the pipe objects generate $(\mcC_B^\loc)_{\Iz(\opA_C)}$, this implies that $\Iz_\opA$ is essentially surjective (see Proposition~\ref{prp:ssi_funct_equiv}).
\end{enumerate}

\subsubsection*{Step (4)}
Specialising the orbifold datum $\opA$ in Lemma~\ref{lem:Iz_is_equiv} to the trivial orbifold datum $\opid$ in $\mcC$, we obtain the orbifold datum $\Iz(\opid_C)$ in $\mcC_B^\loc$ and braided equivalences $\mcC\simeq\mcC_{\opid_C}\simeq (\mcC_B^\loc)_{\Iz(\opid_C)}$.
We now unpack the statements of Propositions~\ref{defnprp:RA_orb_datum} and~\ref{defnprp:FA_orb_datum} for this case to obtain both $\Iz(\opid_C)$ and the equivalences explicitly.

\medskip

As the symbols for the orbifold datum $\opA$ and its constituents are no longer in use, we will abbreviate $\Iz(\opid_C) =: \opA = (A,T,\a,\abar,\psi,\phi)$ for the rest of the section.
We have:
\begin{itemize}
\item
$A = \Iz(C^* \otimes C)$ with
\begin{align} \nonumber
&  \text{multiplication: } && \pic[1.25]{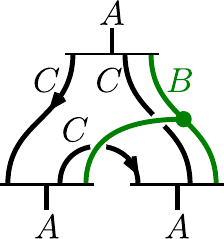} ,
&& \text{comultiplication: } && \pic[1.25]{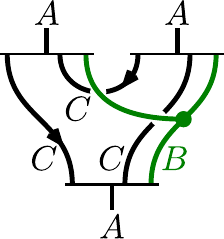},\\ \label{eq:inv-orb-dat_A}
&  \text{unit:} && \pic[1.25]{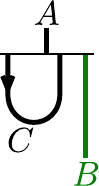} ,
&& \text{counit:} && \pic[1.25]{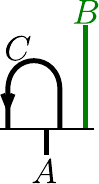} ;
\end{align}
\item
$T =\Iz(C^* \otimes C \otimes C)$ with $A$-actions
\begin{equation}
\label{eq:inv-orb-dat_T}
\pic[1.25]{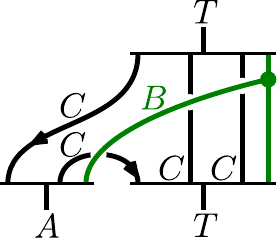} ,~
\pic[1.25]{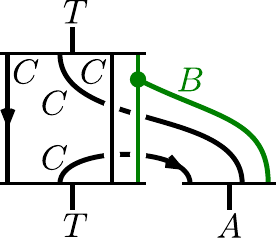} ,~
\pic[1.25]{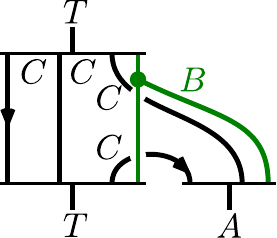} ;
\end{equation}
\item
$\a$, $\abar$ are given by
\begin{equation}
\label{eq:inv-orb-dat_alphas}
\a    = \pic[1.25]{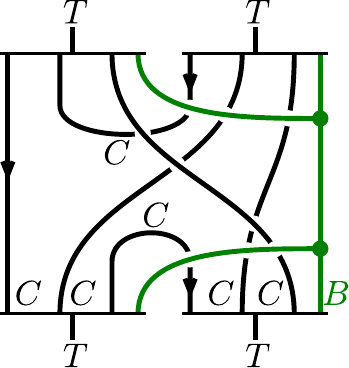} , \quad
\abar = \pic[1.25]{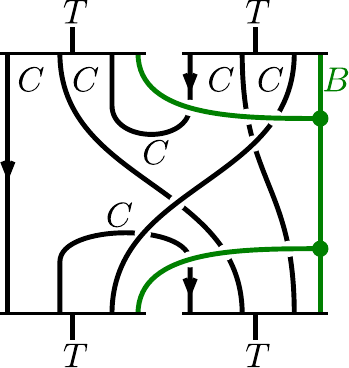} ;
\end{equation}
\item
$\displaystyle \psi = \left(\frac{\dim_\mcC B}{\Dim \mcC}\right)^{1/2} \cdot \pic[1.25]{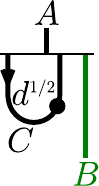}$,
\hspace{80pt}
$\bullet$ $\displaystyle\phi =\frac{1}{(\dim_\mcC B)^{1/2}}$ \, .
\end{itemize}
The equivalence $\Iz_C := [\Iz_{\opid_C}\colon\mcC\xra{\sim}(\mcC_B^\loc)_\opA]$ is defined
\begin{align}\nonumber
&\text{on objects:   } && X && \mapsto && (\Iz(C^* \otimes X \otimes C), \tau_1, \tau_2, \taubar{1}, \taubar{2}) \, ,\\ \label{eq:IzC_functor}
&\text{on morphisms: } && [f\colon X \ra Y] && \mapsto && \Iz(\id_{C^*} \otimes f \otimes \id_C) \, ,
\end{align}
where the left/right $A$-actions on $\Iz(C^* \otimes X \otimes C)$ are
\begin{equation}
\pic[1.25]{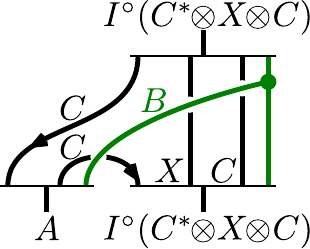} \, , \quad
\pic[1.25]{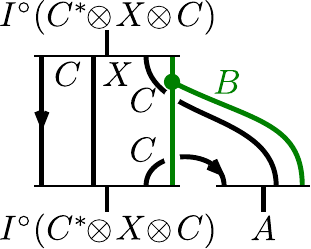}
\end{equation}
and the $T$-crossings are
\begin{equation}
\tau_1    = \pic[1.25]{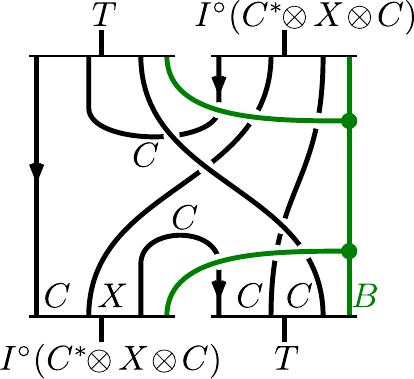} , \quad
\tau_2    = \pic[1.25]{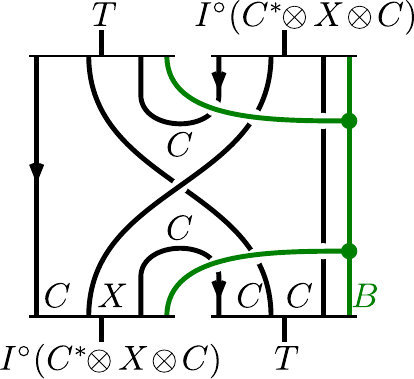} ,
\end{equation}
with $\taubar{1}$, $\taubar{2}$ defined similarly.

Finally, the monoidal structure morphisms $(\Iz_C)_2(X,Y)\colon\Iz_C(X)\otimes_A\Iz_C(Y)\xra{\sim}\Iz_C(X\otimes Y)$ and their inverses are given by the balanced maps
\begin{equation}
\pic[1.25]{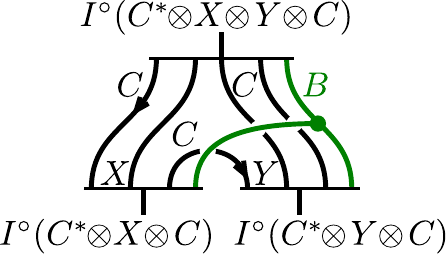} ,
\pic[1.25]{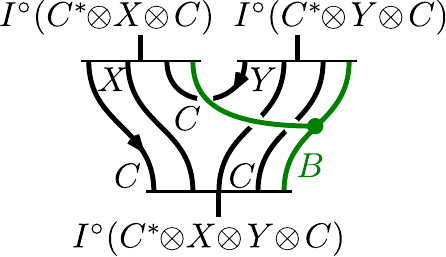}
\end{equation}
and one has $[(\Iz_C)_0\colon A\xra{\sim}\Iz_C(\opid) \cong \Iz(C^* \otimes C) \cong A] = \id_A$.
\begin{thm}
\label{thm:cond_inv_orb_datum}
The tuple $\opA = (A,T,\a,\abar,\psi,\phi)$ listed in~\eqref{eq:inv-orb-dat_A}--\eqref{eq:inv-orb-dat_alphas} is a simple orbifold datum in $\mcC_B^\loc$ and the functor $\Iz_C\colon\mcC\ra(\mcC_B^\loc)_\opA$ defined in~\eqref{eq:IzC_functor} together with the monoidal structure $((\Iz_C)_2, (\Iz_C)_0)$ is a ribbon equivalence.
\end{thm}

\begin{rem}
We should emphasise that having a MFC $\mcD\simeq\mcC_B^\loc$, finding an orbifold datum in $\mcD$ inverting the condensation is still a hard task since one needs to know $\mcC$, $B$ and the equivalence $\simeq$ to make use of $\opA$.
For example, finding all such orbifold data in $\Vect_\opk$ would amount to finding all Drinfeld centres $\mcZ(\mcS)$ for spherical fusion categories $\mcS$ with $\Dim\mcS\neq 0$, see Section~\ref{sec:Witt_equiv} below.
Nevertheless, the results of this section will allow us to make some general statements about the MFCs obtained from orbifold data.
Moreover, the explicit expressions for the constituents of $\opA$ can be used to explore various properties of a general orbifold datum undoing a condensation, for example in Section~\ref{subsec:Morita_theory_for_cond_inv} we briefly look at what can be said about the Morita class of the algebra $A = \Iz(C^* \otimes C)$ in $\mcC_B^\loc$.
\end{rem}

As a final exercise in this section, let us check the formula~\eqref{eq:Dim-CA_formula} for the categorical dimension of $(\mcC_B^\loc)_\opA$.
One has:
\begin{align} \nonumber
&
\tr_{\mcC_B^\loc} \omega_A^2 =
\pic[1.25]{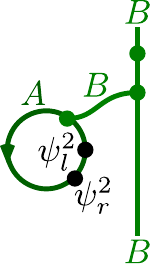} \hspace{-4pt}= \left(\frac{\dim_\mcC B}{\Dim\mcC}\right)^2 \hspace{-2pt}\cdot\hspace{-2pt}
\pic[1.25]{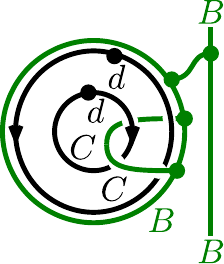} \hspace{-4pt}\stackrel{\eqref{eq:scissors_on_B}}{=} \frac{\dim_\mcC B}{\Dim\mcC} \cdot
\pic[1.25]{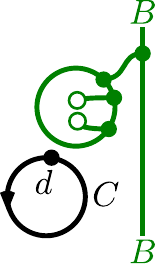} \\ \label{eq:tr-omega_calc}
&=
\dim_\mcC B \cdot \id_B \, ,
\end{align}
where we have used the pivotal structure of $\mcC_B^\loc$ to write the trace.
As usual, one identifies $\mcC_B^\loc(B,B)\cong\opk$ by mapping $\id_B \mapsto 1$, so that~\eqref{eq:tr-omega_calc} reads $\tr_{\mcC_B^\loc} \omega_A^2 = \dim_\mcC B$.
This yields:
\begin{equation}
\Dim (\mcC_B^\loc)_\opA =
\frac{\Dim \mcC_B^\loc}{\phi^8 \cdot \tr_{\mcC_B^\loc}\omega_A^2} =
\frac{\Dim \mcC / (\dim_\mcC B)^2}{(\dim_\mcC B)^{-4} \cdot (\dim_\mcC B)^2} =
\Dim \mcC
\end{equation}
as expected.

\subsection{Example: Morita theory of $E_6$ inversion}
\label{subsec:Morita_theory_for_cond_inv}
In this section we recall the examples of orbifold data found in~\cite{MR2} and show that they are in fact examples of condensation inversion.
We keep the argument general enough to be possibly applicable to other similar examples, in particular we show how in general the Morita class of the algebra $A = \Iz(C^* \otimes C)\in\mcC_B^\loc$ as in Theorem~\ref{thm:cond_inv_orb_datum} is related to that of the algebra $B\otimes B$ in $\mcC$.

\medskip

So far three explicit examples of condensation inversion were considered: in~\cite{CRS3} the following orbifold data were constructed
\begin{enumerate}[i)]
\item
In the trivial MFC $\Vect$ of finite dimensional $\opk$-vector spaces, a spherical fusion category $\mcS$ with $\Dim\mcS\neq 0$ can be used to construct an orbifold datum $\opA^\mcS$.
In~\cite{MR1} it was shown that the associated MFC $\Vect_{\opA^\mcS}$ is equivalent to the Drinfeld centre $\mcZ(\mcS)$, so that $\opA^\mcS$ undoes the condensation by the so-called Lagrangian algebra in $\mcZ(\mcS)$ (see~\cite[Sec.\,4.2]{DMNO}).
\item
Let $G$ be a finite group.
Then a $G$-crossed extension $\mcC_G^\times = \bigoplus_{g\in G} \mcC_g$ of a MFC $\mcC = \mcC_1$ yields an orbifold datum $\opA^G$ in $\mcC$.
The associated MFC $\mcC_{\opA^G}$ is expected (but as of yet not shown) to be equivalent to the equivariantisation $(\mcC_G^\times)^G$, in which case $\opA^G$ undoes the deequivariantisation procedure, which is a condensation of $(\mcC_G^\times)^G$.
\end{enumerate}
Of main interest to us in this section is the following ad-hoc example considered in~\cite{MR2}:
\begin{enumerate}[i)]
\setcounter{enumi}{2}
\item
The rank 11 MFC $\mcC(sl(2),10)$ of integrable highest weight modules of the affine Lie algebra $\widehat{sl}(2)_{10}$ has a condensable algebra $E_6$ whose condensation is an Ising-type MFC $\mcI_{\zeta,\epsilon}$ where the parameters $\epsilon\in\{\pm 1\}$, $\zeta\in\{\sqrt[8]{-1}\}$ describe one of several possible ribbon structures (for $E_6$-condensation one has $\epsilon=-1$ and $\zeta=\exp(3/8 \pi i)$). 
The initial ideas of how condensation inversion can be done has lead to finding instances of orbifold data $\opA_{h,\epsilon}$, $h^3 = \zeta$ in $\mcI_{\zeta,\epsilon}$.
These orbifold data were used to demonstrate how one can obtain some information about the associated MFCs (e.g.\ rank and categorical dimensions of some objects) without knowing them exactly.
It was however never proved that one of the associated MFCs $(\mcI_{\zeta,\epsilon})_{\opA_{h,\epsilon}}$ is actually $\mcC(sl(2),10)$, even though they were shown to all have 11 simples which for $h=\exp(19/24 \pi i)$, $\epsilon=-1$ have the same categorical dimensions as those of $\mcC(sl(2),10)$.
\end{enumerate}

Recall that the Ising categories $\mcI_{\zeta,\epsilon}$ all have 3 simples $\{\opid,\vareps,\sigma\}$ and the fusion rules $\vareps^{\otimes 2} \cong \opid$, $\vareps\otimes\sigma\cong\sigma$, $\sigma^{\otimes 2}\cong\opid\oplus\vareps$.
By construction, the orbifold data $\opA_{h,\epsilon} = (\widetilde{A}, \widetilde{T},\dots)$ in the example iii) above are exactly\footnote{The setting for using orbifold data in~\cite{MR2} was that of $\D$-separable Frobenius algebras and Euler completion, see Remark~\ref{rem:Delta-sep_orb_data}, so technically here we should take a rescaling of $\widetilde{A}$. In this section however we are only interested in the Morita transports of $A$ and $T$ entries, for which the scalings play no role and so we do need to require the Morita transport to be isometric.} the ones which satisfy the following ansatz for the algebra $\widetilde{A}$ and the $\widetilde{A}$-$\widetilde{A}^{\otimes2}$-bimodule $\widetilde{T}$:
\begin{equation}
\label{eq:undo-E6_orb_datum}
\widetilde{A} = \opid_\imath \oplus \opid_\varphi \, , \quad
\widetilde{T} = \bigoplus_{a,b,c\in\{\imath,\varphi\}} {_at_{bc}}\, , ~
\text{where }
\begin{cases}
{_\imath t_{\imath\imath}} = {_\varphi t_{\varphi\imath}} = {_\varphi t_{\imath\varphi} = \opid}\\
{_\imath t_{\varphi\varphi}} = \opid \, , \, {_\varphi t_{\varphi\varphi}} = \sigma\\
{_\imath t_{\imath\varphi}} = \dots = 0
\end{cases} \, .
\end{equation}
Here the indices $\imath$ and $\varphi$ are used only to distinguish between the two copies of $\opid$ in $\widetilde{A}$ and their actions on $\widetilde{T}$.
The splitting of $\widetilde{T}$ is into submodules ${_at_{bc}}$ on which only $\opid_a$-$\opid_b \otimes \opid_c$ acts non-trivially by unitors.
By the end of this section we will show the following
\begin{thm}
\label{thm:E6_is_inversion}
The orbifold datum in Theorem~\ref{thm:cond_inv_orb_datum} undoing the $E_6$-condensation of $\mcC(sl(2),10)$ can be Morita-transported to a one of the form~\eqref{eq:undo-E6_orb_datum}.
\end{thm}

We do this by first considering the Morita class of the algebra $A = \Iz(C^* \otimes C)$ in general and then imposing a series of assumptions both on the Morita class and on the Morita modules relating its elements, all of which apply to the $E_6$ case.

\medskip

Let $(F,\psi)$ be a symmetric separable Frobenius algebra in $\mcC$ and ${}_{B\otimes B} R _F$ a Morita module.
The $F$-$F$-bimodule isomorphism $R^* \otimes_{B^{\otimes 2}} R \cong F$ provides $F$ with a structure of a $B$-$B$-bimodule with the left and right actions given by
\begin{equation}
\pic[1.25]{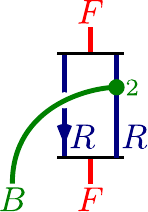} =
\pic[1.25]{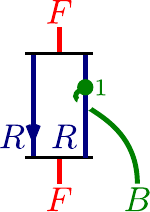}
\qquad \text{and} \qquad
\pic[1.25]{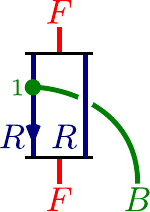} =
\pic[1.25]{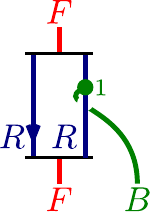}\, .
\end{equation}
By definition the $B$-actions on $F$ commute with the multiplication of $F$ in the sense that the following identities hold:
\begin{equation*}
\pic[1.25]{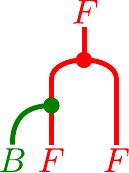}=
\pic[1.25]{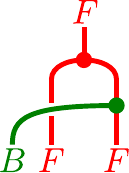}=
\pic[1.25]{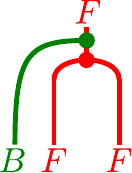},\quad
\pic[1.25]{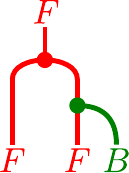}=
\pic[1.25]{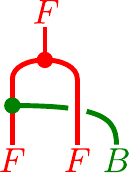}=
\pic[1.25]{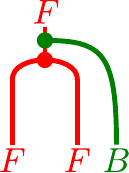}.
\end{equation*}

\begin{rem}
Such kind of compatibility between $F$ and $B$ yields what was called in~\cite[Def.\,2.10]{KMRS} an algebra $F\in\mcC$ over a pair of condensable algebras $(B,B')$ in $\mcC$ (in this case one has $B=B'$).
There they were used for an internal state-sum construction of domain walls between two bulk TQFTs of Reshetikhin--Turaev type $Z^{\operatorname{RT}}_{\mcC_B^\loc}$ and $Z^{\operatorname{RT}}_{\mcC_{B'}^\loc}$ (see Section~\ref{sec:introduction} for a brief review).
The occurrance of such algebras here is not coincidental: we are investigating the algebra $B\otimes B$ which labels the gap domain wall (see Figure~\ref{fig:gap_defect}).
\end{rem}

We proceed to construct out of $F$ and $R$ a symmetric separable Frobenius algebra $G$ in $\mcC^\loc_B$ and a Morita context ${}_A Q_G$ as follows:
\goodbreak
\begin{equation}
\label{eq:G_Q_def}
    G := \im \pic[1.25]{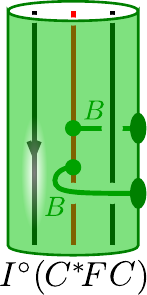} \qquad, \qquad
    Q := \im \pic[1.25]{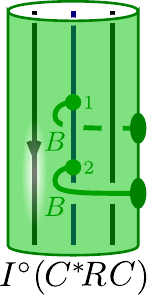} \, .
\end{equation}
The structure maps of $G$ are inherited from those of the algebra $\Iz(C^*FC)$ since the identities~\eqref{eq:rho-gamma_identities}, \eqref{eq:Iz_non_separable} imply that the idempotent defining $G$ is an algebra homomorphism.
Similarly, the idempotent defining $Q$ is a $\Iz(C^*FC)$-$A$-bimodule morphism so that $Q$ has left $A$ and right $G$ actions as needed.
That ${_A Q_G}$ is indeed a Morita module follows from the two computations sketched below:
\begin{align}
\nonumber
&Q \otimes_G Q^* \cong\\
&
\im \frac{\dim_\mcC B}{\Dim\mcC} \cdot \hspace{-18pt}\pic[1.25]{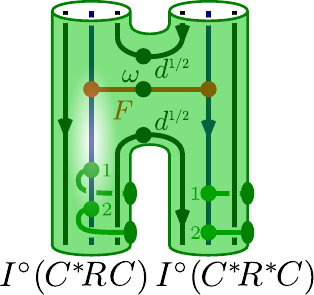} \hspace{-18pt} \cong
\im \hspace{-8pt} \pic[1.25]{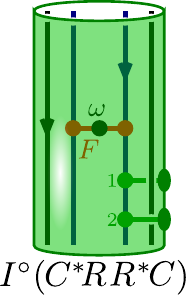} \hspace{-6pt} \cong
\im \hspace{-8pt} \pic[1.25]{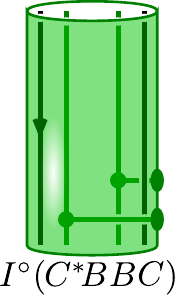} \hspace{-6pt} \cong
    \hspace{-8pt} \pic[1.25]{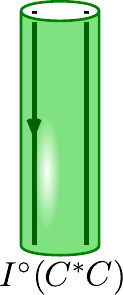} \hspace{-6pt} \cong A \, ,
\end{align}
and
\begin{equation}
Q^* \otimes_A Q \cong
\im \frac{\dim_\mcC B}{\Dim\mcC} \cdot \hspace{-18pt}
    \pic[1.25]{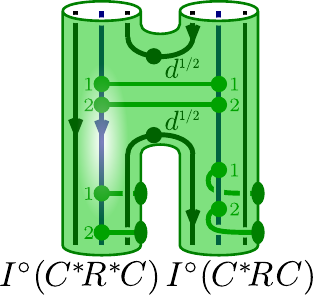} \hspace{-16pt} \cong
\im \hspace{-8pt}\pic[1.25]{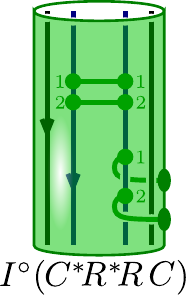} \hspace{-6pt} \cong 
\im \pic[1.25]{63_G_idemp.pdf} \cong
G \, .
\end{equation}
The algebra $F$ can in general be split into a direct sum of algebras $\bigoplus_{a=1}^n F_a$, which also splits $R$ into a direct sum $\bigoplus_a R_a$ of $B^{\otimes 2}$-$F_a$-bimodules.
From~\eqref{eq:G_Q_def} one also obtains the splittings of algebras $G\cong\bigoplus_a G_a$ and of $A$-$G$-bimodules $Q\cong\bigoplus_a Q_a$.
Suppose that for some index $a\in\{1,\dots,n\}$ in this decomposition one has $F_a = B$.
Then $R_a$ is a $B^{\otimes 2}$-$B$-bimodule, whose two left $B$-actions we will index by 1,2 and the right action by 0 (similarly as in~\eqref{eq:right_T_actions_commute}).
For the corresponding algebra $G_a\in\mcC_B^\loc$ one has:
\begin{equation}
G_a =
\im \pic[1.25]{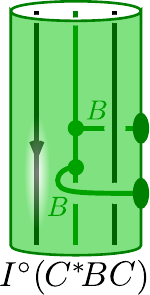} \cong
\im \pic[1.25]{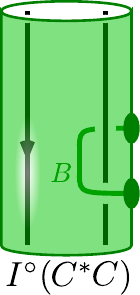} \cong
\Iz(C)^* \otimes_B \Iz(C) \, ,
\end{equation}
i.e.\ $G_a$ is Morita equivalent to the trivial algebra $\opid_{\mcC_B^\loc} = B$ in $\mcC_B^\loc$ via the Morita module $\Iz(C)^*$.
This also allows one to change the entry $Q_a$ into
\begin{align}\nonumber
&
\widetilde{Q}_a := Q_a \otimes_{G_a} \Iz(C^*) \cong\\ \label{eq:Qtild}
&
\im \frac{\dim_\mcC B}{\Dim\mcC} \cdot \hspace{-18pt} \pic[1.25]{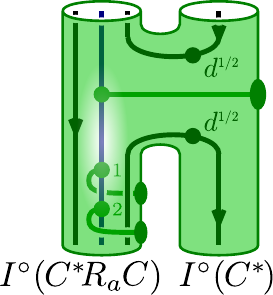} \cong
\im \pic[1.25]{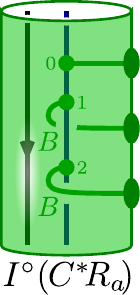} \cong
\im \pic[1.25]{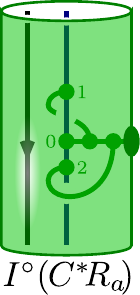} ,
\end{align}
where the right action of $B$ on $\widetilde{Q}_a$ is by the unitor $\mcC_B^\loc$.

\medskip

Until now the discussion about the Morita class of the algebra $A=\Iz(C^*C)\in\mcC_B^\loc$ was completely general.
We now make the following
\begin{equation}
\label{eq:BxB_sum_of_ones_assumption}
\text{assumption: the algebra $B^{\otimes 2}\in\mcC$ is Morita equivalent to $\bigoplus_{a=1}^n B$} \, .
\end{equation}
This means that $A$, as an algebra in $\mcC_B^\loc$ is Morita equivalent to $\bigoplus_{a=1}^n B$, i.e.\ the direct sum of $n$ copies of the tensor unit.
The Morita module is $\widetilde{Q} = \bigoplus_a \widetilde{Q}_a$ whose entries are as in~\eqref{eq:Qtild}.

We now turn to the Morita transport $\opA^{\widetilde{Q}}$ of the condensation inversion orbifold datum $\opA$ in $\mcC_B^\loc$, specifically its entry
\begin{equation}
\widetilde{T} := T^{\widetilde{Q}} = \widetilde{Q}^* \otimes_A T \otimes_{A^{\otimes 2}} (\widetilde{Q}^{\otimes 2})
= \bigoplus_{a,b,c = 1}^n {_at_{bc}} \, ,
\text{ where }
{_at_{bc}} = \widetilde{Q}^*_a \otimes_A T \otimes_{A^{\otimes 2}} (\widetilde{Q}_b \otimes_B \widetilde{Q}_c) \, .
\end{equation}
Explicit calculation of each of the objects ${_at_{bc}}$ yields the following simplification:
\begin{align} \nonumber
&
\im \left(\frac{d_B}{D}\right)^3 \cdot \hspace{-18pt} \pic[1.25]{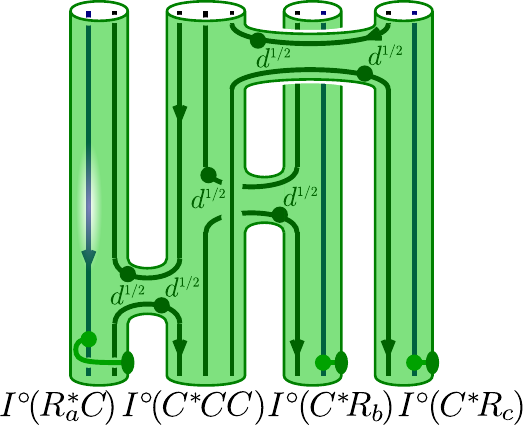} \hspace{-24pt} \cong 
\im \left(\frac{d_B}{D}\right)^2 \cdot \hspace{-18pt} \pic[1.25]{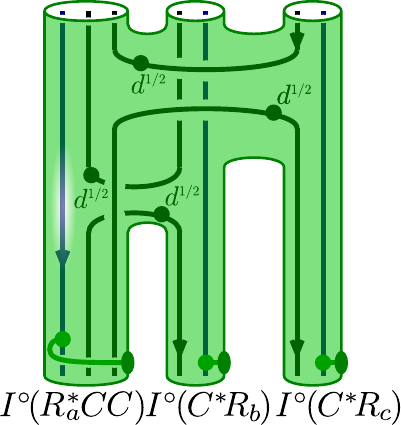}\\ \label{eq:t-abc_calc}
&\cong
\im \frac{d_B}{D} \cdot \hspace{-18pt} \pic[1.25]{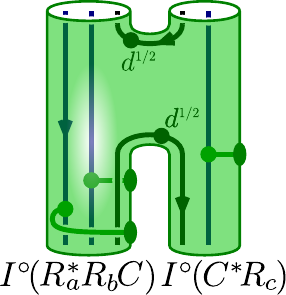}  \cong
\im \pic[1.25]{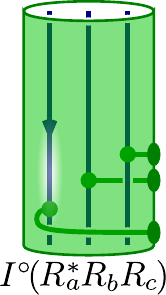} \cong
\im \pic[1.25]{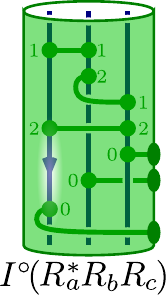} ,
\end{align}
where we have abbreviated $d_B = \dim_\mcC B$, $D = \Dim \mcC$ and composed all the $B$-actions on the modules $R_a$ occurring in the right-hand side of~\eqref{eq:Qtild} into one (before decomposing back in the last step of~\eqref{eq:t-abc_calc}).
Let us now make one more
\begin{align} \nonumber
&\text{assumption:} && \text{the $B^{\otimes 2}$-$B$-bimodules $R_a$ have the form $X_a \otimes B$ for some $X_a\in\mcC$}\\ \label{eq:Ra-XaB_assumption}
& && \text{and the actions} \pic[1.25]{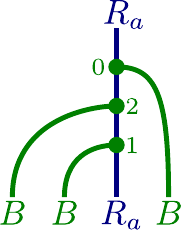} = \pic[1.25]{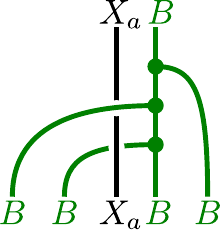}.
\end{align}
Under this assumption, the right hand side of~\eqref{eq:t-abc_calc} can be further simplified by removing all $B$ lines as in~\eqref{eq:Iz_simplify} coming from the expressions $R_a = X_a\otimes B$, which yields:
\begin{equation}
\label{eq:t-abc_under_assumptions}
{_at_{bc}} \cong \Iz(X_a^* X_b X_c) \, .
\end{equation}

The assumption~\eqref{eq:Ra-XaB_assumption} is not trivial in that it is not immediately implied by the assumption~\eqref{eq:BxB_sum_of_ones_assumption}.
We are not sure how severely imposing it limits the examples of condensation inversion orbifold data, but it does hold for the $E_6$ example as shown below.
A way to test it is the following
\begin{prp}
\label{prp:Ra-XaB_assumption_cond}
The assumption~\eqref{eq:Ra-XaB_assumption} holds if and only if there are objects $X_a\in\mcC$, $a=1,\dots,n$ such that:
\begin{enumerate}[i)]
\item $\bigoplus_a X_a  B  X^*_a \cong B^{\otimes 2}$ as objects in $\mcC$;
\item $\Iz(X^*_a X_b) \cong \delta_{ab}B$ as $B$-$B$-bimodules.
\end{enumerate}
\end{prp}
\begin{proof}
Clearly~\eqref{eq:Ra-XaB_assumption} implies i) and ii) by Morita equivalence
\begin{align} \nonumber
&
B^{\otimes 2} && \cong
R \otimes_{B^{\oplus n}} R^* \cong
\bigoplus_a R_a \otimes_B R_a^* \cong 
\bigoplus_a X_a B \otimes_B B X_a^* \cong
\bigoplus_a X_a B X_a^* \, ,\\ \nonumber
&
\delta_{ab}B &&\cong
B_a \otimes_{B^{\oplus n}} B^{\oplus n} \otimes_{B^{\oplus n}} B_b \cong
B_a \otimes_{B^{\oplus n}} R^* \otimes_{B^{\otimes 2}} R \otimes_{B^{\oplus n}} B_b\\
& && \cong
R_a^* \otimes_{B^{\otimes 2}} R_b \cong
\Iz(X_a^* X_b) \, .
\end{align}
The other implication comes from defining $R_a$, $a=1,\dots,n$ as in~\eqref{eq:Ra-XaB_assumption} and constructing the explicit bimodule isomorphisms
\begin{equation}
f\colon R^* \otimes_{B^{\otimes 2}} R\cong \bigoplus_a \Iz(X_a^*X_a) \ra B^{\oplus n} \, , \quad
g\colon R \otimes_{B^{\oplus n}} R^* \cong \bigoplus_a X_a B X^*_a \ra B^{\otimes 2}
\end{equation}
as follows:
\begin{align} \nonumber
&
f=\bigoplus_{a,b} f_{ab} = \bigoplus_{a,b} \d_{ab}\pic[1.25]{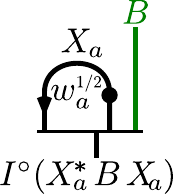},
&&
f^{-1} = \bigoplus_{a,b} f^{(-1)}_{ab} = \bigoplus_{a,b} \d_{ab}\pic[1.25]{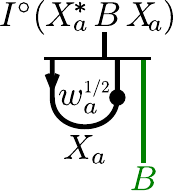},\\
&
g=\bigoplus_{a} g_a = \bigoplus_a \pic[1.25]{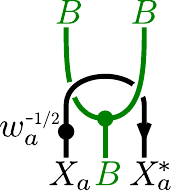},
&&
g^{-1}=\bigoplus_{a} g^{(-1)}_a = \bigoplus_a\pic[1.25]{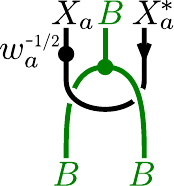},
\end{align}
where $w_a^{1/2}\in\End_\mcC X_a$ are arbitrary invertible morphisms such that for $w_a = w_a^{1/2}\circ w_a^{1/2}$ one has $\tr_\mcC(w_a) = 1$.
By construction, $f$ and $g$ are $B^{\oplus n}$- and $B^{\otimes 2}$-bimodule morphisms, so one only needs to show that $f$, $f^{-1}$ and $g$, $g^{-1}$ are indeed two-sided inverses.

For $f$ one has
\begin{equation}
f\circ f^{-1} = \bigoplus_a f_{aa} \circ f_{aa}^{(-1)} = \bigoplus_a (\tr_\mcC w_a \cdot \id_B) = \id_{B^{\oplus n}} \, ,
\end{equation}
and since ii) implies $\dim\mcBCB(\Iz(X_a^* X_a),B)=1$ and so that $f_{aa}$ and $f_{aa}^{(-1)}$ are two-sided inverses, an analogous computation yields $f^{-1}\circ f= \id$ as well.

For $g$ one has
\begin{align} \nonumber
&g^{-1}\circ g = \bigoplus_{a,b} g^{-1}_b\circ g_a \\
&=
\bigoplus_{a,b} \pic[1.25]{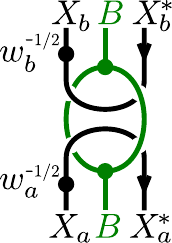} =
\bigoplus_{a,b} \pic[1.25]{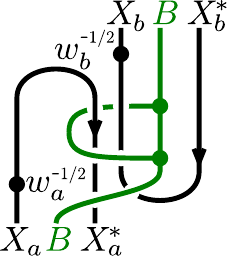} \stackrel{(*)}{=}
\bigoplus_{a,b} \d_{ab}\pic[1.25]{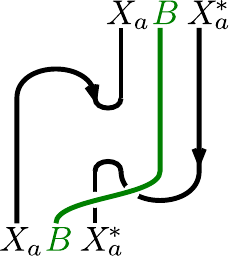} =
\id_{}\, ,
\end{align}
where in step~$(*)$ one recognises the projector onto $\Iz(X_a^* X_b) \cong \d_{ab}\Iz(X_a^* X_a)$ and uses the morphisms $f_{aa}$, $f_{aa}^{(-1)}$ from above to split the identity.
Since by i) the domain and the codomain are isomorphic, $g^{-1}$ must be the two-sided inverse.
\end{proof}

\subsubsection*{Example: $E_6$ inversion\footnote{I would like to thank Ingo Runkel for the help with this calculation.}}
Recall that the category $\mcC = \mcC(sl(2),k)$ of integrable highest weight modules of $\widehat{sl}(2)$ at level $k$ has $k+1$ simple objects $\underline{0},\underline{1},\dots,\underline{k}$ with fusion rules
\begin{equation}
\underline{m} \otimes \underline{n} = \bigoplus_{l=|m-n|}^{\min({m+n,\, 2k-m-n)}} \underline{l} \, ,
\end{equation}
where the sum is by steps of 2: $\underline{|m-n|} \oplus \underline{|m-n| + 2} \oplus \dots$.
Note that all objects are self dual: $\underline{m}^* = \underline{m}$.
In the case $k=10$, $\mcC$ contains a condensable algebra whose underlying object is $B = \underline{0} \oplus \underline{6}$ (the $E_6$ algebra).
The category $\mcC_B$ of right $B$-modules has six simple objects:
\begin{align} \nonumber
&
M_0 = B = \underline{0} \oplus \underline{6} \, ,
&&
M_1 = \underline{1} \otimes B = \underline{1} \oplus \underline{5} \oplus \underline{7} \, ,\\ \nonumber
&
M_2 = \underline{2} \otimes B = \underline{2} \oplus \underline{4} \oplus \underline{6} \oplus \underline{8} \, ,
&&
M_3 = \underline{9} \otimes B = \underline{3} \oplus \underline{5} \oplus \underline{9} \, ,\\
&
M_4 = \underline{10} \otimes B = \underline{4} \oplus \underline{10} \, ,
&&
M_5 = \underline{3} \oplus \underline{7} \, ,
\end{align}
all of which except for $M_5$ happen to be the induced modules and the local modules being $M_0$, $M_4$ and $M_5$ (see~\cite[Sec.\,6]{KO}).
The equivalence $\mcC_B^\loc\xra{\sim}\mcI$, where $\mcI=\mcI_{\zeta,\epsilon}$ is the Ising-type MFC with $\epsilon=-1$, $\zeta=\exp(3/8\pi i)$, is given by
\begin{equation}
\label{eq:sl2-10_to_Ising_funct}
M_0 \mapsto \opid \, , \quad M_4\mapsto \vareps \, , \quad M_5\mapsto\sigma \, .
\end{equation}
Knowing the simples $M_0,\dots,M_5$ one can quickly find the decompositions for the rest of the induced modules:
\begin{align} \nonumber
&
\underline{3} \otimes B = M_3 \oplus M_5 \, ,
&&
\underline{4} \otimes B = M_2 \oplus M_4 \, ,
&&
\underline{5} \otimes B = M_1 \oplus M_3 \, ,\\
&
\underline{6} \otimes B = M_0 \oplus M_2 \, ,
&&
\underline{7} \otimes B = M_1 \oplus M_5 \, ,
&&
\underline{8} \otimes B = M_2 \, .
\end{align}
Since $\Iz(\underline{m})$ is the localisation of the induced module $\underline{m}\otimes B$, one gets:
\begin{equation}
\label{eq:E6_Iz}
\Iz(\underline{m}) =
\begin{cases}
M_0 \, , \,\underline{m} = \underline{0}, \underline{6} \\ 
M_4 \, , \,\underline{m} = \underline{4}, \underline{10}\\
M_5 \, , \,\underline{m} = \underline{3}, \underline{7}\\
~0 \, , \, \text{otherwise}
\end{cases} .
\end{equation}

It is known that for the case of $E_6$ algebra, $B^{\otimes 2}$ is Morita equivalent to $B^{\oplus 2}$ so that the assumption~\eqref{eq:BxB_sum_of_ones_assumption} holds.
As in~\eqref{eq:undo-E6_orb_datum}, let us index the two copies of $B$ by symbols $\imath$, $\varphi$.
We claim that the corresponding summands $R_\imath$, $R_\varphi$ of the Morita module are as in the assumption~\eqref{eq:Ra-XaB_assumption} with $X_\imath = \underline{0}$ and $X_\varphi = \underline{1}$.
For this we need to check the conditions in Proposition~\ref{prp:Ra-XaB_assumption_cond}:
for i) one has:
\begin{equation} 
X_\imath B X^*_\imath \oplus X_\varphi B X^*_\varphi \cong
B ~\oplus ~ (\underline{1} \otimes B \otimes \underline{1}) =
\underline{0}^{\oplus 2} \oplus \underline{2} \oplus \underline{4} \oplus \underline{6}^{\oplus 3} \oplus \underline{8} =
B^{\otimes 2}
\end{equation}
as objects in $\mcC$, whereas for ii) one has
\begin{align} \nonumber
&
\Iz(X^*_\imath X_\imath) = \Iz(\underline{0}) = B \, ,
&&
\Iz(X^*_\imath X_\varphi) = \Iz(\underline{1}) = 0 \, ,\\
&
\Iz(X^*_\varphi X_\imath) = \Iz(\underline{1}) = 0
&&
\Iz(X^*_\varphi X_\varphi) = \Iz(\underline{0} \oplus \underline{2}) = B \, ,
\end{align}
as needed.

Finally, we can compute the summands of $\widetilde{T} = \bigoplus_{a,b,c\in\{\imath,\varphi\}} {_at_{bc}}$ by using~\eqref{eq:t-abc_under_assumptions} and~\eqref{eq:E6_Iz}, which under the identification~\eqref{eq:sl2-10_to_Ising_funct} yields exactly~\eqref{eq:undo-E6_orb_datum}.

\section{Witt equivalence}
\label{sec:Witt_equiv}
The construction of modular fusion categories (MFCs) from orbifold data suggests the following
\begin{defn}
We say that two MFCs $\mcC$ and $\mcD$ are \textit{orbifold equivalent} if there is an orbifold datum $\opA$ in $\mcC$ such that $\mcC_\opA\simeq\mcD$ as ribbon fusion categories.
\end{defn}
\noindent
At this point the word `equivalence' in the above definition is dubious, but it will be justified by the main result of this section: the orbifold relation on MFCs is in fact the same as Witt equivalence, a well known equivalence relation which we recall in Section~\ref{subsec:Witt_vs_orb_rels} below.
As a straightforward implication of this we look at the notion of a unital orbifold datum in Section~\ref{subsec:unital_orb_data} and show that up to an equivalence of the associated MFCs, all simple orbifold data can be taken to be unital. 

In this section only we will assume $\operatorname{char}\opk = 0$\footnote{This is only to use the results of~\cite{DGNO, DMNO} which assume $\operatorname{char}\opk$ to be zero.
Although generalisations to fields of arbitrary characteristic seem possible, one might expect some nuances to arise.}.

\subsection{Witt vs orbifold relations on MFCs}
\label{subsec:Witt_vs_orb_rels}
The Witt equivalence relation on MFCs was introduced in \cite{DMNO} and can be formulated in several ways, two of which are given by
\begin{prp}\cite[Prop.\,5.15]{DMNO}
\label{prp:Witt_eq_conds}
Let $\mcC$ and $\mcD$ be MFCs.
Then the following are equivalent:
\begin{enumerate}[i)]
\item 
\label{prp:Witt_eq_conds:i}
there exists a spherical fusion category $\mcS$ and a ribbon equivalence $\mcC \boxtimes \widetilde{\mcD}\simeq\mcZ(\mcS)$ (recall that $\widetilde{\mcD}$ denotes the category $\mcD$ with the reversed braiding);
\item
\label{prp:Witt_eq_conds:iii}
there exists a MFC $\mcE$ and two condensable algebras $B,B'\in\mcE$ such that $\mcE_B^\loc\simeq\mcC$ and $\mcE_{B'}^\loc\simeq\mcD$ as ribbon fusion categories.
\end{enumerate}
$\mcC$ and $\mcD$ are then called \textit{Witt equivalent} if any one of these conditions apply.
The ribbon equivalence as in \ref{prp:Witt_eq_conds:i} is called a \textit{Witt trivialisation}.
\end{prp}

One of the main results of this paper is:
\begin{thm}
\label{thm:Witt_vs_orb_eq}
Two MFCs $\mcC$ and $\mcD$ are Witt equivalent if and only if they are orbifold equivalent
\end{thm}
\noindent
The remainder of this section will be dedicated to proving the above statement.

\medskip

One of the implications in the statement of Theorem~\ref{thm:Witt_vs_orb_eq} directly follows from the results in Section~\ref{subsec:inv_orb_datum}:
If $\mcC$ and $\mcD$ are Witt equivalent, let $\mcE$, $B$, $B'$ be as in Proposition~\ref{prp:Witt_eq_conds}\ref{prp:Witt_eq_conds:iii}.
By Remark~\ref{rem:condensation_orb_datum}, there is an orbifold datum $\opB'$ in $\mcE$ such that $\mcE_{\opB'}\simeq\mcE_{B'}^\loc\simeq\mcD$ as ribbon fusion categories.
Then by Theorem~\ref{thm:inv_orb_datum} there is an orbifold datum $\Iz_B(\opB'_C)$ in $\mcE_B^\loc\simeq\mcC$ (whose form is inferred from Corollary~\ref{cor:inv_orb_datum_constr}) such that $\mcC_{\Iz(\opB'_C)}\simeq \mcE_{\opB'_C}\simeq\mcE_{\opB'}\simeq\mcD$ (see Lemma~\ref{lem:Iz_is_equiv}), so that $\mcC$ and $\mcD$ are orbifold equivalent.

\medskip

Showing the other implication in the statement of Theorem~\ref{thm:Witt_vs_orb_eq} will need a different approach:
Having a simple orbifold datum $\opA=(A,T,\a,\abar,\psi,\phi)$ in $\mcC$ we will construct a Witt trivialisation for the pair of MFCs $\mcC$ and $\mcC_\opA$, so that the condition in Proposition~\ref{prp:Witt_eq_conds}\ref{prp:Witt_eq_conds:i} holds.

\medskip

Recall from Section~\ref{subsec:orb_data_and_assoc_cats} that for $i\in\{1,2\}$ the category $\mcC_\opA^i$ is defined to have objects $(M\in\mcACA,\tau_i,\taubar{i})$, where $\tau_i$, $\taubar{i}$ are the respective $T$-crossing and its inverse.
We start by defining two functors
\begin{equation}
F_l\colon \mcC_\opA \ra \mcZ(\mcC_\opA^1) \, , \qquad
F_r\colon \widetilde{\mcC} \ra \mcZ(\mcC_\opA^1) \, .
\end{equation}

For $F_l$ one sets for all $M=(M,\tau_1,\taubar{1},\tau_2,\taubar{2})\in\mcC_\opA$
\begin{equation}
\label{eq:FlM}
F_l(M,\tau_1,\taubar{1},\tau_2,\taubar{2}) := ((M,\tau_1,\taubar{1}), c^\opA_{M,-}) \, ,
\end{equation}
where the braiding $c^\opA_{M,-}$ of $\mcC_\opA$ can be used as a half-braiding for $M$ by the same definition~\eqref{eq:CA_br_twist}, but taking $N\in\mcC_\opA^1$ (the hexagon identity is proved in the same way as the braiding property, see~\cite[Prop.\,3.5]{MR1}).

For $F_r$ one defines for all $X\in\widetilde{\mcC}$
\begin{equation}
F_r(X) := ((X \otimes A, \tau_1^X, \taubar{1}^X), \gamma^X) \, ,
\end{equation}
where the left and right actions on $X\otimes A$, the $T$-crossings $\tau_1^X$, $\taubar{1}^X$ and the half-braiding $\gamma^X$ are defined by
\begin{equation}
\label{eq:FrX_morphs}
\pic[1.25]{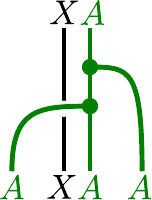} , \,
\tau_1^X := \pic[1.25]{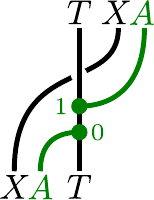} , \,
\taubar{1}^X := \pic[1.25]{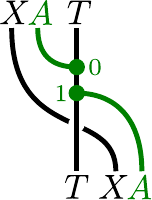} , \,
\gamma^X_N := \pic[1.25]{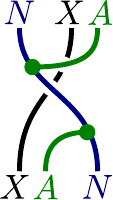} .
\end{equation}

On morphisms $F_l$ acts as identity and $F_r$ sends $[f\colon X\ra Y]$ to $f \otimes \id_A$.

\medskip

The penultimate step in constructing a Witt trivialisation of $\mcC_\opA \boxtimes \widetilde{\mcC}$ is
\begin{prp}
\label{prp:CCA-ZCA1_equiv}
Let $\opA$ be a simple orbifold datum.
Then
\begin{equation}
\label{eq:CCA-ZCA1_equiv:F_funct}
F \colon \mcC_\opA \boxtimes \widetilde{\mcC} \ra \mcZ(\mcC_\opA^1) \, , \quad M \boxtimes X \mapsto F_l(M) \otimes_{\mcZ(\mcC_\opA^1)} F_r(X)
\end{equation}
is a ribbon equivalence.
In particular, $\mcC_\opA^1$ is an indecomposable multifusion category.
\end{prp}
\begin{rem}
\label{rem:CAC-ZCA2_equiv}
One could have similarly constructed an equivalence $\mcC \boxtimes \widetilde{\mcC_\opA}\ra\mcZ(\mcC_\opA^2)$ instead, so that $\mcC_\opA^2$ too is an indecomposable multifusion category.
\end{rem}
\noindent
Before starting with the proof one should note that there is a slight complication due to $\opA$ not a priori being haploid, i.e.\ $\mcC_\opA^1$ not necessarily being a spherical fusion category as required in Proposition~\ref{prp:Witt_eq_conds}\ref{prp:Witt_eq_conds:iii}, but rather a pivotal multifusion category.
This is not hard to remedy by exchanging $\mcC_\opA^1$ for its component category, see Proposition~\ref{prp:ZA-ZAii_equiv}.
This will also later play a role in proving Proposition~\ref{prp:CCA-ZCA1_equiv}.

\medskip

It is straightforward to find a monoidal structure and show that $F$ is indeed a ribbon functor (note that in case $\opA$ is a trivial orbifold datum, $F$ is exactly the equivalence $\mcC\boxtimes\widetilde{\mcC}\xra{\sim}\mcZ(\mcC)$ in Definition~\ref{def:MFC}).
We focus on checking that $F$ is fully faithful and essential surjective.
\begin{proof}[Proof of Proposition~\ref{prp:CCA-ZCA1_equiv}: $F$ is fully faithful]
It is enough to show that for simples $\D\in\mcC_\opA$ and $i\in\widetilde{\mcC}$, the object $F(\D\boxtimes i)\in\mcZ(\mcC_\opA^1)$ is also simple.
Let $X\in\widetilde{\mcC}$ be an arbitrary object and $\Gamma$ the half-braiding of $F(\D\boxtimes i)$, \i.e.\ the product of the half-braidings as in~\eqref{eq:FlM} and~\eqref{eq:FrX_morphs}.
Since morphisms in $\mcZ(\mcC_\opA^1)$ by definition commute with half-braidings, any morphism $f\in\End_{\mcZ(\mcC_\opA^1)}F(\D\boxtimes i)$ satisfies in particular
\begin{equation}
\Gamma_{F_r(X)} \circ (f \otimes \id_{F_r(X)}) =
(f \otimes \id_{F_r(X)}) \circ \Gamma_{F_r(X)}
\end{equation}
for all $X\in\widetilde{C}$.
This can be rewritten as follows:
\begin{equation}
\label{eq:Witt_triv_proof:f_simple}
\pic[1.25]{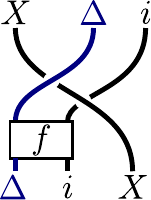} =
\pic[1.25]{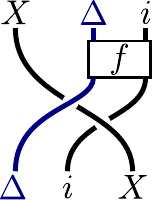} \, ,
\end{equation}
where we have used the identification of the underlying $A$-$A$-bimodules $\D\otimes_A F_r(X) \cong \D\otimes X$ (the latter one with the $A$-actions analogous to those in~\eqref{eq:FrX_morphs}) so that one has $c_{\D,F_r(X)}^\opA = c_{\D, X}$ in~\eqref{eq:FlM}, and similarly $\D\otimes_A F_r(i) \otimes_A F_r(X) \cong  \D \otimes X \otimes i$ for the domains of the morphisms on both sides.

Since~\eqref{eq:Witt_triv_proof:f_simple} also holds as an identity of morphisms in $\mcC$, the modularity of $\mcC$ implies that $f = g \otimes \id_i$ for some $g\in\End_\mcC(\D)$.
We claim that $g$ must also be a morphism in $\End_{\mcC_\opA}\D$, i.e.\ satisfy~\eqref{eq:M}.
To show this, note that since $f$ is a morphism in $\mcZ(\mcC_\opA^1)$, so is its partial trace with respect to $i$
\begin{equation}
\pic[1.25]{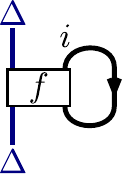} = \dim_\mcC i \cdot g \, ,
\end{equation}
i.e.\ $g$ is readily a morphism in $\mcZ(\mcC_\opA^1)$ and hence also a morphism in $\mcC_\opA^1$ so that~\eqref{eq:M} holds for $i=1$.
To show that it holds for $i=2$ as well, note that being a morphism in $\mcZ(\mcC_\opA^1)$ it commutes with the half-braidings of the objects in this Drinfeld centre, in particular the half-braiding of $\D$ with the object $(T_2,\a)\in\mcC_\opA^1$ which by~\eqref{eq:FlM} is exactly $c^\opA_{\D,T_2} = \tau_2^\D$.

Since $\D$ is a simple object of $\mcC_\opA$ and $g\in\End_{\mcC_\opA}\D$, one has $g = \xi \cdot \id_\Delta$ for some $\xi\in\opk$ and therefore also $f = \xi \cdot \id_{F(\D\boxtimes i)}$.
Hence we have started with an arbitrary endomorphism $f\in\End_{\mcZ(\mcC_\opA^1)}F(\D\boxtimes i)$ and showed that it is proportional to the identity, i.e.\ $F(\D\boxtimes i)\in\mcZ(\mcC_\opA^1)$ is simple as needed.
\end{proof}

Let us now briefly address the complication due $\mcC_\opA^1$ being a multifusion and not necessarily a spherical fusion category.
Note that, by Proposition~\ref{prp:ZA-ZAii_equiv} and the comment preceding, having showed that $F$ is fully faithful also implies that $\mcC_\opA^1$ (and similarly $\mcC_\opA^2$) are indecomposable.
\begin{lem}
\label{lem:ZCAi-ZF_equiv}
Let $\opA$ be a simple orbifold datum in a MFC $\mcC$.
Then $\mcC_\opA^i$, $i\in\{1,2\}$ has a component category $\mcF$ which is a spherical fusion category and the canonical braided equivalence $\mcZ(\mcC_\opA^i)\xra{\sim}\mcZ(\mcF)$ is ribbon.
In particular, if $\opA$ is haploid, $\mcC_\opA^i$ is a spherical fusion category.
\end{lem}
\begin{proof}
Let $A \cong \bigoplus_j A_j$ be a decomposition of the monoidal unit in $\mcC_\opA^i$.
Recall that one has by definition $\mcF = A_j \otimes_\opA^i \mcC_\opA^i \otimes_\opA^i A_j$ for some index $j$.
As was argued in Example~\ref{eg:FF_Ind_functor}, each $A_j$ is trivially a symmetric $\D$-separable Frobenius algebra in $\mcC_\opA^i$, $\mcF$ is the category of $A_j$-$A_j$-bimodules in $\mcC_\opA^i$ and the projector functor $N\mapsto A_j \otimes_\opA^i N \otimes_\opA^i A_j$ is identical to the bimodule induction functor, which is pivotal.
The equivalence $\mcZ(\mcC_\opA^i)\xra{\sim}\mcZ(\mcF)$ is obtained from the same projection (see~\eqref{eq:ZA-ZAii_equiv}) and is therefore both braided and pivotal, hence ribbon. 

As in Example~\ref{eg:FF_forgetful}, by applying the forgetful functor $\mcC_\opA^i\ra\mcACA\ra\mcC$ each $A_j$ becomes a symmetric separable Frobenius algebra in $\mcC$ with the section $\psi_j=[\opid\xra{\psi}A\ra A_j]$.
It is enough to compare the left/right dimensions of a simple $\D\in\mcF$, which can be done in the category of $A_j$-$A_j$-bimodules in $\mcC$.
Using~\eqref{eq:tr_in_ACA} one has:
\begin{equation}
\pic[1.25]{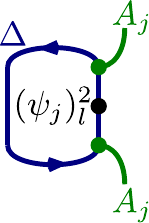} = (\dim_l)_{\mcC_\opA^i}\D \cdot
\pic[1.25]{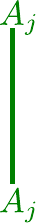} \, .
\end{equation}
Precomposing both sides with $(\omega_j)_A^2$ and taking trace in $\mcC$ yields
\begin{equation}
\label{eq:ZCAi-ZF_equiv:dim-Delta}
\tr_\mcC (\omega_j)_\Delta^2 = (\dim_l)_{\mcC_\opA^i}\Delta \cdot \tr_\mcC (\omega_j)_{A_j}^2 \, .
\end{equation}
Since $\opA$ is simple and so $\tr_\mcC\omega_A^2 = \sum_j \tr_\mcC(\omega_j)_{A_j}^2 \neq 0$, one can choose the index $j$  so that both sides of~\eqref{eq:ZCAi-ZF_equiv:dim-Delta} are non-zero, which yields a formula for the left dimension of $\D$ in $\mcC_\opA^i$.
The same formula is obtained by similarly computing the right dimension of $\D$ (in this step one has to use the fact that $\mcC$ is spherical).
\end{proof}

In light of Lemma~\ref{lem:ZCAi-ZF_equiv} we see that a proper candidate for Witt trivialisation needed for Theorem~\ref{thm:Witt_vs_orb_eq} is obtained by picking a spherical component category $\mcF$ of $\mcC_\opA^1$ and composing $\mcC_\opA \boxtimes \widetilde{\mcC} \xra{F} \mcZ(\mcC_\opA^1) \xra{\sim} \mcZ(\mcF)$.
It remains to show that $F$ is essentially surjective.
For this we use the factorisation property of MFCs (see~\cite[Cor.\,7.8]{Mug2}, \cite[Prop.\,2.11, Thm.\,3.14]{DGNO})
\begin{thm}
\label{thm:MFCs_factor}
Let $\mcC$ be a MFC and $\mcD$ a full subcategory of $\mcC$ which is also a MFC.
Then $\mcD \boxtimes \mcD' \simeq \mcC$ as ribbon fusion categories, where $\mcD'$ is the full subcategory of $\mcC$ with the objects
\begin{equation}
\label{eq:commutator_in_MFC}
\mcD' = \{X\in\mcC ~,~ c_{Y,X} \circ c_{X,Y} = \id_{X\otimes Y} \text{ for all } Y\in\mcD\} \, .
\end{equation}
\end{thm}

We are now ready to continue showing that the functor~\eqref{eq:CCA-ZCA1_equiv:F_funct} is an equivalence.
\begin{proof}[Proof of Proposition~\ref{prp:CCA-ZCA1_equiv}: $F$ is essentially surjective]
It is enough to show that the rank of the source and the target categories of $F$ coincide.
For this we show that
\begin{equation}
\label{eq:FrC-comm-Fl-CA_equiv}
F_r(\widetilde{\mcC})' \simeq F_l(\mcC_\opA) \, .
\end{equation}
Then by Theorem~\ref{thm:MFCs_factor} there exists some equivalence $F_l(\mcC_\opA)\boxtimes F_r(\widetilde{\mcC})\simeq\mcC_\opA\boxtimes\widetilde{\mcC}\simeq\mcZ(\mcC_\opA^1)$, so that the ranks indeed coincide.

Let $((N,\tau_1,\taubar{1}),\gamma)\in F_r(\widetilde{\mcC})' \subseteq \mcZ(\mcC_\opA^1)$ be an arbitrary object where $\gamma\colon N \otimes_\opA^1 - \Ra - \otimes_\opA^1 N$ is the half-braiding.
We define morphisms $\tau_2\colon N \otimes_A T_2 \lra T_2 \otimes_A N : \taubar{2}$ to be the half-braiding $\gamma_{T_2}$ of $N$ with the object $(T_2,\a)\in\mcC_\opA^1$ (see~\eqref{eq:T1_T2_objs}) and its inverse, both seen as morphisms in $\mcACA$.
We claim that the tuple $(N,\tau_1,\tau_2, \taubar{1}, \taubar{2})$ satisfies the identities~\eqrefT{1}--\eqrefT{7} and so is a good candidate to be an object of $\mcC_\opA$:
\begin{itemize}
\item
\eqrefT{1}: already holds by definition.
\item
\eqrefT{2}: holds since $\gamma_{T_2}$ is a morphism in $\mcC_\opA^1$ and therefore satisfies~\eqref{eq:M} for $i=1$:
\begin{equation}
\pic[1.25]{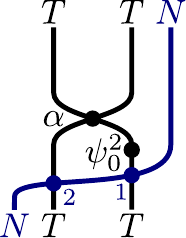} =
\pic[1.25]{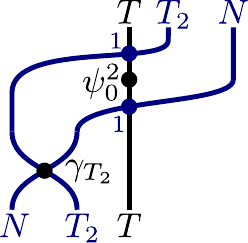} =
\pic[1.25]{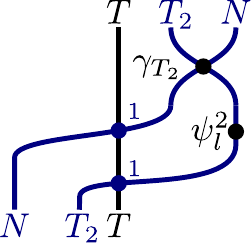} =
\pic[1.25]{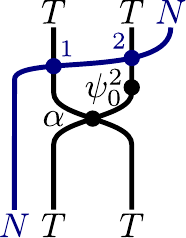}
\end{equation}
\item
\eqrefT{3}: let $T\in\widetilde{\mcC}$ be the underlying object of the corresponding $A$-$A^{\otimes 2}$-bimodule.
One has $T_2 \otimes_\opA^1 F_r(T) \cong T \otimes T$.
Set $\a'\colon T_2 \otimes_\opA^1 T_2 \ra T_2 \otimes_\opA^1 F_r(T)$ to be the morphism in $\mcC_\opA^1$ as defined by the balanced map
\begin{equation}
\a' := \pic[1.25]{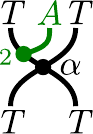} \, .
\end{equation}
That it is indeed a morphism in $\mcC_\opA^1$ follows directly from the identity~\eqrefO{1}.
The half-braiding of $N$ then must commute with it:
\begin{equation}
\pic[1.25]{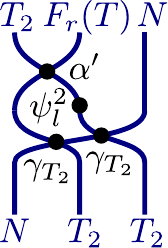} =
\pic[1.25]{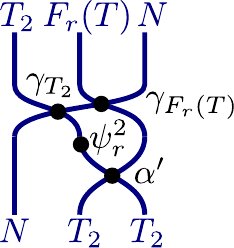} =
\pic[1.25]{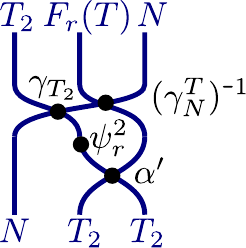} ,
\end{equation}
where in the second step we have used that $N\in F_r(\widetilde{\mcC})'$ and therefore by the condition in~\eqref{eq:commutator_in_MFC}, the braiding of $N$ and $F_r(T)$ in $\mcZ(\mcC_\opA^1)$ (which is by definition $\gamma_{F_r(T)}$) is equal to the inverse braiding of $F_r(T)$ and $N$ (which is equal to $(\gamma^T_N)^{-1}$ and is obtained as the inverse of the half-braiding as in~\eqref{eq:FrX_morphs}).
Unpacking the definitions this yields:
\begin{equation}
\pic[1.25]{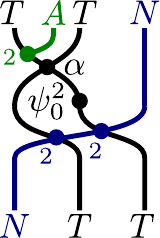} =
\pic[1.25]{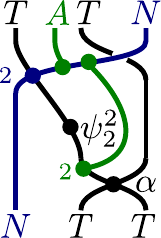} =
\pic[1.25]{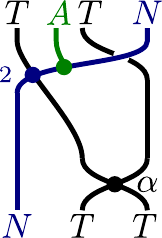} .
\end{equation}
Post-composing the $A$-line on both sides with the counit $\vareps:A \ra \opid$ then yields exactly the identity~\eqrefT{3}.
\item
\eqrefT{4}, \eqrefT{5}: hold since $\gamma_{T_2}$, $\gamma_{T_2}^{-1}$ are inverses of each other.
\item
\eqrefT{6}, \eqrefT{7}: hold since $\mcC_\opA^1$, and therefore also $\mcZ(\mcC_\opA^1)$, is pivotal and the half-braiding $\widetilde{\gamma}\colon N^* \otimes_\opA^1 - \Ra - \otimes_\opA^1 N^*$ of a dual object $N^*\in\mcZ(\mcC_\opA^1)$ is by definition such that $\widetilde{\gamma}_{T_2}$, $\widetilde{\gamma}_{T_2}^{-1}$ are exactly as in~\eqref{eq:CA_MN_A_T-crossings} upon identifying $\tau_2 = \gamma_{T_2}$ (see~\cite[Sec.\,5.2.2]{TV}).
\end{itemize}

It remains to check that the half-braiding $\gamma_{T_2}$ is an $A$-$A^{\otimes 2}$-bimodule morphism.
This is done in a similar way as checking the identity~\eqrefT{3}:
Let us define a morphism $\rho\colon T_2 \otimes_\opA^1 F_r(A) \ra T_2$ in $\mcC_\opA^1$ by
\begin{equation}
\pic[1.25]{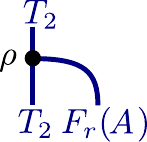} :=
\pic[1.25]{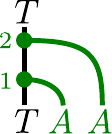} \, .
\end{equation}
That it satisfies~\eqref{eq:M} for $i=1$ follows from identities~\eqref{eq:right_T_actions_commute} and~\eqref{eq:aabar-A_commute} as well as how the $T$-crossing of $F_r(A)$ was defined in~\eqref{eq:FrX_morphs}.
Then one has:
\begin{equation}
\pic[1.25]{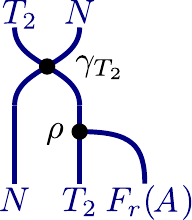} =
\pic[1.25]{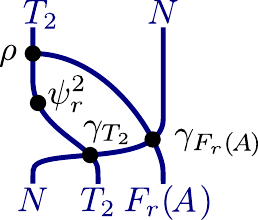} =
\pic[1.25]{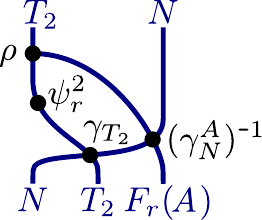} .
\end{equation}
Unpacking the definitions one gets:
\begin{equation}
\pic[1.25]{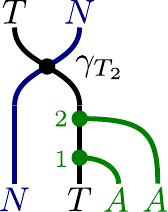} =
\pic[1.25]{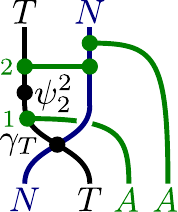} =
\pic[1.25]{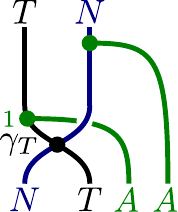} .
\end{equation}
Pre-composing the rightmost $A$ line with the unit $\eta\colon\opid\ra A$ then gives exactly the missing condition for $\gamma_{T_2}$ to be an $A$-$A^{\otimes 2}$-bimodule morphism.

The assignment
\begin{equation}
    F_r(\widetilde{\mcC})' \ni ((N,\tau_1,\taubar{1}),\gamma) \mapsto F_l(N,\tau_1,\tau_2,\gamma_{T_2},\gamma_{T_2}^{-1})\in F_l(\mcC_\opA)
\end{equation}
then provides one with the inverse to the inclusion $F_l(\mcC_\opA)\ra F_r(\widetilde{\mcC})'$ and so the equivalence~\eqref{eq:FrC-comm-Fl-CA_equiv} holds as needed.
\end{proof}

\subsection{Corollary on unital orbifold data}
\label{subsec:unital_orb_data}
Although we have not pursued this point of view, an orbifold datum $\opA = (A,T,\a,\abar,$ $\psi,\phi)$ in a MFC $\mcC$ can be seen as a certain algebra object in the monoidal bicategory $\mcA{}lg_\mcC$ of algebras in $\mcC$, their bimodules and bimodule morphisms (with the monoidal product being given by the tensor product of algebras in $\mcC$).
Indeed, the $A$-$A^{\otimes 2}$-bimodule $T$ is a 1-morphism in $\mcA{}lg_\mcC(A\otimes A, A)$ that can be interpreted as the multiplication and the $A$-$A^{\otimes 3}$-bimodule morphisms $\a$, $\abar$ can be seen as the associator 2-morphisms.
The current definition however does not include the unit into the construction, which suggests the following
\begin{defn}
A \textit{unital orbifold datum} in a MFC $\mcC$ is a tuple $(\opA, U, \l, \rho)$, where $\opA=(A,T,\a,\abar,\psi,\phi)$ is an orbifold datum in $\mcC$, $U$ is a left $A$-module and $\l\colon T \otimes_2 U$ $\xra{\sim }A$, $\rho\colon T \otimes_1 U \xra{\sim} A$ are isomorphisms in $\mcC_\opA^1$ and $\mcC_\opA^2$ respectively (where $A$ has the standard $T$-crossings~\eqref{eq:CA_MN_A_T-crossings} and $T \otimes_i U$, $i\in\{1,2\}$ the ones as in~\eqref{eq:TxL_T-crossings}).
\end{defn}

All examples of orbifold data considered in the references~\cite{CRS3,MR2} were unital.
The main claim of this section is
\begin{thm}
\label{thm:unital_orb_data}
Let $\opA$ be a simple orbifold datum in a MFC $\mcC$.
Then there is a simple unital orbifold datum $\opA'$ such that $\mcC_{\opA}\simeq\mcC_{\opA'}$ are equivalent as ribbon fusion categories.
\end{thm}
This follows almost directly from the proof of Theorem~\ref{thm:Witt_vs_orb_eq}: just below its statement it was argued that an orbifold datum, producing the MFC $\mcD\simeq\mcE_{B'}^\loc$ out of a Witt equivalent MFC $\mcC\simeq\mcE_B^\loc$ (where $\mcC,\mcD,\mcE,B,B'$ are as in Proposition~\ref{prp:Witt_eq_conds}) can be taken to be $\Iz_B(\opB'_C)$.
The converse implication in Theorem~\ref{thm:Witt_vs_orb_eq} then implies that \textit{all} MFCs associated to orbifold data can be obtained this way.
It is then enough to show that $\Iz_B(\opB'_C)$ is naturally a unital orbifold datum, for which we have a more general statement:
\begin{prp}
Let $\mcC$ be a MFC and $\opA=(\opA,U,\l,\rho)$ a unital orbifold datum in $\mcC$.
Then both Morita transports of $\opA$ and transports of $\opA$ along compatible ribbon Frobenius functors are naturally unital orbifold data.
\end{prp}
\begin{proof}
For an isometric Morita module ${_AR_B}$ one defines the unital structure to be given by the left $B$-module $R(U) = R^*\otimes_A U$ and the $B$-$B$-bimodule isomorphisms
\begin{align}\nonumber
&
T^R \otimes_2 R(U) \xra{R_\otimes(R^* \otimes_0 T \otimes_1 R,U)} R(T\otimes_2 U) \xra{R(\l)} R(A) \, ,\\
&
T^R \otimes_1 R(U) \xra{R_\otimes(R^* \otimes_0 T \otimes_2 R,U)} R(T\otimes_1 U) \xra{R(\rho)} R(A) \, ,
\end{align}
where the morphisms $R_\otimes(R^* \otimes_0 T \otimes_i R,U)$, $i\in\{1,2\}$ are defined as in~\eqref{eq:Rtensor_morphs}, with $\zeta = \id$ as $R$ is isometric and treating $R^* \otimes_0 T \otimes_i R$ as a left $A$-module.
That they are morphisms in $\mcC_{R(\opA)}^1$ and $\mcC_{R(\opA)}^2$ can be shown by checking~\eqref{eq:M} directly.

For another MFC $\mcD$ and a ribbon Frobenius functor $F\colon\mcC\ra\mcD$ which is compatible with $\opA$, the unital structure of $F(\opA)$ is the left $F(A)$-module $F(U)$ together with the following $F(A)$-$F(A)$-bimodule isomorphisms
\begin{align}\nonumber
&
F(T) \otimes_2' F(U) \xra{F^A_\otimes(T,U)} F(T \otimes_2 U) \xra{F(\l)} F(A) \, , \\
&
F(T) \otimes_2' F(U) \xra{F^A_\otimes(T,U)} F(T \otimes_2 U) \xra{F(\rho)} F(A) \, ,
\end{align}
where $F^A_\otimes(T_i,U)$ are defined as in~\eqref{eq:Ftensor_morphs}.
Again, a direct computation shows that they are morphisms in $\mcD_{F(\opA)}^1$ and $\mcD_{F(\opA)}^2$.
\end{proof}
This immediately implies:
\goodbreak
\begin{cor}
\label{cor:A-IZAC_unital}
Let $\mcC$ be a MFC, $\opA$ a unital orbifold datum in $\mcC$ and $B\in\mcC$ a condensable algebra.
Then the orbifold datum $\Iz(\opA_C)$ as in Corollary~\ref{cor:inv_orb_datum_constr} is also unital.
\end{cor}
\begin{proof}
$\opA_C$, being a Morita transport of $\opA$, is unital and since $(\Iz,\opA_C)$ are compatible, $\Iz(\opA_C)$ is unital.
\end{proof}

Going back to the previous example, let $\mcC$, $\mcD$, $\mcE$, $B$, $B'$ be as in Proposition~\ref{prp:Witt_eq_conds}.
For the orbifold datum $\opB'$ in $\mcE$ as in Remark~\ref{rem:condensation_orb_datum}, the triple $(U,\l,\rho)$ with $U=B'$ and $\l,\rho\colon B' \otimes_{B'} B' \xra{\sim} B'$ being given by the unitors in ${_{B'}\mcC_{B'}}$, is an obvious unital structure, so that by Corollary~\ref{cor:A-IZAC_unital} $\Iz_B(\opB')$ is unital as needed.

\appendix
\section{Proof of Lemma~\ref{lem:Iz_is_equiv}}
\label{appsec:proof_Iz_is_equiv}
In this appendix section we lay out the details for the steps (a) and (b) in proving Lemma~\ref{lem:Iz_is_equiv} as formulated below its statement.
For the rest of the section, let $B$ be a condensable algebra in a MFC $\mcC$, $\Iz\colon \mcC\ra\mcC_B^\loc$ the local induction functor, $\opA$ an orbifold datum in $\mcC$ and $\opA_C$ its Morita transport along the Morita module $R_C = A \otimes C$.
By Corollary~\ref{cor:inv_orb_datum_constr}, $(\Iz,\opA_C)$ is compatible so that $\Iz(\opA_C)$ is an orbifold datum in $\mcC_B^\loc$ and one has a braided functor $(\Iz)^{\opA_C}\colon\mcC_{\opA_C}\ra(\mcC_B^\loc)_{\Iz(\opA_C)}$.
We will use the abbreviations $d_B = \dim_\mcC B$ for the categorical dimension of $B$, $\opA_C = (A_C, T_C, \a_C, \abar_C, \psi_C, \phi_C)$ for the entries in the orbifold datum $\opA_C$ and $\Iz_\opA=(\Iz)^{\opA_C}$.

\subsubsection*{Step (a)}
We need to show that for two objects $M,N\in\mcC_{\opA_C}$ and an arbitrary morphism $f\in\mcC(M,N)$ between the underlying objects of $\mcC$, the average morphism $\operatorname{avg}f\in\mcC_{\opA_C}(M,N)$ (as defined in~\eqref{eq:avg_map}) is in the image of $(\Iz)^{\opA_C}$.
We first note that using the definition~\eqref{eq:Iz_funct_def} and the auxiliary morphisms~\eqref{eq:rho-gamma_morphs}, it is straightforward to find a morphism $\widetilde{f}\colon M \ra N \otimes B$ in $\mcC$ such that
\begin{equation}
\pic[1.25]{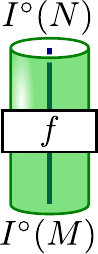} = \pic[1.25]{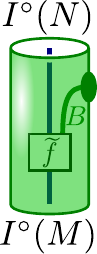} \, ,
\end{equation}
which furthermore can be taken to be a left $A_C$-module morphism since the map $\operatorname{avg}$ factors through the projection on the space of such morphisms.
Specialising the expression~\eqref{eq:avg_map} for the orbifold datum $\Iz(\opA_C)$ and the expressions~\eqref{eq:psi-Iz} and~\eqref{eq:phi-Iz} for its entries $\psi_{\Iz}$ and $\phi_{\Iz}$ yields:
\goodbreak
\begin{align}\nonumber
&
(\phi \cdot d_B)^{-4} \cdot \operatorname{avg}f\\ \nonumber
&=
\pic[1.25]{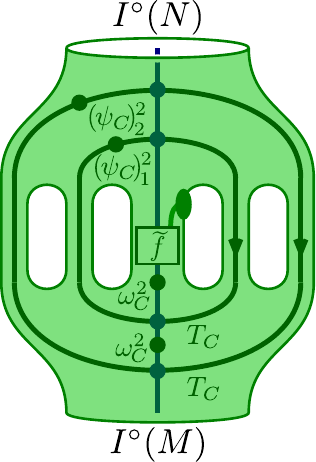} \hspace{-4pt}\stackrel{\eqref{eq:Iz_non_separable}}{=}\hspace{-4pt}
\pic[1.25]{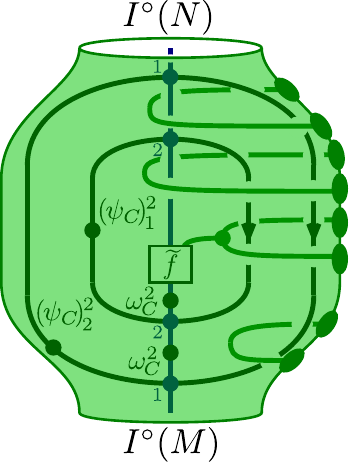} \hspace{-4pt}\stackrel{(*)}{=}\hspace{-4pt}
\pic[1.25]{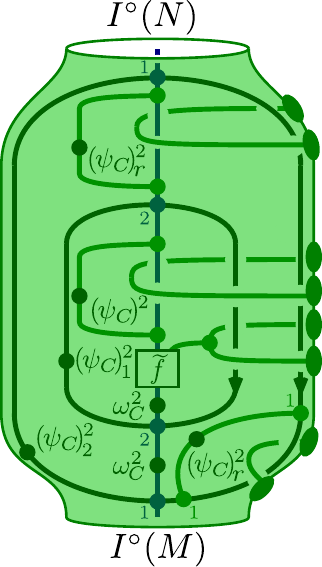}\\ \nonumber
&\stackrel{(**)}{=} \frac{1}{(d_B)^3}
\pic[1.25]{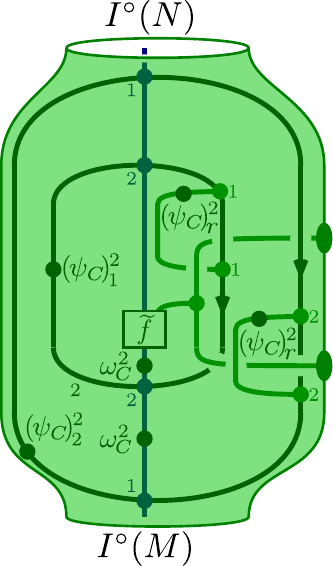} \stackrel{\eqref{eq:scissors_on_B_ito_algs}}{=} \frac{1}{(d_B)^5}
\pic[1.25]{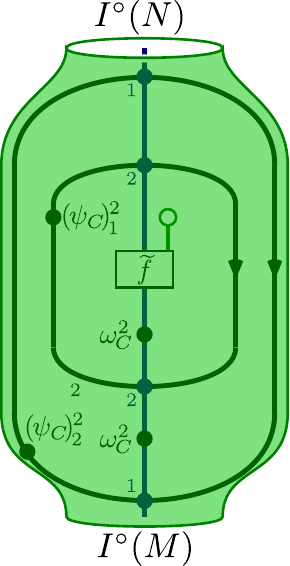}\\ \label{eq:avg-Iz_calc}
&=
(\phi\cdot d_B)^{-4} \cdot \Iz\bigg(d_B^{-1} \cdot \operatorname{avg}((\id_N\otimes\vareps)\circ\widetilde{f})\bigg) \, .
\end{align}
Here in equalities $(*)$ and $(**)$ we have used the separability property of the algebra $A_C$, the commutation relations between $A_C$-actions and the $T$-crossings of $M$ and $N$, as well as the identity~\eqref{eq:scissors_on_B_ito_algs}, for example one of the steps is obtained as follows:
\goodbreak
\begin{align} \nonumber
&
\pic[1.25]{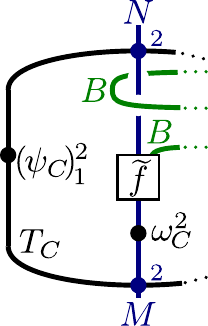} =
\pic[1.25]{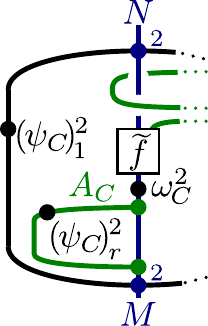} =
\pic[1.25]{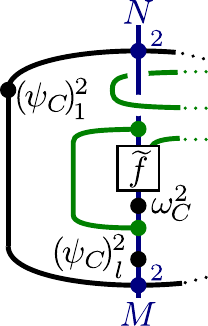} =
\pic[1.25]{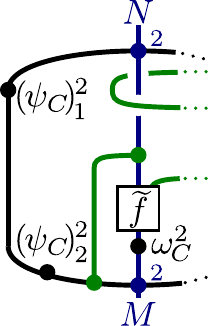}\\
&=
\pic[1.25]{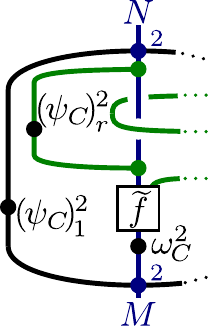} = \frac{1}{d_B}
\pic[1.25]{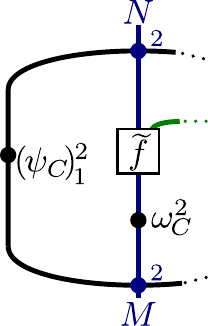} \, .
\end{align}
As the morphism $\operatorname{avg}((\id_N\otimes\vareps)\circ\widetilde{f})$, on which $\Iz$ is applied on the right-hand side of~\eqref{eq:avg-Iz_calc}, is in $\mcC_{\opA_C}$, this computation shows that $\operatorname{avg}f$, and therefore any morphism in $(\mcC_B^\loc)_{\opA_C}$, is in the image of $\Iz_\opA$ as needed.

\subsubsection*{Step (b)}
Let $U\colon\mcC_B^\loc\ra\mcC$ be the forgetful functor.
For an arbitrary object $M\in\mcC_B^\loc$ we claim that one has
\begin{equation}
\label{eq:PI-IP_iso}
    P_{\Iz(\opA_C)}(M) \cong \Iz_\opA(P_{\opA_C}(U(M))) \quad\text{in}~(\mcC_B^\loc)_{\Iz(\opA_C)} \, ,
\end{equation}
which we will show by constructing an invertible morphism.

Recall from~\eqref{eq:pipe_funct_def} and Remark~\ref{rem:pipe_objs_on_induced_bimods} the definition of the pipe functor $P_\opA\colon\mcC\ra\mcC_\opA$ for an arbitrary orbifold datum $\opA=(A,T,\a,\abar,\psi,\phi)$ in a MFC $\mcC$.
Using the canonical isomorphisms $T\otimes_2 A \cong T$, $A \otimes_2 T^* \cong T^*$, for all $X\in\mcC$ one can rewrite
\begin{align}\nonumber
P_\opA(X)
&\cong 
\im \pic[1.25]{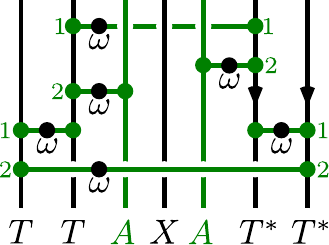} \cong
\im \pic[1.25]{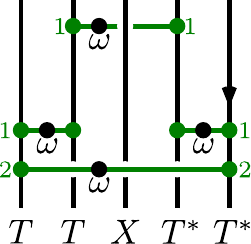}\\ \label{eq:PX_simplified}
&\cong
\im \pic[1.25]{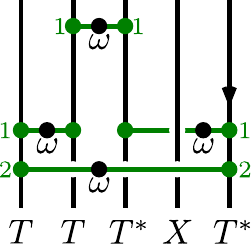} \cong
\im \pic[1.25]{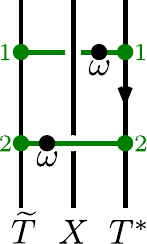} ,
\end{align}
where on the right-hand side we have merged the left three $T$-labelled lines into a single $A$-$A^{\otimes 2}$-bimodule $\widetilde{T} := T \otimes_1 (T_1 \otimes_A T^*_1)$ with $T_1$ denoting the $A$-$A$-bimodule obtained from $T$ by forgetting the second right $A$-action.
We adapt the notations~\eqref{eq:right_T_actions_commute} and~\eqref{eq:psis_on_T} to $\widetilde{T}$.

Let us use the expression on the right-hand side of~\eqref{eq:PX_simplified} for the orbifold datum $\Iz(\opA_C)$.
Since $(\Iz,A_C)$ is strongly separable, $\Iz$ preserves relative tensor products with respect to $A_C$ to those with respect to $\Iz(A_C)$ (see Definition~\ref{def:F-A_strongly_sep_full} and Proposition~\ref{prp:F-A_pivotal_embed}) so one has $\widetilde{\Iz(T_C)} \cong \Iz(\widetilde{T}_C)$.
The isomorphism~\eqref{eq:PI-IP_iso} can then be explicitly given by the balanced maps
\begin{equation}
\label{eq:eeinv_morphs}
e = \pic[1.25]{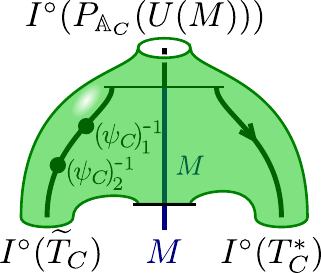} \, , \quad
e^{-1} = \pic[1.25]{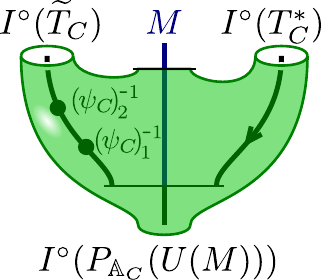} \, ,
\end{equation}
where the morphisms $M\lra\Iz(M)$ are the ones splitting the first idempotent in~\eqref{eq:Iz_simplify}.
To show that they indeed satisfy all the requirements, one needs to check that they are indeed morphisms in $(\mcC_B^\loc)_{\Iz(\opA_C)}$ (i.e.\ satisfy~\eqref{eq:M}) and are inverses of each other.
All this is done by straightforward calculations, we only show the inverse property.

The identity $e\circ e^{-1} = \id$ follows from the computation in Figure~\ref{fig:eeinv_id}, where on the left-hand side one indeed recognises the morphism $d_B^{-2} \cdot e\circ e^{-1}$ with the $(\psi_C)_1^2$- and $(\psi_C)_2^2$-insertions and the factor $d_B^{-2}$ coming from composing the maps, balanced with respect to two actions of the algebra $(\Iz(A_C), \psi_{\Iz})$ as explained in~\eqref{eq:extra_psi}, and the right-hand side is equal to $d_B^{-2} \cdot \id$ by property~\eqref{eq:scissors_on_B_ito_algs}.

The other identity $e^{-1}\circ e = \id$ is implied by the calculation in Figure~\ref{fig:einve_id}, where on the left-hand side one recognises the balanced map corresponding to the identity morphism of $P_{\Iz(\opA_C)}(M)$ and the right-hand side is exactly the balanced map corresponding to $e^{-1}\circ e$: the $B$-line can be absorbed to the projection $\Iz(U(M))\ra M$ and the two $A_C$-lines result from splitting of the idempotent~\eqref{eq:PX_simplified} on  $\widetilde{T}_C \otimes U(M) \otimes T$ defining $P_{\opA_C}(U(M))$.
\begin{figure}
\captionsetup{format=plain, indention=0.5cm}
%\centering
\begin{subfigure}[b]{1.0\textwidth}
	%\centering
	\pic[1.25]{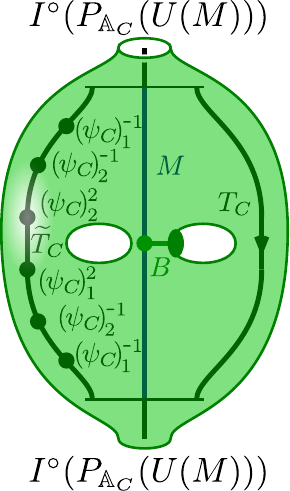} \hspace{-4pt}$\stackrel{\eqref{eq:Iz_non_separable}}{=}$\hspace{-10pt}
    \pic[1.25]{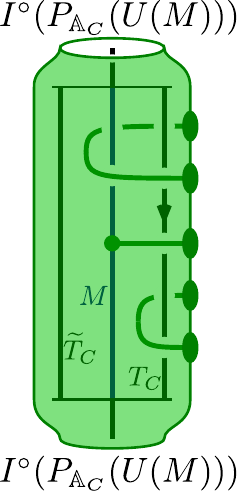} \hspace{-10pt}$=$\hspace{-10pt}
    \pic[1.25]{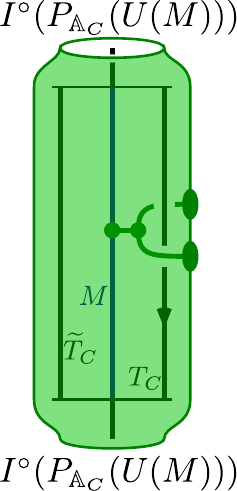} \hspace{-10pt}$=$\hspace{-4pt} 
    \pic[1.25]{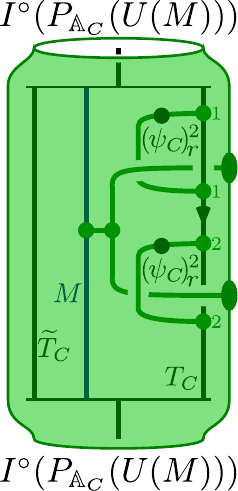}
	\caption{}
	\label{fig:eeinv_id}
\end{subfigure}\\
\begin{subfigure}[b]{1.0\textwidth}
	%\centering
	\pic[1.25]{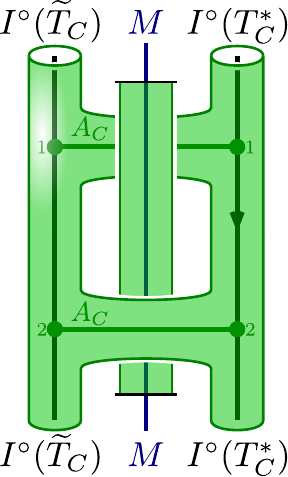} $\stackrel{\eqref{eq:Iz2_inclusion}}{=}$
    \pic[1.25]{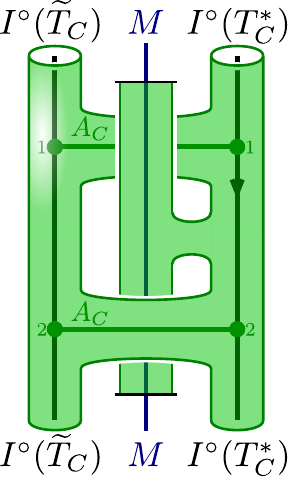} $\stackrel{\eqref{eq:F_graph_calc:braided},~ \eqref{eq:F_graph_calc:cobraided}}{=}$
    \pic[1.25]{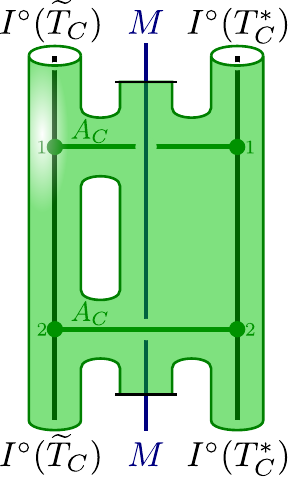}\\
    $\stackrel{\eqref{eq:Iz_non_separable}}{=}$
    \pic[1.25]{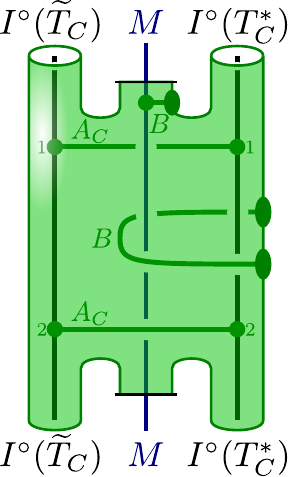} $=$
    \pic[1.25]{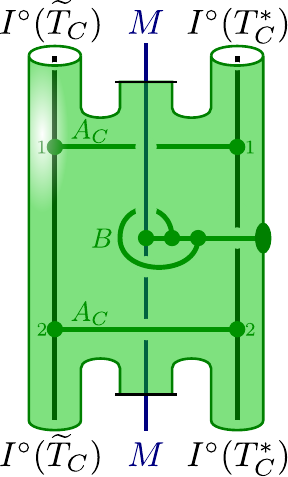} $=$
    \pic[1.25]{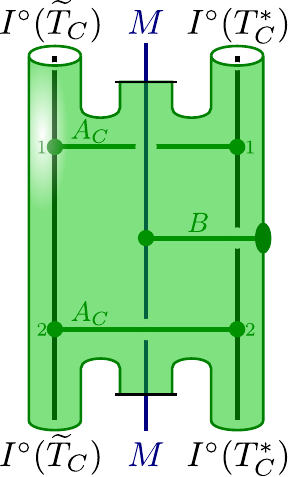}
	\caption{}
	\label{fig:einve_id}
\end{subfigure}
\caption{
Showing that the morphisms $e$, $e^{-1}$ in~\eqref{eq:eeinv_morphs} are inverses of each other.
(a) computation implying $e \circ e^{-1} = \id$; (b) computation implying $e^{-1} \circ e = \id$.
}
\label{fig:inverses_computation}
\end{figure}

\newcommand{\arxiv}[2]{\href{http://arXiv.org/abs/#1}{#2}}
\newcommand{\doi}[2]{\href{http://doi.org/#1}{#2}}


\begin{thebibliography}{XXXXX} \small
\setlength{\parskip}{0pt}
\setlength{\itemsep}{4pt}


\bibitem[BakK]{BakK}
B.~Bakalov, A.~Kirillov,
\textit{Lectures on tensor categories and modular functors},
\doi{10.1090/ulect/021}{University Lecture Series \textbf{21}, AMS, 2001}.

\bibitem[BNRW1]{BNRW1}
P.~Bruillard, S.~H.~Ng, E.~Rowell, Z.~Wang,
\textit{Rank-finiteness for modular categories},
\doi{10.1090/jams/842}{J.\ Am.\ Math.\ Soc.\ \textbf{29} (2016) 857--881},
\arxiv{1310.7050}{arXiv:1310.7050 [math.QA]}.

\bibitem[BNRW2]{BNRW2}
P.~Bruillard, S.-H.~Ng, E.C.~Rowell, Z.~Wang,
\textit{On classification of modular categories by rank},
\doi{10.1093/imrn/rnw020}{Int.\ Math.\ Res.\ Not.\ \textbf{2016} (2016) 7546--7588},
\arxiv{1507.05139}{arXiv:1507.05139 [math.QA]}.

\bibitem[Bur]{Bur}
F.~J.~Burnell,
\textit{Anyon condensation and its applications},
\doi{10.1146/annurev-conmatphys-033117-054154}{Annu.\ Rev.\ Condens.\ Matter Phys.\ \textbf{9} (2018) 307--327},
\arxiv{1706.04940}{arXiv:1706.04940 [cond-mat.str-el]}.

\bibitem[Ca]{Ca}
N.~Carqueville,
\textit{Lecture notes on $2$-dimensional TQFT},
\doi{10.4064/bc114-2}{Banach Center Publ.\ \textbf{114} (2018) 49--84},
\arxiv{1607.05747}{arXiv:1607.05747 [math.QA]}.

\bibitem[CMRSS1]{CMRSS1}
N.~Carqueville, V.~Mulevi\v{c}ius, I.~Runkel, D.~Scherl, G.~Schaumann, 
\textit{Orbifold graph TQFTs}, 
\arxiv{2101.02482}{arXiv:2101.02482 [math.QA]}.

\bibitem[CMRSS2]{CMRSS2}
N.~Carqueville, V.~Mulevi\v{c}ius, I.~Runkel, D.~Scherl, G.~Schaumann, 
\textit{Reshetikhin--Turaev TQFTs close under generalised orbifolds}, 
\arxiv{2109.04754}{arXiv:2109.04754 [math.QA]}.

\bibitem[Cr]{Cr}
D.~Creamer,
\textit{A Computational Approach to Classifying Low Rank Modular Categories},
\arxiv{1912.02269}{arXiv:1912.02269 [math.QA]}.

\bibitem[CR]{CR}
N.~Carqueville, I.~Runkel, 
\textit{Orbifold completion of defect bicategories}, 
\doi{10.4171/QT/76}{Quantum Topology \textbf{7}:2 (2016), 203--279}, 
\arxiv{1210.6363}{arXiv:1210.6363 [math.QA]}.

\bibitem[CRS1]{CRS1}
N.~Carqueville, I.~Runkel, G.~Schaumann, 
\textit{Orbifolds of $n$-dimensional defect TQFTs}, 
\doi{10.2140/gt.2019.23.781}{Geom.\ Topol.\ \textbf{23} (2019) 781--864},
\arxiv{1705.06085}{arXiv:1705.06085 [math.QA]}.

\bibitem[CRS2]{CRS2}
N.~Carqueville, I.~Runkel, G.~Schaumann, 
\textit{Line and surface defects in Reshetikhin--Turaev TQFT},  
\doi{10.4171/QT/121}{Quantum Topology \textbf{10} (2019) 399--439}, 
\arxiv{1710.10214}{arXiv:1710.10214 [math.QA]}.

\bibitem[CRS3]{CRS3}
N.~Carqueville, I.~Runkel, G.~Schaumann,
\textit{Orbifolds of Reshetikhin--Turaev TQFTs},
\href{http://www.tac.mta.ca/tac/volumes/35/15/35-15abs.html}{Theory and Appl.\ of Categories \textbf{35} (2020) 513--561},
\arxiv{1809.01483}{arXiv:1809.01483 [math.QA]}.

\bibitem[CZW]{CZW}
S.~X.~Cui, M.~S.~Zini, Z.~Wang,
\textit{On generalized symmetries and structure of modular categories},
\doi{10.1007/s11425-018-9455-5}{Science China Math.\ \textbf{62} (2019) 417--446},
\arxiv{1809.00245}{arXiv:1809.00245 [math.QA]}.

\bibitem[DGNO]{DGNO}
V.~Drinfeld, S.~Gelaki, D.~Nikshych, V.~Ostrik,
\textit{On braided fusion categories I},
\doi{10.1007/s00029-010-0017-z}{Sel.\ Math.\ New Ser.\ \textbf{16} (2010) 1--119},
\arxiv{0906.0620}{arXiv:0906.0620 [math.QA]}.

\bibitem[DKR]{DKR}
A.~Davydov, L.~Kong, I.~Runkel, 
\textit{Field theories with defects and the centre functor}, 
in ``Mathematical Foundations of Quantum Field Theory and Perturbative String Theory'',
\href{https://bookstore.ams.org/pspum-83/}{Proc.\ Symp.\ Pure Math.\ \textbf{83} (2011) 71--130},
\arxiv{1107.0495}{arXiv:1107.0495 [math.QA]}.

\bibitem[DMNO]{DMNO}
A.~Davydov, M.~M\"uger, D.~Nikshych, V.~Ostrik,
\textit{The Witt group of non-degenerate braided fusion categories},
\doi{10.1515/crelle.2012.014}{J.\ reine und angew.\ Math.\ \textbf{677} (2013) 135--177},
\arxiv{1009.2117}{arXiv:1009.2117 [math.QA]}.

\bibitem[DP]{DP}
B.~J.~Day, C.A.~Pastro,
\textit{Note on Frobenius monoidal functors},
New York J.\ Math.\ \textbf{14} (2008) 733--742, 
\arxiv{0801.4107}{arXiv:0801.4107 [math.CT]}.

\bibitem[EG]{EG}
D.~E.~Evans, T.~Gannon,
\textit{The exoticness and realisability of twisted Haagerup-Izumi modular data},
\doi{10.1007/s00220-011-1329-3}{Commun.\ Math.\ Phys.\ \textbf{307} (2011) 463--512},
\arxiv{1006.1326}{arXiv:1006.1326 [math.QA]}.

\bibitem[EGNO]{EGNO}
P.~Etingof, S.~Gelaki, D.~Nikshych, V.~Ostrik, 
\textit{Tensor categories},
\href{https://bookstore.ams.org/surv-205}{Mathematical Surveys and Monographs \textbf{205}, AMS, 2015}, \href{http://www-math.mit.edu/~etingof/egnobookfinal.pdf}{http://www-math.mit.edu/$\,\widetilde{~}$etingof/egnobookfinal.pdf}.

\bibitem[ENO1]{ENO1}
P.~Etingof, D.~Nikshych, V.~Ostrik, 
\textit{On Fusion Categories}, 
\href{https://www.jstor.org/stable/20159926}{Ann.\ Math.\ \textbf{162} (2005) 581--642},
\arxiv{math/0203060}{arXiv:math/0203060 [math.QA]}.

\bibitem[FFRS]{FFRS}
J.~Fr\"ohlich, J.~Fuchs, I.~Runkel, C.~Schweigert,
\textit{Correspondences of ribbon categories},
\doi{10.1016/j.aim.2005.04.007}{Adv. Math. \textbf{199} (2006) 192--329},
\arxiv{math/0309465}{arXiv:math/0309465 [math.CT]}.

\bibitem[FRS1]{FRS1}
J.~Fuchs, I.~Runkel, C.~Schweigert,
\textit{TFT construction of RCFT correlators. 1: Partition functions},
\doi{10.1016/S0550-3213(02)00744-7}{Nucl.\ Phys.\ B \textbf{646} (2002) 353--497},
\arxiv{hep-th/0204148}{arXiv:hep-th/0204148}.

\bibitem[FSV]{FSV}
J.~Fuchs, C.~Schweigert, A.~Valentino,
\textit{Bicategories for boundary conditions and for surface defects in 3-d TFT}, 
\doi{10.1007/s00220-013-1723-0}{Commun.\ Math.\ Phys.\ \textbf{321} (2013) 543--575}, 
\arxiv{1203.4568}{arXiv:1203.4568 [hep-th]}.

\bibitem[GJ]{GJ}
D.~Gaiotto, T.~Johnson-Freyd,
\textit{Condensations in higher categories},
\arxiv{1905.09566}{arXiv:1905.09566 [math.CT]}.

\bibitem[GM]{GM}
T.~Gannon, S.~Morrison,
\textit{Modular data for the extended Haagerup subfactor},
\doi{10.1007/s00220-017-3003-x}{Commun. Math. Phys. \textbf{356} (2017) 981--1015},
\arxiv{1606.07165}{arXiv:1606.07165 [math.QA]}.

\bibitem[Gr]{Gr}
D.~Green,
\textit{Classification of Rank 6 Modular Categories with Galois Group {(012)(345)}},
\arxiv{1908.07128}{arXiv:1908.07128 [math.QA]}.

\bibitem[HPT]{HPT}
A.~Henriques, D.~Penneys, J.~Tener,
\textit{Planar algebras in braided tensor categories},
\arxiv{1607.06041}{arXiv:1607.06041 [math.QA]}.

\bibitem[HRW]{HRW}
S.-m.\ Hong, E.~C.~Rowell, Z.~Wang,
\textit{On exotic modular tensor categories},
\doi{10.1142/S0219199708003162}{Commun.\ Contemp.\ Math.\ \textbf{10} (2008) 1049--1074},
\arxiv{0710.5761}{arXiv:0710.5761 [math.GT]}.

\bibitem[JMS]{JMS}
V.~F.~R.~Jones, S.~Morrison, N.~Snyder
\textit{The classification of subfactors of index at most 5}
\doi{10.1090/S0273-0979-2013-01442-3}{Bull.\ AMS \textbf{51} (2014) 277--327},
\arxiv{1304.6141}{arXiv:1304.6141 [math.OA]}.

\bibitem[KMRS]{KMRS}
V.~Koppen, V. Mulevi\v{c}ius, I.~Runkel, C.~Schweigert
\textit{Domain walls between 3d phases of Reshetikhin--Turaev TQFTs},
\arxiv{2105.04613}{arXiv:2105.04613 [hep-th]}.

\bibitem[Ko]{Ko}
L.~Kong,
\textit{Anyon condensation and tensor categories},
\doi{10.1016/j.nuclphysb.2014.07.003}{Nucl.\ Phys.\ B \textbf{886} (2014) 436--482},
\arxiv{1307.8244}{arXiv:1307.8244 [cond-mat.str-el]}.

\bibitem[KO]{KO}
A.~Kirillov, V.~Ostrik,
\textit{On a $q$-analog of the McKay correspondence and the ADE classification of $\widehat{\mathfrak{sl}_2}$ conformal field theories},
\doi{10.1006/aima.2002.2072}{Adv.\ Math.\ \textbf{171} (2002) 183--227},
\arxiv{math/0101219}{arXiv:math/0101219 [math.QA]}.

\bibitem[KSa]{KSa}
A.~Kapustin, N.~Saulina, 
\textit{Surface operators in 3d topological field theory and 2d rational conformal field theory}, in ``
Mathematical Foundations of Quantum Field Theory and Perturbative String Theory'', Proc. Symp.\ Pure Math.\ \textbf{83} (2001) 175--198,
\arxiv{1012.0911}{arXiv:1012.0911 [hep-th]}.

\bibitem[KZ1]{KZ1}
L.~Kong, H.~Zheng,
\textsl{The center functor is fully faithful}
\doi{10.1016/j.aim.2018.09.031}{Adv.\ Math.\ \textbf{339} (2018) 749--779},
\arxiv{1507.00503}{arXiv:1507.00503 [math.CT]}.

\bibitem[KZ2]{KZ2}
L.~Kong, H.~Zheng,
\textsl{Semisimple and separable algebras in multi-fusion categories},
\arxiv{1706.06904}{[1706.06904 [math.QA]}.

\bibitem[KZ3]{KZ3}
L.~Kong, H.~Zheng,
\textit{Drinfeld center of enriched monoidal categories},
\doi{10.1016/j.aim.2017.10.038}{Adv.\ Math.\ \textbf{323} (2018) 411--426},
\arxiv{1704.01447}{arXiv:1704.01447 [math.CT]}.

\bibitem[LP]{LP}
A.~D.~Lauda, H.~Pfeiffer,
\textit{State sum construction of two-dimensional open-closed Topological Quantum Field Theories},
\doi{10.1142/S0218216507005725}{J.\ Knot Theor.\ Ramif.\ \textbf{16} (2007) 1121--1163},
\arxiv{math/0602047}{arXiv:math/0602047 [math.QA]}.

\bibitem[MP]{MP}
S.~Morrison, D.~Penneys,
\textit{Monoidal categories enriched in braided monoidal categories},
\doi{10.1093/imrn/rnx217}{Int.\ Math.\ Res.\ Notices (2019) 3527--3579},
\arxiv{1701.00567}{arXiv:1701.00567 [math.CT]}.

\bibitem[MPP]{MPP}
S.~Morrison, D.~Penneys, J.~Plavnik,
\textit{Completion for braided enriched monoidal categories},
\arxiv{1809.09782}{arXiv:1809.09782 [math.CT]}.

\bibitem[MR1]{MR1}
V.~Mulevi\v{c}ius, I.~Runkel, 
\textit{Constructing modular categories from orbifold data}, 
\arxiv{2002.00663}{arXiv:2002.00663 [math.QA]}.

\bibitem[MR2]{MR2}
V.~Mulevi\v{c}ius, I.~Runkel, 
\textit{Fibonacci-type orbifold data in Ising modular categories}, 
\arxiv{2010.00932}{arXiv:2010.00932 [math.QA]}.

\bibitem[MS]{MS}
M.~McCurdy, R.~Street,
\textit{What separable Frobenius monoidal functors preserve?},
Cah.\ Topol.\ G\'eom.\ Diff\'er.\ Cat\'eg.\ \textbf{51}:1 (2010) 29-50,
\arxiv{0904.3449}{arXiv:0904.3449 [math.CT]}.

\bibitem[M\"u2]{Mug2}
M.~M\"uger, 
\textit{From Subfactors to Categories and Topology II. The quantum double of tensor categories and subfactors},
\doi{10.1016/S0022-4049(02)00248-7}{J.\ Pure Appl.\ Alg.\ \textbf{180} (2003) 159--219},
\arxiv{math/0111205}{arXiv:math/0111205 [math.CT]}.

\bibitem[Mul]{Mul}
V.~Mulevi\v{c}ius,
\textit{Defects and orbifolds in 3-dimensional TQFTs},
PhD thesis, available at \url{http://ediss2.sub.uni-hamburg.de/handle/ediss/9255}.

\bibitem[RSW]{RSW}
E.~C.~Rowell, R.~Stong, Z.~Wang,
\textit{On classification of modular tensor categories},
\doi{10.1007/s00220-009-0908-z}{Commun.\ Math.\ Phys.\ \textbf{292} (2009) 343--389},
\arxiv{0712.1377}{arXiv:0712.1377 [math.QA]}.

\bibitem[RT]{RT}
N.~Reshetikhin, V.~Turaev,
\textit{Invariants of 3-manifolds via link polynomials and quantum groups},
\doi{10.1007/BF01239527}{Invent.\ Math.\ \textbf{103} (1991) 547--597}.

\bibitem[Schm]{Schm}
G.~Schaumann, 
\textit{Traces on module categories over Fusion categories},
\doi{10.1016/j.jalgebra.2013.01.013}{J.\ Algebra \textbf{379} (2013) 382--423},
\arxiv{1206.5716}{arXiv:1206.5716 [math.QA]}.

\bibitem[Sz1]{Sz1}
K.~Szlach\'anyi,
\textit{Finite quantum groupoids and inclusions of finite type},
\doi{10.1090/fic/030}{Fields Inst.\ Commun.\ \textbf{30} (2001) 393--407},
\arxiv{math/0011036}{arXiv:math/0011036 [math.QA]}.

\bibitem[Sz2]{Sz2}
K.~Szlach\'anyi,
\textit{Adjointable Monoidal Functors and Quantum Groupoids},
\doi{10.1201/9780429187629-18}{Lect.\ Notes in Pure and Appl.\ Math.\ \textbf{239} (2005) 291--307}, 
\arxiv{math/0301253}{arXiv:math/0301253 [math.QA]}.

\bibitem[Tu1]{Tu1}
V.~Turaev,
\textit{Modular categories and 3-manifold invariants},
\doi{10.1142/S0217979292000876}{Int.\ J.\ Mod.\ Phys.\ B \textbf{6}:11-12 (1992) 1807--1824}.

\bibitem[Tu2]{Tu2}
V.~Turaev, 
\textit{Quantum Invariants of Knots and 3-Manifolds},  
de Gruyter, New York, 1994.

\bibitem[TV]{TV}
V.~Turaev, A.~Virelizier, 
\textit{Monoidal Categories and Topological Field Theories}, 
\doi{10.1007/978-3-319-49834-8}{Progress in Mathematics \textbf{322}, Birkh\"auser, 2017}.

\end{thebibliography}
\end{document}